\documentclass{article}
\usepackage[utf8]{inputenc}
\usepackage[T1]{fontenc}
\usepackage{textcomp}
\usepackage{babel}
\usepackage[authoryear]{natbib}
\usepackage{amsthm}
\usepackage{amsfonts}
\usepackage{amsmath}
\usepackage{amssymb}
\usepackage{tabularx}
\usepackage{graphicx}
\usepackage[ruled, noend, noline]{algorithm2e}
\usepackage{comment}
\usepackage{lscape} 
\usepackage{titling}
\usepackage{blindtext}
\usepackage{lipsum}
\usepackage{microtype} 
\usepackage{xspace}
\usepackage[inactive]{srcltx}		 	
\usepackage{relsize}
\usepackage{booktabs} 
\usepackage{mathrsfs} 
\usepackage{caption}
\usepackage{placeins}
\usepackage[font=small,skip=0pt]{caption}
\usepackage{multirow}
\usepackage{xspace}
\usepackage[scale=.95,type1]{cabin}
\usepackage[bigdelims,cmintegrals,vvarbb]{newtxmath}
\usepackage[zerostyle=c,scaled=.94]{newtxtt}
\usepackage{subfigure}
\usepackage{xcolor}
\usepackage{color}
\usepackage{hyperref}
\usepackage{acronym}
\usepackage{fullpage}
\usepackage[normalem]{ulem} 
\usepackage{makeidx}
\usepackage{wrapfig}
\usepackage{fancyhdr}
\usepackage{algpseudocode}
\usepackage{pifont}
\usepackage{setspace,calrsfs}
\usepackage{setspace}
\newtheorem{theorem}{Theorem}
\theoremstyle{definition}
\newtheorem{definition}{Definition}[section]
\newtheorem{assumption}{Assumption}[section]
\newtheorem{remark}{Remark}[section]

\newtheorem{lemma}{Lemma}[section]
\newtheorem{example}{Example}[section]
\makeatletter
\patchcmd{\@eqnnum}{\normalcolor}{\color{myblue}}{\typeout{eqnnum patch: OK!}}{\typeout{eqnnum patch: Oh, dear!}}
\newcommand{\inte}[1]{{\kern0pt#1}^{\mathrm{o}}}
\newcommand\myatop[2]{\genfrac{}{}{0pt}{}{#1\hfill}{#2\hfill}}
\definecolor{myblue}{RGB}{0, 128, 128}
\hypersetup{colorlinks=true,linkcolor=myblue,urlcolor=myblue,citecolor=myblue,}

\def\tagform@#1{\maketag@@@{\color{myblue}(#1)}}

\title{\LARGE \bf
Iterated Function Systems: \\ A Comprehensive Survey}
\author{Ramen Ghosh \thanks{\href{mailto:ramen.ghosh@ucdconnect.ie}{School of Electronic and Electrical Engineering, University College Dublin, Ireland, ramen.ghosh@ucdconnect.ie}}  \and Jakub Mare\v{c}ek \thanks{\href{mailto:jakub.marecek@fel.cvut.cz}{Faculty of Electrical Engineering, Czech Technical University of Prague, The Czech Republic, jakub.marecek@fel.cvut.cz}}
}

\begin{document}
\maketitle
\begin{abstract}
We provide an overview of \emph{iterated function systems} (IFS), where randomly chosen state-to-state maps are iteratively applied to a state.
We aim to summarize the state of art and, where possible, identify fundamental challenges and opportunities for further research. 
\end{abstract}

\section{Introduction}\label{sec:intro}
The fundamental mathematical tool used in this study is \emph{discrete-time Markov chains}, which are generated by random repetition of functions and exhibit long-term statistical behaviour.
Allow us to encourage readers through an example from \cite[Section 2.1, Chapter 1]{Arnab2016}; similar examples could be found in \cite[Chapter 3]{Barnsley1993}. Assume a particle is located at the origin, i.e., at $(0,0)$ of a two-dimensional Euclidean plane. Consider a die with four sides denoted by the integers $1,2,3,4$. The particle  moves on the plane according to rules as shown in Table~\ref{table:1}:
 \begin{table}[h]
\begin{center}
\begin{tabularx}
{0.8\textwidth} { 
  | >{\raggedright\arraybackslash}X 
  | >{\centering\arraybackslash}X 
  | >{\raggedleft\arraybackslash}X | }
\hline
\hline
A number appears in the dice & Where the particle moves \\
\hline
\hline
$1$ & $(0.8x+0.1, 0.8y+0.04)$\\
\hline
$2$ & $(0.6x+0.19, 0.6y+0.5)$\\
\hline
$3$ & $(0.466(x-y)+0.266, 0.466(x+y)+0.067)$\\
\hline
$4$ & $(0.466(x+y)+0.456, 0.466(y-x)+0.434)$\\
\hline
\hline
\end{tabularx}
\caption{Rules describing how the particle moves according to the outcome in the dice}
\label{table:1}
\end{center}
\end{table}
Initially, the particle is at $(0,0)$; now, imagine $2$ shows up after throwing the die. From the rules described  in the table, the particle will move to $(0.6x+0.19, 0.6y+0.5)$ where $(x,y)$ is its current coordinate. For us, $(x,y)=(0,0)=X_0$ the initial position of the particle, and since $2$ has appeared, the next position or co-ordinate of the particle is
\begin{align*}
(0.6\times 0+0.19, 0.6\times 0+0.5)=(0.19,0.6)=X_1.
\end{align*}
We mark the point $(0.25,0.4)$ with a pen. We throw the die again, and depending on the number, and the particle moves to a different point from  $(0.25,0.4)$. If we repeat these steps for, say, $20,000$ times (Of course, that requires the help of a computer) and mark those points on the plane, we get a picture similar to Figure~\ref{fig:maple}.
\begin{figure}[bt]
\centering
\includegraphics[width=0.61\columnwidth]{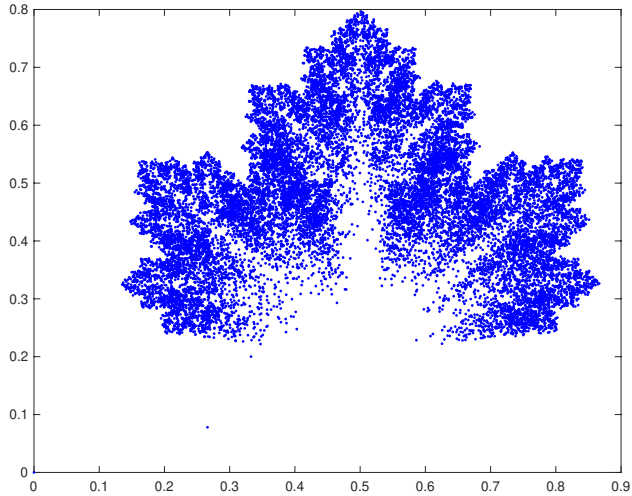}
\caption{Random trajectory of the particle described in Table \ref{table:1}}
\label{fig:maple}
\end{figure}
These points visited by the particle are insufficient to cover all points in the leaf and produce a fully formed naturally shaped leaf devoid of holes. After several iterations, a two-dimensional maple leaf is the particle's sole visited or so-called accessible position. One may demonstrate that with probability one, the sequence of points $X_0, X_1, X_2, \dots$, referred to as a sample path or trajectory, would ultimately pass through any portion of the \emph{Maple leaf}. The exciting aspect of the game is that even though the particle's travel was random in each step, the movement's long-term behaviour took on the form of a Maple leaf. The outcome will be the same regardless of how many individuals play the game concurrently or how many times. Particular emphasis should be placed on the terms \emph{``with probability one''}, which will be explained subsequently. The trajectory will approach the same point if the dice consistently display the same face. However, for a fair dice, this occurrence has zero chance. The act of continually tossing the dice is referred to as \emph{``iteration''}, and the rule dictating how the particle should advance to the next step from the present position is known as a \emph{``function''}. It could be easily seen that the four rules described in the table above can be seen as the following functions $w_1, w_2, w_3, w_4$ acting on $\begin{bmatrix}x\\y\end{bmatrix}\in\mathbb R^2$ as follows:
\begin{align*}
w_1\left(\begin{bmatrix}x\\y\end{bmatrix}\right)&=\begin{bmatrix}0.8&0\\0&0.8\end{bmatrix}\begin{bmatrix}x\\y\end{bmatrix}+\begin{bmatrix}0.1\\0.04\end{bmatrix}\\    
w_2\left(\begin{bmatrix}x\\y\end{bmatrix}\right)&=\begin{bmatrix}0.6&0\\0&0.6\end{bmatrix}\begin{bmatrix}x\\y\end{bmatrix}+\begin{bmatrix}0.19\\0.5\end{bmatrix}\\
w_3\left(\begin{bmatrix}x\\y\end{bmatrix}\right)&=\begin{bmatrix}0.466&-0.466\\0.466&0.466\end{bmatrix}\begin{bmatrix}x\\y\end{bmatrix}+\begin{bmatrix}0.266\\0.067\end{bmatrix}\\
w_4\left(\begin{bmatrix}x\\y\end{bmatrix}\right)&=\begin{bmatrix}0.466&0.466\\-0.466&0.466\end{bmatrix}\begin{bmatrix}x\\y\end{bmatrix}+\begin{bmatrix}0.456\\0.434\end{bmatrix}\\
\end{align*}
The chaotic game seen above illustrates the \emph{``random repetition of functions''} (where the randomness appeared due to the throwing of the die).

\subsection{Some Motivating Applications of IFS}
\begin{example}\emph{Non-linear time series in Statistics:}\label{ex:NLT}
A first-order autoregression model is defined in non-linear time series analysis as \begin{align*}
X_k=\phi X_{k-1}+Z_k,\quad k\in \mathbb{Z},
\end{align*} is critical in time series analysis because it represents a strong (i.i.d.) white noise sequence. Establishing the requirements for a stationary solution to the first-order autoregressive equation is simple. A nearly comparable process is provided by 
\begin{align}\label{eq:nlinr-tme-srs}
Y_{k}=A_kY_{k-1}+B_k, \quad k\in \mathbb{Z},
\end{align}
where $(A_k, B_k)$ is a series of independently generated random vectors in $\mathbb R^2$. These models are critical in non-linear time series modelling. The presence of strict-sense and weak-sense stable solutions to the equation \eqref{eq:nlinr-tme-srs} is a natural issue. Under certain weak circumstances, \cite[Theorem 4.1]{Douc2014} established the existence of strict-sense stationery. The theorem has been further generalized to functional autoregressive process using IFS techniques \cite[Theorem 4.40,  4.41]{Douc2014} and based on these results, a weak-sense stationary solution to \eqref{eq:nlinr-tme-srs} has been shown in \cite[Theorem 4.47]{Douc2014} under the assumptions \cite[Assumption A4.43-A4.45]{Douc2014}.
\end{example}

\begin{example}\emph{IFS in finance:}\label{ex:FIN} 
Uniquely ergodic IFS with constant probability are valuable alternatives to the standard binomial models (see \cite{Van2006}) of stock price development that are used to estimate one-period option pricing \cite[Section 3.8]{Anatoly2013}. A generalised binomial model \cite{Cox1979} is found using IFS with state-dependent probability for call option prices, bond valuation, stock price evolution, and interest rates. See  \cite[Section 4.6.1, 4.6.4]{Anatoly2013} and \citep{Bahsoun2005(2)}.
\end{example}

\begin{example}\emph{Stability of congestion management, e.g., in Transmission Control Protocol:}\label{ex:CM}
The distribution of limited resources among several agents in different scientific fields and applications is one of the most challenging problems. For example, the bandwidth available for communication on computer networks is limited. Each connection can simultaneously support several data flows, each seeking to optimize its portion of the available bandwidth. A circumstance like that occurs when many electric cars are charged concurrently at the same charging station and compete for the same electricity or charge rate. A scalable and robust solution to this challenge is required.
Additionally, a decentralized technique is appropriate when there is limited interaction between users and maintaining anonymity is critical. The method \emph{additive increase and multiplicative decrease} (AIMD) has been quite effective in resolving the issue \citep{Chiu1989}. In actuality, a more complex and challenging scenario occurs. To address such scenarios, the stochastic linear AIMD method and the stochastic additive increase and non-linear decrease (AINLD) algorithm developed in \cite[Chapter 6, 7, 8]{Corless2016} have been extensively investigated utilizing IFS and their ergodic features. For example, in \citep[Chapter 7]{Corless2016} the AIMD with state-dependent transition probabilities has been modelled as
\begin{align}\label{eq:aimd-dyna}
w(k+1)=A(k)w(k) \quad k=0,1,2,\dots,
\end{align}
where $w(k)\in S_n$ is vector-valued random variable and defined as
\begin{align*}
S_n=\{w(k)=(w_1(k),\dots, w_n(k))\in \mathbb R^n:\sum_{i=1}^{n}w_i(k)=c\},
\end{align*}
where for each $i$, $w_i(k)$ is the amount of share consumed by the $i^{th}$ user and $k=0,1,2,\dots$ represents discrete time instants, $c$ represent the overall capacity of the resource accessible to the whole system, and $c=1$ is assumed for the purpose of easy computation, AIMD matrix $A(k)$ is given by
\begin{align*}
A(k)=\begin{bmatrix}\beta_1(k)&\cdots&0\\
\vdots&\ddots&\vdots\\
0&\cdots&\beta_n(k)\end{bmatrix}+\frac{1}{\sum_{j=1}^{n}\alpha_j}\begin{bmatrix}\alpha_1(k)\\\vdots\\\alpha_n(k)\end{bmatrix}\begin{bmatrix}\left(1-\beta_1(k)\right)&\cdots&\left(1-\beta_n(k)\right)\end{bmatrix}
\end{align*}
where $\alpha_i(k)>0$ is the rate of growth of the agent $i$ after the $k^{th}$ capacity event (a time when the resource reaches its capacity constraint, i.e., $\sum_{i=1}^{n} w_i(k)=c=1$) and $\beta_i(k)$ denote multiplicative decrease quantity at the $k^{th}$ capacity event depending on the response of the $i^{th}$ user at the capacity event, $0\le \beta_i(k)\le 1$.  

It can be shown that the AIMD matrices are non-negative, column stochastic. The set of matrices arising through \eqref{eq:aimd-dyna} is denoted by finite set of indices $\mathcal S$ and $\mathcal{A}=\{A_{\sigma}: \sigma\in \mathcal{S}\}$. The technique is stochastic because the index $\sigma(k)$ is created randomly for each event $k$; that is, $\sigma(k)$ is a random variable. Thus,
\begin{align*}
A(k)= A_{\Sigma(k)}.
\end{align*}
The equation  \eqref{eq:aimd-dyna} may be used to represent the dynamical system of interest in the networks of AIMD flows. We refer to a state-dependent AIMD model in the paper \cite[Chapter 7]{Corless2016}. 

Consider $\{p_j: j\in \mathcal S\}$ as a collection of probability functions from the simplex $S_n$ into the closed interval $[0,1]$ that fulfil \begin{align*}
\sum\limits_{j\in \mathcal S} p_{j}\left(w\right)=1 \quad \text{ for all } w\in S_n.
\end{align*}
Here, $p_j(w)$ denotes the probability that the matrix $A_j$ occurs when the state or the share vector is $w$, i.e.,
\begin{align}\label{eq:ocuurence-index}
\mathbb{P}\left(\sigma(k)=j\mid x(k)=x\right)=p_{j}(x) \quad k=0,1,2,\dots  \end{align}
or,
\begin{align}\label{eq:ocuurence-mat}
\mathbb{P}\left(A(k)=A_j\mid w(k)=w\right)=p_{j}(w) \quad k=0,1,2,\dots    
\end{align}
Take note that \eqref{eq:ocuurence-index} and \eqref{eq:ocuurence-mat}  create a stochastic AIMD algorithm, and they are an IFS with the state (or place) dependent probabilities, and more precisely, a Markov chain on $S_n$, whose state transition probabilities are provided by 
\begin{align}\label{eq:transition-prob-aimd}
\mathbf{P}(w, \mathcal A)= \mathbb{P}\left(w(k+1)\in \mathcal A\mid w(k)=w\right)=\sum\limits_{j: A_{j}w\in \mathcal A} p_{j}(w) \quad \text{ for all } w\in S_n, \text{ and for any event } \mathcal A.   
\end{align}
Having created AIMD's mathematical framework, we may question how the AIMD algorithm distributes resources efficiently among the participants. Can the algorithm ensure equal shares for all network agents? Several other relevant research questions are highlighted in \cite[Chapter 1, Section 1.4]{Corless2016}. To know more about this direction please see \cite{Dumas2002, Fabian2006, Bob2007, Wirth2006, King2006, Shorten2007, Corless2012, Martin2012,Fioravanti2017, Fioravanti2019, Bob2016, Corless2016, Jakub2017, Wirth2019, Lesniak2021, Ramen2021}.
\end{example}

\begin{example}\emph{Closed-loop models within sharing economy:}\label{ex:CLM} 
Imagine if $\mathcal{S}_1,\dots, \mathcal{S}_N$ an $N$ number of agent-based system with  with a controller $\mathcal{C}$, filter $\mathcal{F}$, in a closed-loop feedback as in Figure \ref{fig:CL-system}. At time $k$, $\mathcal{C}$ generates a signal  $\pi(k)$. As a result, agents adjust the way they utilize the resource. Due to unpredictability in the agent's reaction to the control signal, randomized control signal, or random perturbations of the control signal, for all $i$, the $i^{th}$ agent's state $x_i(k)$ at time $k$ is treated as a random variable, and we note that $\mathcal{C}$ has no access to  $x_i(k)$ or $y(k)=\sum\limits_{i=1}^{N} x_i(k)$ but only able to get access to the error signal $e(k)$, which indicates a difference between the value $\hat{y}(k)$ that is produced by the filter and the desired output value  $y(k)$. Smart grid and smart city applications may be shown with this simple setup. In this closed-loop feedback system, a stochastic difference equation governs the state of each agent:
\begin{align}\label{eq:agents-dynamics}
x_i(k+1)=w_{\sigma_k}(x_i(k)),
\end{align}
where $w_1, w_2,\dots, w_N$ are taken as general response functions of $N$ agents in the network, $ \sigma_1,\sigma_2,\dots$ are $\{1,2,\dots, N\}$ valued independent identically distributed (i.i.d) discrete random variables. This representation is precisely an IFS system. 

The controller and filter take all conceivable initial state distributions into account to guarantee that the system's long-term behaviour is acceptable for all agents involved. For each stable linear controller and filter, we seek conditions where the feedback loop controls total resource consumption and converges to a unique invariant measure. The above concept and strategy have been investigated in a variety of other engineering fields, including smart-city and smart-grid control \citep{Fioravanti2017, Fioravanti2019}, congestion control \citep{Shorten2005}, resource allocation \citep{Jakub2017}, load aggregation and operation of virtual power plants \citep{Marecek2021}, social sensing platforms \citep{Ramen2021}, and model predictive control \citep{Slava2021}.
\begin{figure}[t!]
\centering
\includegraphics[width=0.45\columnwidth]{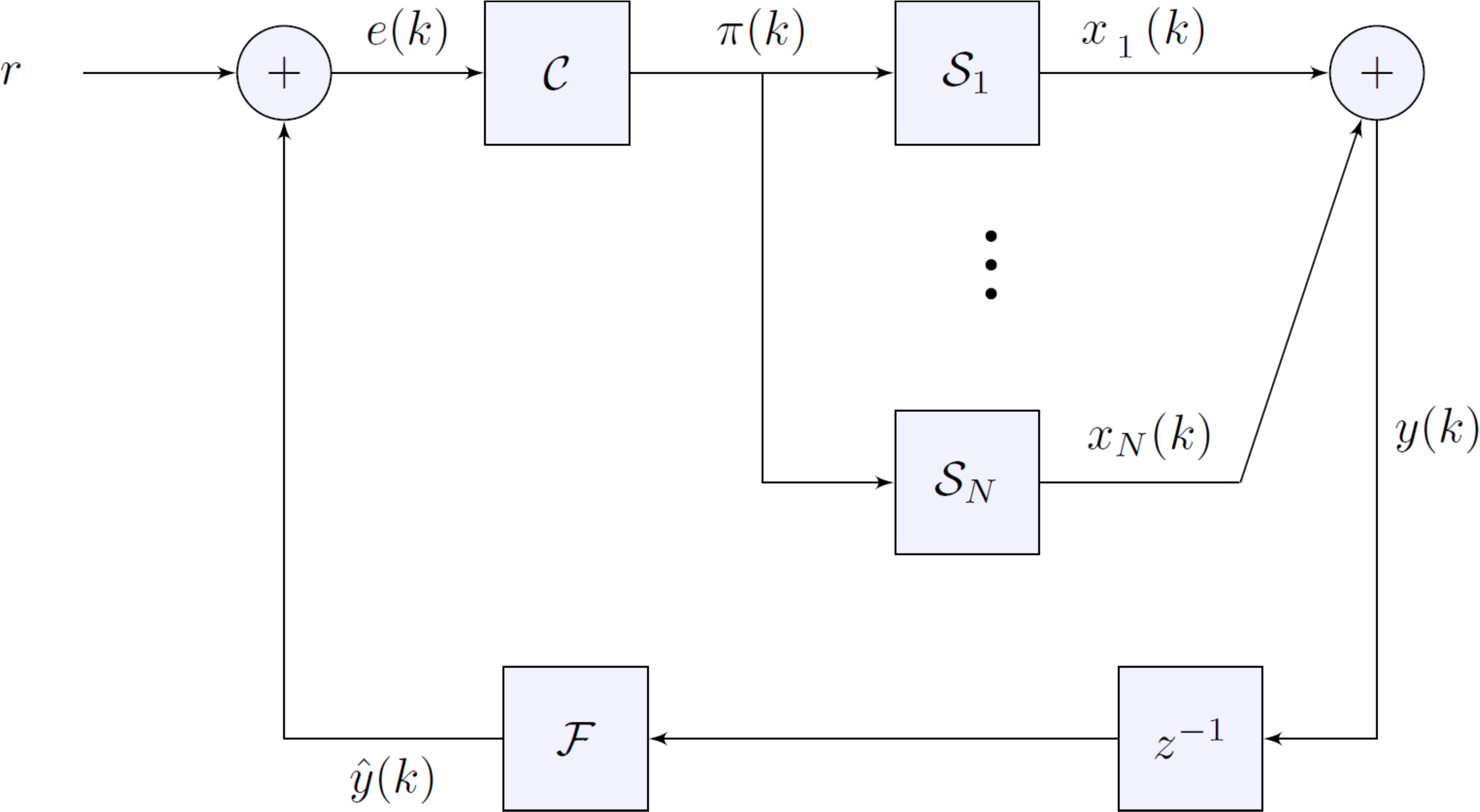}\\[6mm]
\caption{An illustration of basic closed-loop model, following \cite{Fioravanti2017, Fioravanti2019}}
\label{fig:CL-system}
\end{figure}
\end{example}

\begin{example}\emph{Opinion dynamics:} \label{ex:OPD}
There has long been a focus of research at the junction of economics and dynamical systems; for example, Acemoglu's pioneering work utilizing IFS \cite{Acemoglu2010, Acemouglu2013}, and the references therein. \emph{The Friedkin-Johnsen model} \cite{Friedkin2011} accurately depicts opinion dynamics characterized by heterogeneity. \cite{Karl2019} shown that the \emph{Friedkin-Johnsen model} is a specific example of their expanded \emph{Rescorla-Wagner model} \cite{Rescorla1972}. Surprisingly, this extended model with random time-varying topology is exactly a recurrent IFS, and the suggested model's internal states exhibit both convergence and ergodicity.
\end{example}

\begin{example}\emph{Stochastic dynamic decision model}:\label{ex:SDDM} 
Following \cite[Chapter 1]{Hinderer1970}, let a set $\mathcal{X}\ne \phi$ be called the state space, and let $\mathcal A\ne \phi $ be another set called action space. Let $W=(w_k)_{k\in \mathbb N}$ is a sequence of maps $w_k$ from certain sets $H_k\subseteq \bar{H}_k=\mathcal{X}$ or $\mathcal{X}\times \mathcal{A}\times\cdots\times \mathcal{A}\times \mathcal{X}$ ($2k-1$ factors), according as $k=1$ or $\ge 2$, into the class of all non-empty subsets of $\mathcal{A}$ such that
\begin{align*}
H_1&=\mathcal{X},
H_{k+1}=\{(h,a,x): h\in H_k, a\in w_k(h), x\in \mathcal{X} \}, k\in \mathbb N.
\end{align*}
$\bar{H}_k$ is called the set of histories at time $k$, $H_k$ the set of all admissible histories at time $k$, and $w_k(h)$ the set of all admissible actions at time $k$ under admissible history $h$. Denoting $S_k=\{(h,a): h\in H_k, a\in w_k(h)\}$ we can write $H_{k+1}=S_k\times \mathcal{X}, k\in \mathbb N$. Let $p_0$ be a probability distribution on $\mathcal{X}$ called the initial distribution and $p_k$ is a transition probability function from $S_k$ to $\mathcal{X}$, i.e. for any $(h, a)\in S_k$, $p_k(h, a, \cdot)$ is a probability distribution on $\mathcal{X}$, called transition law from time $k$ to $k+1$. And, finally, let $r_k$ is an extended real-valued function on $S_k$ called reward during the time interval $(k,k+1)$. A collection $\{\mathcal{X}, \mathcal{A}, W, (p_k)_{k\in \mathbb N}, (r_k)_{k\in \mathbb N}\}$ with the above properties is called a stochastic dynamic decision model. Choosing a policy and an optimality criterion is necessary to define a dynamic optimisation problem under the suitable assumption on $r_k$. A deterministic admissible policy is  $f=(f_k)_{k\in \mathbb N}$ of maps $f_k: \mathcal{X}^k\to \mathcal{A}$ such that
\begin{align*}
f_k(y)\in w_k\left(h_k^f(x)\right), x\in \mathcal{X}^k, k\in \mathbb{N},
\end{align*}
where
\begin{align*}
h_k^f(x):=\left\{\begin{array}{ll}
(x_1), &\text{ if } k=1\\
(x_1, f_1(x_1), x_2, f_2(x_1, x_2),\dots, f_{k-1}(x_1, \dots, x_{k-1}), x_k), & \text { if } k \geq 2.
\end{array}\right.   
\end{align*}
denotes the history at time $k$ obtained when the sequence $x=(x_1, \dots, x_k)$ of states occurred and actions were chosen according to $f$. The application of a policy $f$ generates a stochastic process, the decision process determined by $f$, is as follows: Starting at some $k=1$ at some point $x_1$, selected from $\mathcal{X}$ according to the initial probability distribution $p_0$, one takes action $f_1(x_1)$ whereupon one gets the reward $r_1(x_1, f_1(x_1))$, and the system moves to some $x_2\in \mathcal{X}$ selected according to  $p_1(x_1, f_1(x_1), \cdot)$. Then the action  $ f_2(x_1, x_2)$ is taken, giving the  $r_2(x_1, f_1(x_1), x_2, f_2(x_2))$ and the system moves to some $x_3\in \mathcal{X}$ selected according to  $p_2(x_1, f_1(x_1), x_2, f_2(x_1, x_2), \cdot)$ and so on. Thus for a given policy, the sequence of states $\{x_k\}_{k\in \mathbb N}$ then appears equivalent to the IFS. 
\end{example}

\begin{example}\emph{In DNA replication model:}\label{ex:DNA}
The DNA replication kinetic equation is shown in \cite[Equation 2]{Gaspard2017}. New DNA strands are created by molecular machinery utilizing templates from previous generations. To correctly solve replication kinetic equations, one may use iterative function systems that run along the template sequence and provide copy sequence statistics and kinetic and thermodynamic parameters. DNA polymerase local velocity distributions with fractal and continuous sequence heterogeneity are studied using this approach, as well as the transition between linear and sub-linear copy growth rates.
\end{example}

\subsection{Organization of the Survey and Past Surveys on IFS}
\label{subsec:ch2-past-surveys}

Our survey is organized as follows: 
Sections \ref{sec:IFS-cpt-space}--\ref{sec:CLT-IFS-invp} survey key research areas within the current IFS literature.
Section \ref{sec:geometry} of the survey focuses on structural results, which can be seen as the geometry of the unique invariant measure.
Finally, Section \ref{sec:further} highlights several possible research directions.
In Figure~\ref{fig:des-1}, \ref{fig:des-2} and \ref{fig:des-3}, we have broken down the key areas to ease navigating the survey. 
In the appendices, we provide pointers to key papers witihin applications of IFS. 

We should also like to point to the excellent and well-known surveys of \cite{Diaconis1999}, and \cite{Iosifescu2009},
and the lesser-known surveys and reviews of \cite{Kaijser1981, Stenflo2012(s), Stenflo1998, Letac1986, Chamayou1991, Niclas2005(1), Bhattacharya2003, Athreya2003(4), Athreya2016, Fuh2004, Majumdar2009,kunze2011fractal, Denker2016}. For a survey on a random iteration of quadratic functions, see \cite{Athreya2001}. 
The sheer fact that our survey appears more than two decades after \cite{Diaconis1999} makes it possible to cover more of the recent developments and thus be more comprehensive. 

\section{Introduction and Historical Background of IFS}\label{sec:ch2-ifswp-intro}
This part will concentrate on IFS concepts and conclusions critical for understanding how the theory might be used for various practical uses. The objective of this part is twofold: to present a brief overview of IFS theory, concentrating on a few instances relevant to image generation, and to summarise numerous additional conclusions of general interest relevant to this thesis to keep this paper self-contained as feasible. After briefly discussing some essential facts about metric spaces, contractive transformations, and affine transformations, the rest of this section will explore the definitions and a few fundamental features of IFS. We refer \cite{Barnsley1993, Edgar2007, Peitgen2006} for a more in-depth discussion.

The term IFS has become well-known and popular in pure and applied science, particularly mathematics, probability, quantum physics, mathematical biology, computer science, electrical engineering, and economics, over the last seventy years.  Numerous IFS theories derive from the \emph{theory of random systems with complete connections}. According to \cite{Iosifescu1990}'s book, the first explicit definition of the idea of \emph{dependency with complete connection} was created by \cite{Mihoc1935}, who used a discrete state space. The idea of random systems with complete connection was initially established in 1963 by \cite{Iosifescu1963}. A few applications of \emph{random systems with complete connection} have been recognised in the extensively investigated field of theory of \emph{learning models} in \cite{Norman1972, Isaac1962, Burton1993}. In this topic, \cite{Karlin1953} wrote a seminal study in which he studied a learning model that can be thought of as an \emph{IFS with place-dependent probability} on an interval of the real line. The proof of \cite[Theorem 36]{Karlin1953} had some gap which later led to some exciting counterexamples in the IFS with a finite number of maps with place dependent probabilities, see for example \cite{Stenflo2001(2), Keane1972, Barnsley1988(3), Kaijser1994, Kaijser1981} for further comments on Karlin's work \cite{Karlin1953}.

\cite{Barnsley1985} first named the word IFS and has garnered much attention since then. Their study sparked widespread interest in field applications involving the representation of real-world images utilising IFS's two-dimensional transformation. It was well-known that IFS could accurately depict a complicated object with just a few carefully selected parameters and that this system could be coupled to make more complex pictures.

\subsection{Mathematical Preliminaries on IFS}\label{subsec:ch2-nota-dfn-mp}
Before we go any further, let us review some fundamental concepts, notations, and nomenclature. Throughout this article, we assume $\mathcal X\ne \emptyset$, and $\rho: \mathcal{X}\times \mathcal{X}\to [0,\infty)$ is a metric on it, we shall denote $\left(\mathcal X,\rho\right)$ the metric space consisting of the set $\mathcal{X}$ endowed with the metric $\rho$. We will often be assumed to be \emph{complete and separable, i.e., a Polish space}. The requirement that the metric space $\left(\mathcal X,\rho\right)$ is to be complete will be apparent upon the statement of the contraction mapping theorem and specification of its role in developing the theory of IFS.  

A $\sigma$-algebra of $\mathcal X$ is a collection of subsets of $\mathcal{X}$ that is closed by complements, countable unions, and countable intersections operation includes $\mathcal{X}$. The smallest $\sigma$-algebra, which contains all the open sets of $\mathcal X$,  is called the Borel $\ sigma$ algebra. $\mathcal{B}\left(\mathcal X\right)$ is the Borel sigma-algebra on  $\mathcal X$, and $\mathcal{A}$ is always an event in $\mathcal B\left(\mathcal X\right)$. $C\left(\mathcal X,\mathbb R\right)$ is the Banach space that corresponds to all real-valued continuous functions operating on  $\mathcal X\to \mathbb R$ endowed with the supremum norm $\left\|\cdot\right\|_{\infty}$.  ${C}_{\text{b}}\left(\mathcal X,\mathbb R\right)$ is structurally identical to $C(\mathcal X,\mathbb R)$, except its functions are continuous and bounded. $\mathcal{M}\left(\mathcal X\right)$ is the real vector space that contains all signed finite Borel measures on $ \mathcal X$ that include $\mathcal{M}_{\text{f}}\left(\mathcal X\right)$, which is the space that contains all positive measures. $\mathcal{M}_{\text{p}}\left(\mathcal X\right)$ represents space of all probability measures on  $\mathcal X$ included in $\mathcal{M}_{\text{f}}\left(\mathcal X\right)$.  $\mathcal B_{\text{bm}}\left(\mathcal X, \mathbb R\right)$ denotes the space of all bounded, measurable, real-valued functions on $\mathcal X$ and $w: \mathcal X\to \mathbb R$ is always continuous, Lipschitz, and measurable and in addition to possibly having other properties which will be mentioned accordingly. 

An appropriate probability triplet $(\mathcal{X}, \mathcal{E}, \mathbb P)$, where $\mathbb P$ is the probability measure, the set $\mathcal{X}\ne \emptyset$, and  $\mathcal{E}$ is a collection of events, is considered throughout the article. $\mathbb P_{x}$  denote the probability measure in a context where the initial condition is taken as $x$, more precisely, when an event $\mathcal{A}$ has probability $\mathbb P_x(\mathcal{A})=1$, we say that the event occurs almost surely $\mathbb P_{x}$ or in short a.s $\mathbb P_x$. The expectation operator is always denoted by $\mathbb E$ and $\mathbb E_{\mu}$ when taken with respect to a probability measure $\mu$.

$[m]:= \{1,2,\dots, m\}$ is a finite set of natural numbers from $1$ to $m<\infty$, $\mathbb R,\mathbb N,\mathbb Z$ denote the set of real, natural numbers and integers respectively.  $\mathbb{R}^n$ be the $n$ fold Cartesian product of $\mathbb{R}$, by  $\mathbb R^{+}$ we mean the set $\{x\in\mathbb R: x>0\}$. Throughout the article a point $x \in \mathbb{R}^n$ will be represented by a $n\times 1$ column vector $\left(x_1,x_2,\dots,x_n\right)^{\top}$, where $\top$ is usual transpose notation. A familiar example of a complete metric space is given by $\left(\mathbb{R}^n,\rho\right)$, where $\rho$ is the Euclidean metric on $\mathbb{R}^n$ defined by for any two points $x=(x_1,x_2,\dots,x_n)^{\top}, y=(y_1,y_2,\dots,y_n)^{\top}\in \mathbb{R}^n$:
\begin{align}\label{eq:ch2-mp-euclid-dist}
\rho\left(x,y\right)=\left\{\sum\limits_{i=1}^{n} \left(y_i-x_i\right)^2\right\}^{\frac{1}{2}}.
\end{align}
Given a metric space $\left(\mathcal{X},\rho\right)$,  $\left(\mathcal{H}\left(\mathcal{X}\right),\rho_{\mathrm{h}}\right)$ be another metric space formed by the collection of nonempty, compact subsets of $\mathcal{X}$ which is denoted by $\mathcal{H}\left(\mathcal{X}\right)$ and $\rho_{\mathrm{h}}$ be the Hausdorff metric to be defined momentarily. First define the distance from $x$ to $Y$ as:
\begin{align}\label{eq:ch2-mp-dist-pt-set}
\rho\left(x, Y\right)=\min_{y\in Y}\{\rho(x,y)\}\; \text{ for } X, Y \in \mathcal{H}\left(\mathcal{X}\right).
\end{align}
Since $y\mapsto \rho(x,y)$ is a continuous map and $Y\ne \emptyset$ and compact, the minimum is attained. Next, define $\rho\left(X, Y\right)$, the distance between two compact sets in $\mathcal{H}\left(\mathcal{X}\right)$:
\begin{align}\label{eq:ch2-mp-dist-set-set}
\rho\left(X, Y\right)=\max_{x\in X}\{\rho(x,Y)\}\; \text{ for } X, Y \in \mathcal{H}\left(\mathcal{X}\right).
\end{align}
For the same reason as above, the maximum is also attained. One should notice that, in general, $\rho\left(X, Y\right)$ and $\rho\left(Y, X\right)$ may not be same, i.e $\rho\left(\cdot, \cdot\right)$ is not symmetric on $\mathcal{H}\left(\mathcal{X}\right)\times \mathcal{H}\left(\mathcal{X}\right)$, and hence is not a metric \cite[Section 2.6, Exercise 6.7]{Barnsley1993}.
Now we are in a position to define the Hausdorff metric $\rho_{\mathrm{h}}$:
\begin{align}\label{eq:ch2-mp-dist-fract-set-set}
\rho_{\mathrm{h}}\left(X, Y\right)=\max\{\rho(X,Y), \rho(Y,X)\}\;\text{ for } X, Y \in \mathcal{H}\left(\mathcal{X}\right).
\end{align}
We refer \cite[Section 2.6, Exercise 6.15]{Barnsley1993} for the fact that $\left(\mathcal{H}\left(\mathcal{X}\right),\rho_{\mathrm{h}}\right)$ is indeed a metric space.

Given a non-empty set $\mathcal{X}$, by a \emph{self-transformation or self map} on it, we mean a map $w$ with domain and range is $\mathcal{X}$, i.e., $w:\mathcal{X}\to \mathcal{X}$.  A very important class of transformations of one metric space into itself or another, which we shall use frequently, consists of the collection of contractive transformations. Given two metric spaces $\left(\mathcal{X},\rho_1\right)$ and $\left(\mathcal{Y},\rho_2\right)$, a transformation $T:\mathcal{X}\to \mathcal{Y}$ is said to be a \emph{contraction} if and only if there exists a real number $r\in [0,1]$, such that for all $x_1,x_2\in \mathcal{X}$ 
\begin{align}
\rho_2\left(T(x_1),T(x_2)\right)\le r \rho_1\left(x_1,x_2\right).
\end{align} 
$r$ is called a \emph{contractivity factor} for $T$. $T$ is called a \emph{strict contraction} if $r<1$. Thus when $T$ is a self-transformation on $\mathcal{X}$, i.e. $T: \mathcal{X}\to \mathcal{X}$, then it acts on pairs of points in $\mathcal{X}$ by bringing them closer together, their distance is reduced by a factor at least $r$. Let $M_n\left(\mathbb{R}\right)$ denotes the set of all $n\times n$ matrices with entries from $\mathbb{R}$, $T$ is called an \emph{affine transformation} on $\mathbb{R}^n$ if for some $A\in M_{n}(\mathbb{R})$ and a vector $b\in \mathbb{R}^n$, $T$ is of the form
\begin{align}
T: \mathbb{R}^n\to \mathbb R^n \text{ and } T(x)=Ax+b.
\end{align}
The next theorem describes a contractive transformation on a complete metric space to itself and known \emph{contraction mapping theorem}. This result confirms the existence of a single fixed point for the contraction map $T$ and proposes a mechanism for calculating a fixed point.
\begin{theorem}
\cite[Section 3.6, Theorem 1]{Barnsley1993}\label{th:ch2-fpt}
Let $T: \mathcal{X}\to \mathcal{X}$ be a strict contraction on a complete metric space $\left(\mathcal{X}, \rho\right)$. Then, there exists a unique point $x^{\star}\in \mathcal{X}$ such that $T\left(x^{\star}\right)=x^{\star}$. Furthermore, for any $x\in \mathcal{X}$, we have
\begin{align*}
\lim\limits_{n\to \infty} T^{n}\left(x\right)=x^{\star},
\end{align*}
where $T^n$ denotes the $n$-fold composition of $T$ with itself.
\end{theorem}
Consider a metric space $\left(\mathcal{X},\rho\right)$ and a finite set of strictly contractive transformations $w_k:\mathcal{X}\to \mathcal{X}$ $1\le k\le m$, with respective contractivity factors $r_k$. Define a  transformation $w: \mathcal{H}\left(\mathcal{X}\right)\to \mathcal{H}\left(\mathcal{X}\right)$, where $\mathcal{H}\left(\mathcal{X}\right)$ is the collection of nonempty, compact subsets of $\mathcal{X}$, by:
\begin{align}\label{eq:ch2-trans-on-fract-space}
w(X)=\bigcup\limits_{k=1}^{m} w_n\left(X\right), \quad \text{for any } X\in \mathcal{H}\left(\mathcal{X}\right).
\end{align}
Since $w$ is a strict contraction with contractivity factor $r^{\star}=\max\limits_{1\le k\le m}\{r_i\}$\cite[Section 3.7, Lemma 5]{Barnsley1993},
and $\left(\mathcal{H}\left(\mathcal{X}\right),\rho_{\mathrm{h}}\right)$ is a complete metric space\cite[Section 2.7, Theorem 1]{Barnsley1993}, we conclude that $w$ posseses a unique fixed point by Theorem~\ref{th:ch2-fpt}. Thus if we denote the fixed point by $X^{\star}$, then the map $w$ satisfy the following self-covering condition
\begin{align}\label{eq:ch2-self-covering}
w(X^{\star})=\bigcup\limits_{k=1}^{m} w_k\left(X^{\star}\right).
\end{align}
We are now able to define what is called hyperbolic IFS formally.
\begin{definition}[Hyperbolic IFS; \cite{Hutchinson1981}]
A hyperbolic IFS $\{\mathcal{X}; w_1,w_2,\dots,w_m\}$ consists of a complete metric space $\left(\mathcal{X}, \rho\right)$ and a finite set of strictly contractive transformations $w_k: \mathcal{X}\to \mathcal{X}$ with contractivity factors $r_k$, for $k=1,2,\dots, m$. The maximum  $r^{\star}=\max\limits_{1\le k\le m}\{r_i\}$ is called \emph{a contractivity factor of the IFS} and the unique fixed point $X^{\star}\in \mathcal{H}\left(\mathcal{X}\right)$ of the transformation defined in \eqref{eq:ch2-trans-on-fract-space} is called \emph{the attractor of the IFS}. 
\end{definition}

\begin{definition}[Modulus of uniform continuity]\label{dfn:ch2-modulus-uni-cont}
Let $w:\mathcal X\to \mathbb R$ be a uniformly continuous function. The \emph{modulus of uniform continuity} is the function $\phi: [0,\infty)\to [0,\infty)$ defined by
\begin{align}\label{eq:ch2-modulus-uni-cont}
\phi(t):=\sup _{\myatop{x, y \in \mathcal X} {\rho(x, y) \leq t}}\{|w(x)-w(y)|\}.    
\end{align}
\end{definition}

\begin{definition}[Dini continuous]\label{dfn:ch2-dini-cont}
A function $w:\mathcal X\to \mathbb R$ is said to be \emph{Dini continuous} if it has modulus of continuity $\phi(t)$ such that 
\begin{align}\label{eq:ch2-dini-cont}
\frac{\phi(t)}{t} \text{ is integrable over } (0,\delta) \text{ for some } \delta>0     
\end{align} and we say that $\phi$ satisfies the \emph{Dini's condition}.
\end{definition}
Note that Lipschitz and Holder continuous functions are Dini continuous and that Dini continuity is a stronger condition than continuity.
\begin{definition}[Total-variation (TV) distance; Proposition 4.2 in \cite{Levin2006}]\label{dfn:ch2-tv-dist}
Let $\nu_1, \nu_2 \in \mathcal{M}_{\text{p}}\left(\mathcal{X}\right)$, then throughout the article, the \emph{total-variation distance} between $\nu_1$ and $\nu_2$ will be denoted by $d_{\text{tv}}(\nu_1,\nu_2)$, and is expressed as
\begin{align}\label{eq:ch2-tv-dist}
d_{\text{tv}}\left(\nu_1,\nu_2\right)=\sup_{ \mathcal{A}\in \mathcal{B}\left(\mathcal X\right)}\left|\nu_1\left(\mathcal A\right)-\nu_2\left(\mathcal A\right)\right|.
\end{align}
If $\mathcal X$ is finite with cardinality $n<\infty$, one can show the above expression is equivalent to the following:
\begin{align}\label{eq:ch2-equiv-tv-dist}
d_{\text{tv}}\left(\nu_1,\nu_2\right)=\frac{1}{2}\sum_{i=1}^n\left| \nu_1\left(i\right) -\nu_2\left(i\right)\right|.
\end{align}
\end{definition}

\begin{definition}[Prohorov metric \cite{Billingsley2013}]\label{dfn:ch2-proho-dist}
Let $\nu_1, \nu_2\in \mathcal M_p(\mathcal X)$ then the \emph{Prohorov distance} between $\nu_1$ and $\nu_2$ is defined as 
\begin{align}\label{eq:ch2-proho-dist}
d_{\text{pr}}(\nu_1,\nu_2)=\inf\{\delta>0: \nu_1(\mathcal K)<\nu_2(\mathcal K_{\delta})+\delta \text{ and } \nu_2(\mathcal K)<\nu_1(\mathcal K_{\delta})+\delta 
\end{align}
for all compact subsets  $\mathcal K \text{ of } \mathcal X\}$
and $\mathcal K_{\delta}:=\{ x\in \mathcal X: \rho(x, \mathcal K)< \delta\}$. \cite{Dudley1976} is a standard reference for further discussions on \emph{Prohorov metric}.
\end{definition}

\begin{definition}[Kolmogorov-Smirnov distance]\label{dfn:ch2-kolmo-smir-dist}
For $\mathcal X=\mathbb R$, \emph{Kolmogorov-Smirnov distance} between $\nu_1, \nu_2 \in \mathcal{M}_{\text{p}}\left(\mathcal{X}\right)$ is denoted by $d_{\text{ks}}$ and is defined by
\begin{align}\label{eq:ch2-kolmo-smir-dist}
d_{\text{ks}}\left(\nu_1,\nu_2\right)= \sup\limits_{x\in \mathcal X}\left|\nu_1\left((-\infty, x]\right)-\nu_2\left((-\infty,x]\right)\right|
\end{align}
In, other words if $F_{\nu_1}$ and $F_{\nu_2}$ denotes the distribution function corresponding to the probability measure $\nu_1$ and $\nu_2$ respectively then  \eqref{eq:ch2-kolmo-smir-dist} is equivalent to the following expression:
\begin{align}\label{eq:ch2-kolmo-smir-dist-equ}
d_{\text{ks}}(\nu_1,\nu_2)= \sup\limits_{x\in \mathcal X}\left|F_{\nu_1}(x)-F_{\nu_2}(x)\right|    
\end{align}
\end{definition}

The following metric is due to \emph{Kantorovich and Rubinstein} \cite{Evans1999, Mass1998, Villani2009}, also known as \emph{Wasserstein-$1$ distance}.
\begin{definition}[Wasserstein-$1$ distance; Remark 6.5, p. 95 in \cite{Villani2009}]\label{dfn:ch2-Wass-dist}
Let $\mathcal{W}_1$ denote the space of all Lipschitz maps with Lipschitz constant $1$, i.e
\begin{align*}
\mathcal W_1=\{w\in \mathcal C\left(\mathcal X, \mathbb R\right): |w(x)-w(y)|\le \rho(x,y) \quad \forall x, y\in \mathcal X\}.
\end{align*} 
For $\nu_1, \nu_2\in \mathcal M_{\text{p}} \left(\mathcal X\right)$,  $d_{\text{ws}}\left(\nu_1, \nu_2\right)$ denotes the Wasserstein-$1$ distance between $\nu_1, \nu_2\in \mathcal M_{\text{p}} \left(\mathcal X\right)$  and is given by:
\begin{align}\label{eq:ch2-Wass-dist}
d_{\text{ws}}\left(\nu_1, \nu_2\right)=\sup\limits_{w\in \mathcal W_{1}}\left[\int\limits_{\mathcal X} w\mathrm{d}\nu_1-\int\limits_{\mathcal X} w \mathrm{d}\nu_2\right].
\end{align}
\end{definition}
For more on different kinds of metrics on the space of probability measures on metric space and their usefulness, see \citep{Parthasarathy1967, Rachev1991, Rao1998, Gibbs2002, Gibbs2000, Kravchenko2006, Mostafa2011}.

\begin{remark}\label{rem:MK-metric}
The \emph{Monge-Kantorovich} metric is another name for the metric mentioned above \cite[Definition B.28]{Kunze2012}. In the context of mass transportation, this metric was developed. Please see \cite{Hanin1992, Hanin1999, Kravchenko2006, Vershik2006} for more general findings and a historical timeline. Convergence in this metric is related to weak convergence, as seen from the definition of $d_{\text{ws}}$. Convergence of probability measures in a compact metric space in  $d_{\text{ws}}$ and weak convergence are equivalent \cite[Proposition B.29]{Kunze2012}. When $\mathcal{X}$ is compact, the space $\left(\mathcal{M}_\text{p},d_{\text{ws}}\right)$ is complete  \cite{Hanin1992, Hanin1999, Kravchenko2006, Weaver2018}. Furthermore, this distance is difficult to calculate since finding an optimal $w\in \mathcal{W}_1$  to maximize the difference in \eqref{eq:ch2-Wass-dist} is generally not straightforward.
\end{remark}

\begin{definition}[Markov Chain on Metric Space]\label{dfn:ch2-MC}
Let $(\Omega, \mathcal{E}, \mathbb P)$ be a probability space. A Markov chain on a metric space $\mathcal{X}$ is a sequence of $\mathcal{X}$ valued random variables $\{X_k\}_{k\ge 0}$ whose interdependence satisfies the Markov property, i.e,
\begin{align}
\mathbb P\left[X_{k+1}\in \mathcal{A}\mid X_0, X_1,\dots, X_k\right]=\mathbb P\left[X_{k+1}\in \mathcal{A}\mid X_k=x\right]\quad \text{ for all } \mathcal{A}\in \mathcal{B}(\mathcal{X}).
\end{align}
\end{definition}

\begin{definition}[Stochastic Kernel or Transition Probability Function]\label{dfn:ch2-stoch-ker}
For each $x\in \mathcal{X}$ and $\mathcal{A}\in \mathcal{B}(\mathcal{X})$, let 
\begin{align}
\nu(x, \mathcal{A}):=\mathbb{P}\left[X_{k+1}\in \mathcal{A}\mid X_k=x\right]
\end{align}
is called a stochastic kernel on $\mathcal{X}$ or a transition probability function of the chain $\{X_k\}$, which means
\begin{itemize}
\item [(a)] $\forall x\in \mathcal{X}$, $\nu(x,\cdot)$ is a probability measure on $\mathcal{X}$, hence a function from  $\mathcal{X}\times \mathcal{B}(\mathcal X)$ to $[0,1]$, so if we fix some $x\in \mathcal{X}$, $\nu(x, \mathcal A)=\sum\limits_{y\in A}p_{x,y}$. Thus, $\nu(x, A)$ is simply the probability of reaching some state in $\mathcal {X} $ given that the current state is $x$. When $\mathcal{X}$ is discrete or, more precisely, finite, the transition kernel is simply a  transition matrix $\nu$ with entries $p_{x,y}$ as follows:
\begin{align*}
p_{x,y}=\mathbb P\left[X_n=x\mid X_{n-1}=y\right]\quad x,y\in \mathcal{X}.
\end{align*}
\item[(b)] $\forall \mathcal{A}\in \mathcal{B}(\mathcal{X})$, $\nu(\cdot, \mathcal{A} )$ is a measurable function from $\mathcal{X}\times \mathcal{B}(\mathcal X)$ to $\mathcal{X}$.
\end{itemize}
One can introduce the \emph{$1$-step transition probability} as
\begin{align}\label{eq:ch2-one-step-transi}
\mathbb{P}\left[X_{1}\in \mathcal{A}\mid X_0=x\right]=\nu(x, \mathcal{A}).
\end{align}
One can also introduce the \emph{$k$-step transition probability} as
\begin{align}\label{eq:ch2-n-step-transi}
\mathbb{P}\left[X_{k+1}\in \mathcal{A}\mid X_0=x\right]=\nu^k(x, \mathcal{A}).
\end{align}
\end{definition}
\begin{definition}[Feller Chain; \cite{Lasserre2003}]\label{dfn:ch2-Feller-chain}
Let $\{X_k\}_{k\ge 0}$ be a Markov chain on a compact metric space $\mathcal{X}$, and then the chain is said to be a Feller chain if for every continuous function $\phi:\mathcal{X}\to \mathbb{R}$, the function
\begin{align}
x\mapsto \mathbb E\left[\phi(X_1)\mid X_0=x\right] 
\end{align}
is continuous.
\end{definition}
\begin{definition}[Weak-Feller Chain; \cite{Lasserre2003}] \label{dfn:ch2-weak-feller}
Let $\{X_k\}_{k\ge 0}$ be a Markov chain on a metric space $\mathcal{X}$ with transition probability function $\nu$. Then the chain is said to satisfy the Weak-Feller property if every sequence $x_n\in \mathcal{X}$ such that $x_n\to x\in \mathcal{X}$, $\nu f(x_n)\to \nu f(x)$ whenever $f$ is bounded and continuous function on $X$.
\end{definition}
\begin{definition}\label{dfn:ch2-weak-flr-prop}
A Markov chain is said to have weak Feller property if its transition probability kernel (see Definition~\ref{dfn:ch2-stoch-ker}) maps a continuous function to another continuous function, i.e., it leaves $C(\mathcal{X})$ invariant.
\end{definition}
\begin{definition}\label{dfn:ch2-dirac-msr}
Let $\mathcal{A}\subseteq \mathcal X$, the Dirac measure on a point $x\in \mathcal{ X}$ is generally denoted by $\delta_x$ and defined as 
\[ \delta_x(\mathcal{A})=\begin{cases} 
      1 & x\in \mathcal{A} \\
      0 & x\in \mathcal{A}^c \\
   \end{cases}
\]
\end{definition}

\begin{remark}\label{rem:IFS-MC-unique}
Any discrete-time Markov chain can be generated by an \emph{IFS with probabilities} see \cite[Section 1.1]{Kifer2012} or \cite[Page 228]{Bhattacharya2009}, although such representation is not unique, see, Stenflo \cite{Stenflo1999}.
\end{remark}

\begin{figure}[ht!]
    \centering
    \includegraphics[width=1.02\textwidth]{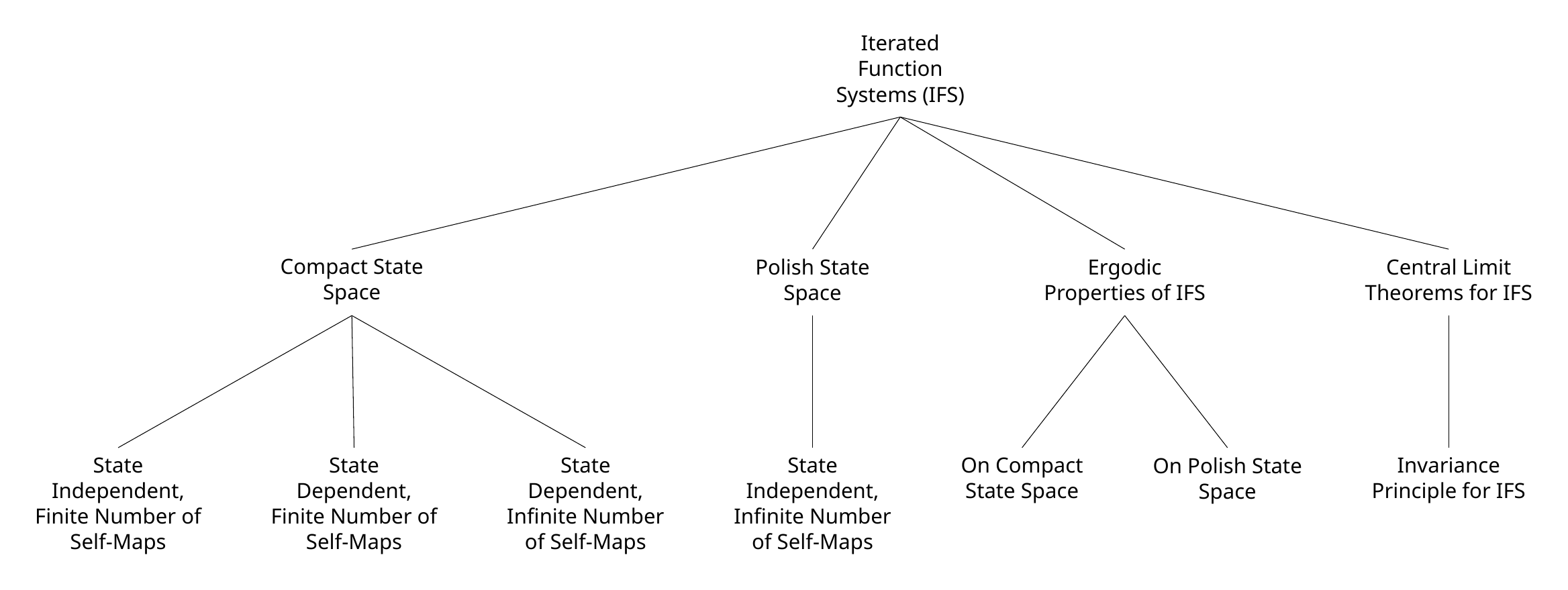}
    \caption{Sections \ref{sec:IFS-cpt-space}--\ref{sec:CLT-IFS-invp} of the survey try to partition the current IFS literature. While we do not yet provide a ``plant atlas view, '' we hope to facilitate the retrieval of a particular theorem. }
    \label{fig:des-1}
\end{figure}

\section{IFS on Compact State Space}\label{sec:IFS-cpt-space}
\subsection{State Independent IFS with Finite Number of Self Maps}
\label{sec:ch2-ifswsindpwfnm}
\label{subsec:ch2-dfn-ifswsindpwfnm}
\begin{definition}[IFS with state-independent probability with finite number of maps]\label{df:ch2-ifswsindp}
Let $\left(\mathcal X, \rho\right)$ be a compact metric space. Let $\{w_i\}_{i=1}^{m}$ be continuous self-transformations on $\mathcal X$ and $\mu$ be a probability measure on $[m]$. Let $\sigma_0, \sigma_1,\ldots$ be i.i.d discrete random variables taking values in $[m]$ and
\begin{align}\label{eq:ch2-ifswsip-disc-prob}
\mathbb P\left(\sigma_i=j\right)=\mu\left(j\right) \quad j\in [m]\text{ and } i=0, 1, 2, \dots, m.
\end{align}
An \emph{IFS with state-independent probability} is a Markov process that is realized by the recursion
\begin{align}\label{eq:ch2-ifswsindp}
X_{k+1}&:= w_{\sigma_k}(X_k), \quad k=0, 1, 2,\dots,
\end{align}
where $X_0\sim \nu\in \mathcal M_p\left(\mathcal X\right)$ is assumed to be independent of $\sigma_k \; \forall k$.
\end{definition}
We describe below how the iteration progress in Algorithm \ref{algo:ch2-ifswsidp}.

\begin{algorithm}[H]
\caption{State-Independent IFS Iteration}
\label{algo:ch2-ifswsidp}
\begin{algorithmic}[1]
\Statex \textbf{Input}: Initial value \(X_0\), \(m\) functions $\{w_k\}_{k=1}^{m}$, probability measure \(\mu\) over $[m]$
\Statex \textbf{Output:} A sequence of random variables \(X_k\) with Markov property (see Definition~\ref{dfn:ch2-MC})
\Statex Step 1: Set $n = 0$ 
\Statex Step 2: Repeat the following:
\Statex Step 2.1: Draw \((\sigma_n)\)  from \(\mu\)
\Statex Step 2.2: $X_{n+1} =  w_{\sigma_n}\left( X_n\right)$
\Statex Step 3: Return  {\(\{X_k\}_{k\in\mathbb{N}}\)}
\end{algorithmic}
\end{algorithm}

\begin{definition}[Markov operator]\label{dfn:ch2-markov-op}
To describe the evolution of the Markov process described by \eqref{eq:ch2-ifswsindp}, we define an operator, known as \emph{Markov operator} $\mathrm{P}$,
\begin{align}\label{eq:ch2-Markov-ifswsip}
\mathrm{P}: C\left(\mathcal{X}, \mathbb{R}\right)\to C\left(\mathcal{X}, \mathbb{R}\right),
\mathrm{P}\left(w\left(x\right)\right)=\mathbb{E}\left[w\left(X_{k+1}\right)\mid X_k\right]=\sum\limits_{\substack{k\in \mathbb{N}\cup\{0\}\\ \sigma_k\in [m]}}\mu\left(\sigma_k\right)\left(w\circ w_{\sigma_k}\right)\left(x\right)
\end{align}
\end{definition}
To find the dual of the operator \eqref{eq:ch2-Markov-ifswsip}, suppose the elements or points in the metric space $\left(\mathcal{X},\rho\right)$ are distributed according to a measure $\nu_0$. For each point, we throw a dice independently with $m$ sides for which the side $i$ falls with probability $\mu(i)$, $i=1,2,\dots,m$. If $i$ falls for the point at position $x_0$, we move the point to $w_{i}\left(x_0\right)$, in other words, if $\sigma_0,\sigma_1,\dots$ is a sequence of i.i.d discrete random variable taking values in $[m]$, starting from some $x_0\in \mathcal{X}$, then the point $x_0$ moves to $w_{\sigma_0}\left(x_0\right)$ with probability $\mu\left(\sigma_0\right)$ as stated in \eqref{eq:ch2-ifswsip-disc-prob}. The new amount of elements in a set $\mathcal{A}$ come from the set $w_{\sigma_0}^{-1}\left(\mathcal{A}\right)$ with probability $\mu\left(\sigma_0\right)$. The elements transported to $\mathcal{A}$ by $w_{\sigma_0}$ have then mass $\mu\left(\sigma_0\right)\nu_0\left(w_{\sigma_0}^{-1}\left(\mathcal{A}\right)\right)$. Hence the new distribution of the elements in $\mathcal{X}$ is
\begin{align}\label{eq:ch2-mass-transport}
\nu_1\left(\mathcal{A}\right)&=\mu\left(\sigma_0\right)\nu_0\left(w_{\sigma_0}^{-1}\left(\mathcal{A}\right)\right)+\mu\left(\sigma_1\right)\nu_0\left(w_{\sigma_1}^{-1}\left(\mathcal{A}\right)\right)+\cdots+\mu\left(\sigma_m\right)\nu_0\left(w_{\sigma_m}^{-1}\left(\mathcal{A}\right)\right)
\end{align}
Thus one can consider for the IFS with maps $w_1, w_2,\dots, w_m$ and the probabilities $\mu\left(1\right),\ldots, \mu\left(m\right)$, an operator of the evolution of densities of elements under the action of the IFS as defined below.
\begin{definition}\label{dfn:ch2-inv-prob-meas}
For an IFS as in Definition~\ref{df:ch2-ifswsindp}, the dual operator of the Markov operator defined \eqref{eq:ch2-Markov-ifswsip} is given by
\begin{align}\label{eq:ch2-dual-Markov-ifswsip}
{\mathrm{P}}^{\star}: \mathcal{M}\left(\mathcal{X}\right)\to \mathcal{M}\left(\mathcal{X}\right), \mathrm{P}\left(\nu\right)=\sum\limits_{\substack{k\in \mathbb{N}\cup\{0\}\\ \sigma_k\in [m]}}\mu\left(\sigma_k\right)\nu\left(w_{\sigma_k}^{-1}\left(\mathcal{A}\right)\right);
\end{align}
A probability measure $\nu$ on $\mathcal X$ is called invariant under $\mathrm{P}^{\star}$, if $\mathrm{P}^{\star}\nu=\nu$.
\end{definition}
For the dynamical systems with only one transformation $w: \mathcal{X}\to \mathcal X$, the operator $\mathrm{P}^{\star}$ is known as the \emph{Perron-Frobenius operator}, see \cite[Section 3.2]{Lasota1998}. The operator defined in the Definition~\eqref{eq:ch2-dual-Markov-ifswsip} is also known as \emph{Foais operator} \cite[Section 12.4]{Lasota1998}.

\subsection{Ergodic Properties of the Associated Markov Processes}\label{subsec:ergo-prop-mp1}
\label{subsec:ch2-ergo-prop-ifswsindpwfnm}
It may be critical in practice to determine if an IFS is uniquely ergodic or not, both theoretically and practically, since this will facilitate the conclusion of simulation results for the process.
If we simulate an ergodic weak-Feller chain on a compact space, the sampled points' distributions will converge to the \emph{invariant measure}. To establish the ergodic property of the related Markov process created by an IFS, \cite{Breiman1960} established the following version of the ergodic theorem noted in \cite[Chapter 6, Theorem 6.7]{Corless2016}.
\begin{theorem}[\cite{Breiman1960}]\label{thm:ch2-breiman-ergo}
Let $\{X_k\}$ be a Markov chain on a compact metric space $\left(\mathcal{X}, \rho\right)$ with a unique-invariant distribution $\nu^{\star}$, suppose also that the chain is Feller (see Definition~\ref{dfn:ch2-Feller-chain}), then for every initial condition $X_0=x_0\in \mathcal{X}$ and any continuous function $w: \mathcal{X}\to \mathbb{R}$, the sampled points will satisfy the following limiting behavior
\begin{align}\label{eq:ergodicity-breiman}
\lim_{k\to \infty}\frac{1}{k+1}\sum\limits_{i=0}^{k} w\left(X_i\right)=\int w(x)  \nu^{\star}\left(\mathrm{d}x\right)  \text{ almost surely } \mathbb{P}_{x_0}. 
\end{align}
\end{theorem}
Using several \emph{contractivity condition} is a critical strategy for establishing the unique ergodicity of an IFS. There are two distinct ways to characterise these contractions. One may consider contractivity in terms of individual trajectories. If a process is contractive, the trajectories of individuals beginning from distinct positions should approach each other over time. Thus, there will be one random trajectory in the long run, and the behaviour will be determined only by the original distribution. Another method for demonstrating contractivity is to use the weak convergence of several probability metrics, as described in \cite[Section 2.2]{Niclas2005(1)}. Contractivity conditions assure that the trajectories can meet regardless of their starting points see \cite{Kaijser1981, Kaijser2015}. We refer \cite[Section 1.1]{Kaijser1978} for sufficient conditions to show the uniqueness of invariant measures. We refer \cite[Table 2.1, Section 2.1]{Niclas2005(1)} for a quick overview of several conditions, for example, \emph{local and global arithmetic mean, geometric mean, and power mean} conditions. We refer  \cite[Lemma 1, Section 2.1]{Niclas2005(1)} where the relation between power mean and geometric mean conditions were established following a necessary contraction condition from \cite{Barnsley1988(2), Barnsley1988(2)erratum}. For more on contraction conditions of IFS, e.g. \emph and a discussion on \emph{non-contractive IFS}, we suggest \cite{Krzysztof2020} and the references therein.

\subsection{State-Dependent IFS with Finite Number of Self Maps}\label{subsec:sdIFS-FN-SM}
\label{sec:ch2-ifswsdpwfnm}
\label{subsec:ch2-dfn-ifswsdpwfnm}
Let\emph{$\left(\mathcal{X}, \rho\right)$ be a compact metric space}, in this section, we discuss definitions, evolution, and ergodic properties of IFS whose dynamics generated by a finite number of self-maps with state or place dependent probabilities. 
\begin{definition}[IFS with state-dependent probabilities with finite number of self-maps] \label{dfn:ch2-ifswsdpwfnm}
Let $\{w_k\}_{k=1}^{m}$ be continuous self-maps on $\mathcal X$ and $\{p_k(x)\}_{k=1}^{m}$ be probability functions such that, 
\begin{align*}
p_k(x):\mathcal X \to [0,1] \quad \forall k\in [m], \text{ and } \sum\limits_{k=1}^{m} p_k(x)=1.
\end{align*}
The pair of $(\{w_k\}_{k=1}^{m}, \{p_k(x)\}_{k=1}^{m})$ 
is called an IFS.
\end{definition}

Given an initial condition $x_0\in \mathcal{X}$, and a sequence $\{\sigma_k\}$ in $[m]$, one generates $\{x_k\}_{k\ge 0}$ through the following recurrence relation
\begin{align}\label{eq:ch2-iterate-rel}
x_{k+1}=w_{\sigma_k}\left(x_k\right), \quad k=0,1,2,\dots.
\end{align} 
The semantics of the definition is captured in Algorithm \ref{algo:ch2-ifswsdp}.

\begin{algorithm}[H]
\caption{State-dependent IFS Iteration}\label{algo:ch2-ifswsdp}
\begin{algorithmic}[1]
\Statex \textbf{Input}: Initial value \(x_0\), \(m\) functions $\{w_k\}_{k=1}^{m}$, $m$ probability functions $\{p_k\}_{k=1}^{m}$
\Statex \textbf{Output:} A sequence of random variables \(X_k\) with Markov property (see Definition~\ref{dfn:ch2-MC})
\Statex Step 1: Set $n = 0$ 
\Statex Step 2: Repeat the following:
\Statex Step 2.1 Calculate \(\boldsymbol{p}(X_n) = (p_1(X_n), p_2(X_n), \dotsc, p_m(X_n))\)
\Statex Step 2.2: Select \(w_{\sigma_{n+1}} \in \{1, \dotsc, m\}\) according to \(\boldsymbol{p}(X_n)\)
\Statex Step 2.3: Apply \(w_{\sigma_{n+1}}\) to the state:  \(X_{n+1} = w_{\sigma_{n+1}}(X_n)\)
\Statex Step 3: Return  {\(\{X_k\}_{k\in\mathbb{N}}\)}
\end{algorithmic} 
\end{algorithm}

To see the evolution of the IFS the kernel (see Definition~\ref{dfn:ch2-stoch-ker}) of the corresponding Markov chain should be obtained. Given a state $x\in \mathcal X$ at time $k$, the set of possible states at the next time $k+1$ is the finite set
\begin{align*}
\{w_{\sigma}\left(x\right): \sigma\in [m]\}
\end{align*}
and the occurrence of $w_{\sigma}$ has the probability $p_{\sigma}\left(x\right)$. Hence, for any event $\mathcal A$,
\begin{align*}
\mathbb{P}\left( X_{k+1}\in \mathcal{A}\mid X_k=x\right)=\sum\limits_{\sigma: w_{\sigma}\left(x\right)\in \mathcal{A}} p_{\sigma}\left(x\right).
\end{align*} 
The kernel for this Markov process is given by 
\begin{align}\label{eq:ch2-exprssn-kernel}
\nu\left(x, \mathcal{A}\right)=\sum\limits_{\sigma: w_{\sigma}\left(x\right)\in \mathcal{A}} p_{\sigma}\left(x\right)=\sum\limits_{\sigma: x\in \left(w_{\sigma}\right)^{-1}\left(\mathcal{A}\right)} p_{\sigma}\left(x\right)=\sum\limits_{\sigma\in [m]} p_{\sigma}\left(x\right)\mathrm{1}_{\left(w_{\sigma}\right)^{-1}\left(\mathcal{A}\right)}\left(x\right).
\end{align}
This Markov chain's kernel is hence a Dirac-measure's sum (see Definition~\ref{dfn:ch2-dirac-msr})
\begin{align}\label{eq:ch2-kernel-as-sum-Dirac-msr}
\nu\left(x, \cdot\right)=\sum\limits_{\sigma\in [m]} p_{\sigma}\left(x\right) \delta_{w_{\sigma}\left(x\right)}.
\end{align}
To describe evolution of the dynamics \eqref{eq:ch2-iterate-rel}, let us define the \emph{Markov operator}  $\mathrm{P}$ (see Definition~\ref{dfn:ch2-markov-op}), note that, for any function $g:\mathcal{X}\to \mathbb{R}$,
\begin{align}\label{eq:ch2-expect-func-of-mc}
\mathbb{E}\left[g\left(X_{k+1}\right)\mid X_k=x\right]= \sum\limits_{\sigma\in [m]} p_{\sigma}\left(x\right) g\left(w_{\sigma}\left(x\right)\right).
\end{align}
Hence,
\begin{align}\label{eq:ch2-markov-operator}
\left(\mathrm{P}g\right)(x)=\sum\limits_{\sigma\in [m]} p_{\sigma}\left(x\right) \left(g\circ w_{\sigma}\right)\left(x\right).
\end{align}
The Markov operator's Feller property is crucial, which says $\left(\mathrm{P}g\right)\in C_{b}\left(\mathcal{X}\right)$ whenever $g\in C_{b}\left(\mathcal{X}\right)$. It is clear that the Markov operator defined in \eqref{eq:ch2-markov-operator} is indeed has \emph{Feller property} see \cite[Theorem 4.22]{Hairer2006}.
Since $\mathcal{X}$ is compact, all functions in $C\left(\mathcal{X}, \mathbb{R}\right)$ are bounded, and  $\mathrm{P}: C\left(\mathcal{X}, \mathbb{R}\right)\mapsto C\left(\mathcal{X}, \mathbb{R}\right)$, by Riesz representation theorem (originally due to \cite{Riesz1909} or see \cite[Chapter 6, Section 4]{Royden1988} for a proof), we can conclude that the dual space of $C\left(\mathcal{X}, \mathbb{R}\right)$ is is the space of all signed or complex measures on $\mathcal{X}$ with the total-variation norm. The dual map $\mathrm{P}^{\star}$ is given by
\begin{align}
\mathrm{P}^{\star}\nu\left(\mathcal A\right)=\sum\limits_{i=1}^{m} \left(p_i\circ w_i^{-1}\left(\mathcal{A}\right)\right)\mathrm{d}\left(\nu\circ w_i^{-1}\left(\mathcal{A}\right)\right)\quad \text{ for } \mathcal{A}\in \mathcal{B}\left(\mathcal X\right).
\end{align} 
where we define
\begin{align}
\left(\nu\circ w_i^{-1}\right)\left(\mathcal{A}\right)= \nu\left(w_i^{-1}\left(\mathcal{A}\right)\right)\quad \text{ for all events } \mathcal{A}\in \mathcal{B}\left(\mathcal X\right).
\end{align}
and it is not required that $w_i$ are invertible maps, but we interpret
\begin{align}
w_i^{-1}\left(\mathcal{A}\right)=\{ x\in \mathcal{X}: w_i(x)\in \mathcal{A}\}\text{ as the pre-image of } \mathcal{A}\text{ under the map } w_i.
\end{align}
The theory of Markov operators began in 1906 when Markov demonstrated that stochastic matrices might be used to study the asymptotic features of certain stochastic processes \cite{Markov1906}. Positive linear operators on $\mathbb R^n$ are defined by these matrices. Markov's concepts have been expanded in several areas. Feller developed, in particular, the theory of Markov operators operating on Borel measures defined on certain topological spaces. In \cite{Nummelin1984, Revuz1975, Cinlar1975, Ethier1986, Foguel1966(1),  Foguel1966(2), Foguel1966(3), Foguel1969} one finds some historical notes and a plethora of literature.
\begin{definition}[Contraction on average in $1$ step]\cite[Section 2.2]{Anthony2015}\label{dfn:ch2-contra-on-av-1-step}
We say that an IFS contracts on average on the metric space $(\mathcal X, \rho)$ after $1$ step if there exists an $0<r<1$ such that 
\begin{align}\label{eq:ch2-contra-on-av-1-step}
\sup_{\substack{x,y,z\in \mathcal X\\ y\ne z}} \sum_{i} p_{i}(x) \frac{\rho\left(w_{i}(y), w_{i}(z)\right)}{\rho(y, z)} \leq r.
\end{align}
\end{definition}
This can be generalised to define IFS that contract on average after, say, $k$ step where $k\in \mathbb N$. Let
\begin{align*}
\mathcal I:=\{\sigma=(\sigma_k)_{k=1}^{\infty}: \sigma_k\in [m]\}
\end{align*}
represents the collection of all conceivable sequences of maps. Let $\circ$ denote the usual composition operation of functions, for $x_0\in \mathcal X$, $\sigma\in \mathcal I$, $k\in \mathbb N$, define 
\begin{align}\label{eq:ch2-possible-sequ-maps}
F_k(x, \sigma):=\left\{\begin{array}{ll}
x_0, & \text { if } k=0 \\
\left(w_{\sigma_{k}}\circ\ldots\circ w_{\sigma_{1}}\right)\left(x_0\right), & \text { if } k \geq 1.
\end{array}\right.
\end{align}
For $x_0\in \mathcal X$, we define $\mu_{x_0}$ on $\mathcal I$ for some $\sigma=(\sigma_1,\dots, \sigma_k)$ as follows:
\begin{align}\label{eq:ch2-mu-on-sequ-maps}
\mu_{x_0}\left(\sigma\right)=p_{\sigma_1}\left(x\right)\times p_{\sigma_2}\left(w_{\sigma_1}\left(x\right)\right)\times\ldots\times p_{\sigma_k}\left(\left(w_{\sigma_{k-1}}\circ w_{\sigma_{k-2}}\circ\ldots\circ w_{\sigma_1}\right)\left(x_0\right)\right),
\end{align}
to quantify the probability that we apply the sequence of maps $w_{\sigma_1},\dots,w_{\sigma_k}$ given that we have started from the point $x_0$. It is easy to notice that $\left(F_k\left(x_0,\cdot\right)\right)_{k\ge 1}$ is a Markov chain with transitional probability $\mu_{x_0}$. 
Let $\mathcal M_k: [m]^k\to \mathbb R$ be a function on $\mathcal I$ depending only on the first $k$ coordinates. Then expectation with respect to $\mu_{x_0}$ is defined as \citep[Section 2.2]{Anthony2015}
\begin{align}\label{eq:ch2-expec}
\mathbb E_{x_0}\left[\mathcal M_k\right]:=\int \mathcal M_k\left(\sigma\right)\mathrm{d}\mu_{x_0}\left(\sigma\right)=\sum\limits_{\sigma_1,\dots,\sigma_k\in [m]}\mu_{x_0}\left(\sigma\right) \mathcal M_k\left(\sigma\right).
\end{align}
\begin{definition}[Contraction on average in $k$ steps]\cite[Section 2.2]{Anthony2015}\label{dfn:ch2-contra-on-av-k-step}
We say an IFS is \emph{contract on average after $k$ steps} if there exists $r\in (0,1)$ such that 
\begin{align}\label{eq:ch2-contra-on-av-k-step}
\sup _{\myatop{x, y, z \in \mathcal X}{y\ne z}} \sum\limits_{\sigma_{1}, \ldots, \sigma_{k}} \mu_{x_0}\left(\sigma\right) \frac{\rho\left(F_k\left(y, \sigma\right), F_{k}\left(z, \sigma\right)\right)}{\rho\left(y, z\right)} \leq r.
\end{align}
If for a $\sigma=\left(\sigma_1,\dots, \sigma_k\right)$, we denote $w_k\left(\sigma\right)=\left(w_{\sigma_{k}}\circ\ldots\circ w_{\sigma_{1}}\right)\left(x\right)$ and assume that each map is Lipschitz and define the Lipschitz norm $\|\cdot\|_{\text {Lip }}$ as follows:
\begin{align}\label{eq:ch2-lipschitz-norm}
\left\|w\right\|_{\operatorname{Lip}}=\sup _{\myatop{x, y \in X} {x \neq y}} \frac{\rho\left(w\left(x\right), w\left(y\right)\right)}{\rho\left(x, y\right)},
\end{align}
then we can re-write \eqref{eq:ch2-contra-on-av-k-step} in a compact form as follows \cite[Section 2.2]{Anthony2015}:
\begin{align}\label{eq:ch2-contra-on-av-k-step-comp-form}
\sup _{x \in \mathcal X} \mathbb{E}_{x}\left[\left\|w_k(\cdot)\right\|_{\mathrm{Lip}}\right] \leq r.
\end{align}
\end{definition}
In \cite{Hermer2019}, several other contraction conditions were introduced to achieve convergence of random iteration of functions. To introduce them, we first need to define the following:
\begin{align}
\text{Fix } w_i=\{x\in \mathcal{X}: w_i(x)=x\}
\end{align}
\begin{definition}
(Quasi-nonexpansive mappings.) $w_i$ is called a quasi-nonexpansive mapping  if 
\begin{align}
\text{for all }x \notin \text{Fix } w_i \text{and, } \text{for all } y\in \text{Fix } w_i \Rightarrow \rho (w_i(x), y)\le \rho(x,y).
\end{align}
\end{definition}
\begin{definition}
(Paracontraction.) $w_i$ is called paracontraction if it is continuous and 
\begin{align}
\text{for all }x \notin \text{Fix } w_i \text{ and, } \text{ for all } y\in \text{ Fix } w_i \Rightarrow \rho (w_i(x), y)< \rho(x,y).
\end{align}
\end{definition}
\begin{definition}
(Nonexpansive mappings.) $w_i$ is called nonexpansive if
\begin{align}
\text{for all }x, y\in \mathcal{X} \Rightarrow \rho (w_i(x), w_i(y))< \rho(x,y).
\end{align}
\end{definition}
For the following definition, the space $\mathcal{X}$ needs to be a normed-linear space. 
\begin{definition}
(Averaged mappings on a normed linear space.) A mapping $w_i: \mathcal{X}\to \mathcal{X}$ is called averaged mapping if there exists an $\alpha\in (0,1)$ such that 
\begin{align}
\text{for all }x, y\in \mathcal{X} \Rightarrow \left\| w_i(x)- w_i(y)\right\|^2 +\frac{1-\alpha}{\alpha}\left\| (x-w_i(x))- (y-w_i(y))\right\|^2\le \left\|x-y\|^2.\right.
\end{align}
\end{definition}

\subsection{Ergodic Properties of the Associated Markov Processes}\label{subsec:ergo-prop-mp2}
\label{subsec:ch2-ergo-prop-ifswsdpwfnm}
To ensure that the Markov process given by the IFS stated in the Definition~\ref{dfn:ch2-ifswsdpwfnm} possesses a unique invariant  measure, we will need to place some mild restrictions on the probability functions $\{p_i(x):\mathcal X\to [0,1]\}$. Throughout the article, we will assume that each map $w_i$ of the IFS is Lipschitz continuous and that the probabilities $p_i\left(x\right)$ satisfies the Dini continuity as stated in the Definition~\ref{dfn:ch2-dini-cont}. Research has focused heavily on IFS models with state-dependent probability (e.g., \cite{Stenflo2012(s), Stenflo2002} and references there).
For an IFS with state-dependent probabilities, the convergence of a chaotic game to an invariant measure was shown in \cite{Barnsley1988(2), Barnsley1988(2)erratum}. See also, \cite{Gwozdz2005, Andres2005} for extension of the \emph{Barnsley-Hutchinson} type results to an infinite number of self-maps. Let us state a significant result on ergodicity of IFS on compact metric space due to \cite{Barnsley1988(2)}: 
\begin{theorem}[\cite{Barnsley1988(2)}]
Suppose $(\mathcal X, \rho)$ be a compact metric space and that $\{w_i: \mathcal X\to \mathcal X\}_{i=1}^{N}$ be an IFS that contracts on average with corresponding Dini continuous probabilities $\{p_i: \mathcal X\to [0,1]\}_{i=1}^{m}$, then Then there exists a unique invariant Borel probability measure $\nu^{\star}$ for the IFS i.e. for all event $\mathcal A\in \mathcal B\left(\mathcal X\right)$ we have
\begin{align}
\int \nu\left(x, \mathcal A\right) \nu^{\star}\left(\mathrm{d}x\right)=\nu^{\star}\left(\mathcal A\right).
\end{align}
\end{theorem}

\subsection{State-Dependent IFS with Infinite Number of Maps}\label{subsec:SDIFS-Inf-SM}
\begin{definition}[\cite{Anthony2015}]\label{dfn:ch2-ifswsipwinm}
Let $\{w_i\}_{i\in \mathbb N}$ be a collection of Borel measurable transformations on $ \mathcal X$ along with associated continuous probability functions $\{p_i(x)\}_{i\in \mathbb N}$, $p_i:\mathcal X\to [0,1]$, where for all $x\in \mathcal X, \sum\limits_{i\in \mathbb N}p_i(x)=1$. Then similar to the Definition~\ref{dfn:ch2-ifswsdpwfnm} if the iteration process starts from a point $x_0$, a map $w_{\sigma_1}(x)$ is chosen with probability $p_{\sigma_1}(x_0)$ and if the process is repeated up to $k>1$ times, then a random orbit $(w_{\sigma_k}\circ w_{\sigma_{k-1}}\circ\cdots\circ w_{\sigma_1})(x_0)$ is obtained
\begin{align*}
\text{ with probability } p_{\sigma_k}(w_{\sigma_{k-1}}\circ w_{\sigma_{k-2}}\circ\cdots\circ w_{\sigma_1}x_0)\times \cdots\times p_{\sigma_2}(w_{\sigma_1}(x_0))\times  p_{\sigma_1}(x_0).     
\end{align*}
And for a given starting point $x_0\in \mathcal{X}$ and $\mathcal{A}\in \mathcal{B}\left(\mathcal{X}\right)$ the transition probability from $x_0$ to $\mathcal{A}$ is given by
\begin{align}\label{eq:ch2-transi-inf-sdifs}
\nu\left(x_0, \mathcal{A}\right)=\sum_{k\in \mathbb N} p_{\sigma_k}(x) \boldsymbol{1}_{\mathcal{A}}\left(w_{\sigma_k}(x_0)\right)
\end{align}
\end{definition}
Under some regularity assumptions on the transformations and the associated probability functions, \cite{Peigne1993} has shown the existence of a unique, invariant measure of the Markov chain arising from the IFS defined in the Definition~\ref{dfn:ch2-ifswsipwinm}.

\section{IFS on Polish Spaces}\label{sec:IFS-on-polish}

\subsection{State-Independent IFS with Infinite Number of Self Maps}
According to \cite{Dumitru2013}, IFS with an infinite number of self-maps was first introduced due to \cite{Wicks2006} in $1991$. See  \cite{Lesniak2004,Gwozdz2005, Dumitru2009, Mihail2009,Mihail2010(1),Mihail2010(2), Secelean2013, Mauldin1995, Hanus2000,Hyong2005, Mendivil1998(2)} for their theoretical development. Some applications of IFS with an infinite number of self-maps can be found in \cite{Diaconis1999} and the references there. Here, we try to summarise some of their basic properties:
\begin{definition}[\cite{Diaconis1999}](State independent IFS with infinite number of self maps.)\label{dfn:ch2-ifs-diaconis}
Let  $\mu$ be a \emph{probability distribution} on a countably infinite index set $\Theta$ and, $\{w_{\sigma}: \sigma \in\Theta \}$ be a family of self-maps on the metric space $\mathcal{X}$. One way to characterize a state-independent IFS that has a countably infinite number of self-maps is as follows: If the starting location is $X_0=x_0\in \mathcal{X}$, then the movement occurs by selecting a $\sigma$ at random from $\mu$, and then continuing to  $w_{\sigma}\left(x\right)$. $\mu$ is independent of any $x\in \mathcal{X}$. In the same way, as was covered in section \eqref{eq:ch2-ifswsindp}, this process will be described inductively by the realization equation.
\begin{align}\label{eq:ch2-ifs-diac-realiz}
X_{k+1}=w_{\sigma_{k}}\left(X_k\right)
\end{align}
where $\sigma_1,\sigma_2,\dots$ are drawn  independently from the $\mu$ distribution.
\end{definition}
Based on the above definition, the \emph{forward and the backward iteration} (see, \cite{Chamayou1991, Letac1986, Diaconis1999}) of the mappings are given by the compositions of iteration of the mappings as follows respectively: 
\begin{align}
X_{k}(\cdot)&:=\left(w_{\sigma_{k}} \circ w_{\sigma_{k-1}} \circ \cdots \circ w_{\sigma_{1}}\right)(\cdot), \quad X_{0}(\cdot)=X_0,\label{eq:ch2-for-diac}\\
Y_{k}(\cdot)&:=\left(w_{\sigma_{1}} \circ w_{\sigma_{2}} \circ \cdots \circ w_{\sigma_{k}}\right)(\cdot), \quad Y_{0}(\cdot)=Y_0.\label{eq:ch2-back-diac}
\end{align}
Both \eqref{eq:ch2-for-diac} and \eqref{eq:ch2-back-diac} have the same distribution for every $k$, but their characteristics are different. For any fixed $x$  \eqref{eq:ch2-for-diac} is ergodic with some regularity assumptions on the maps, while \eqref{eq:ch2-back-diac} converges almost surely and may not even have a Markov property \cite[Theorem 5.1]{Diaconis1999}. The following result is due to \cite{Letac1986}, which gives a simple way to prove that \eqref{eq:ch2-for-diac} is ergodic and obtains its \emph{stationary distribution}. Let $\mathcal D(\mathcal Y)$ denote the  distribution of a random variable $\mathcal Y$ defined below.
\begin{theorem}[\cite{Letac1986}]
If $\lim\limits_{k\to \infty} Y_k(x)=\mathcal Y< \infty$ almost surely and independent of $x$, then the Markov chain $\{X_n(x)\}_{n\ge 0}$ is ergodic with unique stationary distribution $\mathcal L(\mathcal Y)$.
\end{theorem}
The main theorem in \cite{Diaconis1999} and the immediate consequences were already established in several articles; for example, see \cite{Crauel1992, Dubins1966, Barnsley1988(2), Barnsley1988(2)erratum, Duflo2013, Elton1990, Hutchinson1981}.

The map $(\sigma, x)\mapsto f_{\sigma}(x)$ is measurable with respect to product sigma-field on $\mathcal X\times \Phi$. This representation is always viable because $\mathcal X$ is Polish, cf.  \cite{Kifer2012, Borovkov1992}. In \cite{Diaconis1999} it is assumed that if $w_{\sigma}$ are all \emph{Lipschitz}, so that there exists a $l_{\sigma}$ satisfying 
\begin{align}\label{eq:ch2-lip-ccontra-cond}
\rho(w_{\sigma}\left(x\right), w_{\sigma}\left(y\right))\le l_{\sigma}\rho\left(x,y\right)\; \forall x,y\in \mathcal X    
\end{align}  
and these are contracting on average in the sense that 
\begin{align}\label{eq:ch2-exp-log-lip-const-contra-on-avg}
\mathbb E_{\mu}\left[\ln\left(l_{\sigma}\right)\right]<0,    
\end{align}
and further if there exists some $x_0\in \mathcal X$ and a $\beta>0$ such that 
\begin{align}\label{eq:ch2-beta-cond}
\mathbb E_{\mu}\left[l_{\sigma}^{\beta}\right]< +\infty,    
\end{align}
and there exists $\alpha>0$
\begin{align}\label{eq:ch2-exp-lip-distant-condi}
\mathbb E_{\mu}\left[\rho^{\alpha}\left(w_{\sigma}\left(x_0\right), x_0\right)\right]< +\infty
\end{align}
then the Markov chain arising from the random iteration of functions whose realization equation is given by \eqref{eq:ch2-ifs-diac-realiz} has a unique invariant probability distribution $\nu^{\star}$. Moreover under any initial condition $X_0=x$, the distribution of $\nu_k\to \nu^{\star}$ in distribution in Prohorov metric where $X_k\sim \nu_k$. We now state the main result in \cite{Diaconis1999}:
\begin{theorem}[\cite{Diaconis1999}]
Let $(\mathcal X, \rho)$ be a complete, separable metric space. Let $\{w_{\sigma}: \sigma\in \Theta\}$ be a family of Lipschitz function on $\mathcal X\to \mathcal X$ and $\mu$ be a probability distribution on $\Theta$ and let $\nu_k(x, dy)$ denote the transition probability kernel of the chain $X_k$. Suppose further that conditions stated in \eqref{eq:ch2-exp-log-lip-const-contra-on-avg}, \eqref{eq:ch2-beta-cond} and \eqref{eq:ch2-exp-lip-distant-condi} holds, then 
\begin{itemize}
\item [(a)] the induced Markov chain has a unique invariant probability distribution $\pi$,
\item[(b)] there is a positive, finite constant $a_x$ and $r\in (0,1)$ such that
\begin{align*}
d_{\text{pr}}(\nu_k(x, \cdot), \pi)\le a_x r^k \quad \forall k=1,2,\dots, \text{ and } \forall x\in \mathcal X,   
\end{align*}
\item[(c)] the constant $r$ does not depend on $x$ or $k$; the constant $a_x$ does not depend on $k$ and 
\begin{align*}
a_x < a+b \rho(x, x_0) \text{ for some } 0<a,b < +\infty.
\end{align*}
\end{itemize}
\end{theorem}
\cite{Iosifescu2003} has shown the same exact conclusion using contraction characteristics of the two linear operators associated with the Markov chain.  \cite{Wu2000} and \cite{Jarner2001} exhibit these chains' additional stability features. See \cite{Abrams2003} for an extension of some results from \cite{Diaconis1999} with the maps with Lipschitz number one. 

In \cite{Jarner2001} the following assumptions have been made on the complete separable metric space $(\mathcal X, \rho)$:
\begin{assumption}\label{assump:ch2-jerner}
\begin{itemize}
\item [(a)] \emph{The metric is bounded and the diameter is finite}:
\begin{align}\label{eq:ch2-bdd-met-diam}
R=\rho(x,y)< \infty,  
\end{align}
\item[(b)] \emph{Maps are contracting on average}:
\begin{align}\label{eq:ch2-non-sep-avg}
\mathrm{E}\left[\rho\left(X_{1}(x), X_{1}(y)\right)\right]&\leq \rho(x, y)\quad \forall x,y \in \mathcal X,   \end{align}
\item[(c)] \emph{Local contraction}: Let $\tau_{C}(x)=\inf \left\{n \geq 1: X_{n}(x) \in C\right\}$. There exists $r\in (0,1)$ and a region $C$ in $\mathcal X$, satisfy the following:
\begin{align}\label{eq:ch2-local-contra}
\mathrm{E}\left[\rho\left(X_{\tau_{C}(x) \vee \tau_{C}(y)}(x), X_{\tau_{C}(x) \vee \tau_{C}(y)}(y)\right)\right] &\leq r \rho(x, y) \quad \forall x,y \in \mathcal X,
\end{align}
\item[(d)] \emph{Drift condition}: There exists $H: \mathcal X\to \mathbb R$ with $H\ge 1$, $H$ is bounded in the region $C$ with 
\begin{align}\label{eq:ch2-drift-condi}
D&=\sup\limits_{x\in C}\{H(x)\}<+\infty,    
\end{align}
and there exists constant $\lambda \in (0,1)$ and $b< +\infty$ such that
\begin{align}\label{eq:ch2-lyapu-condi}
\nu H(x)=\int H(x) \nu(x, dy)\le \lambda H(x)+ b 1_{C}(x),
\end{align}
where $\nu_k(x, \cdot)$ is the distribution of $X_k(x)$.
\end{itemize}
\end{assumption}
Now we state the main results from  \cite{Jarner2000}:
\begin{theorem}[\cite{Jarner2000}]
Under the assumptions made in Assumption~\ref{assump:ch2-jerner}, the IFS \eqref{eq:ch2-ifs-diac-realiz} possesses an unique invariant measure $\pi$ and regardless of any initial distribution
\begin{align}
X_k\to \pi \quad \text{ in distribution }, 
\end{align}
further there exists  $\gamma<1$,  $\kappa< \infty$ and for all $x\in \mathcal X$, $n\ge 1$ we have,
\begin{align}
d_{\text{pr}}\left(\nu_{k}(x, \cdot), \pi\right) \leq \gamma^{k} \kappa \left[H(x)\right]^{\frac{1}{2}}.
\end{align}
\end{theorem}
It is clear from the above result that the chain exhibit stability properties as those found in \cite{Diaconis1999}. Further, under the following assumptions:
\begin{assumption}\label{assump:ch2-wu-shao}
\begin{itemize}
\item [(a)] There exists $x_0\in \mathcal X$ and a $\alpha\in (0,1]$ such that
\begin{align}\label{eq:ch2-wu-shao-cond1}
\mathbb E\left(\rho^{\alpha}\left(x_{0}, X_{1}\left(x_{0}\right)\right)\right)=\int_{\Theta} \rho^{\alpha}\left(x_{0}, X_{k}\left(x_{0}\right)\right) \mu(\mathrm{d} \theta)<\infty
\end{align}
\item[(b)] with $x_0$ and $\alpha$ chosen above, further, there exists $r(\alpha)=r\in (0,1)$ and $c(\alpha)=c>0$ such that 
\begin{align}\label{eq:ch2-wu-shao-cond2}
\mathbb E\left(\rho^{\alpha}\left(X_{k}(x_0), X_{k}\left(x\right)\right)\right) \leq c r^{k} \rho^{\alpha}\left(x_0, x\right) \quad \forall x_0,x\in \mathcal X \text{ and } \forall k\in \mathbb N.
\end{align}
\end{itemize}
\end{assumption}
\cite{WuShao2004} has showed that \eqref{eq:ch2-back-diac} converges almost surely as $n\to \infty$.

If $X\subseteq \mathbb R^n$ is a convex subset, then \cite{Steinsaltz1999}  has added a local contractibility requirement that essentially amounts to the presence of a drift function $g: \mathcal X \to (0,\infty)$ and a $r\in (0,1)$  such that
\begin{align}\label{eq:ch2-steinsalz}
\mathbb E\left(\limsup _{y \rightarrow x} \frac{\rho\left(X_{k}(y), X_{k}(x)\right)}{\rho(y, x)}\right) \leq g(x) r^{k} \quad \forall x \in \mathcal X \text{ and } \forall k\in \mathbb N.
\end{align}
Further to it, he  proved that under \eqref{eq:ch2-wu-shao-cond1} in conjunction with a suitable assumption on $g$, which is a generalization of \eqref{eq:ch2-wu-shao-cond1} a more general contraction condition than \eqref{eq:ch2-wu-shao-cond2} does hold, but not \eqref{eq:ch2-wu-shao-cond2} itself. Unifying results from \cite{Steinsaltz1999} and \cite{WuShao2004},   a more general contraction condition was introduced in \cite{Herkenrath2007} as follows:
\begin{itemize}
\item [(a)] There exists $x_0\in \mathcal X, \alpha \in (0,1], r\in (0,1)$, and measurable functions $g_k: \mathcal X\to [0,\infty), k\in \mathbb N$, such that both
\begin{align}\label{eq:ch2-gen-conds-H&I}
&r_1:= \frac{1}{ \limsup\limits _{k \rightarrow \infty} \delta_{k}^{\frac{1}{k}}} \text{ and } r_{2}(x):=\frac{1}{ \limsup\limits_{k \rightarrow \infty} g_{k}^{\frac{1}{k}}(x)} \text{ strictly exceed } r\quad \forall x\in \mathcal X, \text{ where }\\
&\delta_{k}:=\mathbb E\left(g_{k}\left(X_{1}\left(x_{0}\right)\right) \rho^{\alpha}\left(x_{0}, X_{1}\left(x_{0}\right)\right)\right), \quad n \in \mathbb{N}.\nonumber
\end{align}
\item[(b)] For all $x\in \mathcal X$ and for any $n\in \mathbb N$, one has,
\begin{align}\label{gencond2}
\mathbb E\left(\rho^{\alpha}\left(X_{k}(x), X_{k}\left(x_{0}\right)\right)\right) \leq r^{k} g_{k}(x) \rho^{\alpha}\left(x, x_{0}\right)
\end{align}
\end{itemize}
Clearly, $r_1$ and $r_2(x)$ are radius of convergence of the power series $\sum\limits_{n\in \mathbb N}\delta_kt^k$ and  $\sum\limits_{k\in \mathbb N}g_k(x)t^k$, $t\in \mathbb R$, respectively.

\section{Ergodic Theorems and Law of Large Number for IFS}\label{sec: IFS-ergo-theo-SLLN}
The principal objective of this part is to provide an overview of ergodic theory as a method for investigating statistical patterns in IFS. First, we attempt to use an illustrative example from a deterministic dynamical system, and then we try to use a random example. Finally, we try to understand how the patterns rely on the system's initial state. Assume we have modelled a highly complex phenomenon using a Markov chain formed by an IFS, $\{X_k\}$, developing in the state space $\mathcal X$. Thus, the system's state at time instant $k$ is $X_k\in \mathcal X$. The description of the system at any point in time may include many parameters, each of which must be represented by an element of $\mathcal X$. Experimental investigation of such systems may be arduous since the system's behaviour may be so complex that following all the specifics of its development is relatively tricky or implausible. Most of the time, it is not easy to speculate on the trajectory such a system will take. This may be the case sometimes because of computation difficulties and the system's inherent unpredictability and sensitivity to initial inputs.

A statistical technique is one approach that may be used when attempting to characterize the characteristics of a system. In most cases, the first step in conducting an experimental study of a system is to gather statistical data about the system by making several observations or taking multiple measurements. A measurement is a technique to associate any state $x\in \mathcal{ X}$ of the system in consideration with a real number $w(x)\in \mathbb R$, i.e., it is a map $w:\mathcal X \to \mathbb R$. A measurement $w \left(X_k\right)$ of the system $\{X_k\}_{k\ge 0}$ at time $k$ yields the value $w(X_k)$. Because a single measurement like this discloses nothing about the system, it makes sense to do the measurement many times and then take the average of the results obtained over different periods. The outcome of this technique if the measurements are carried out at periods $1,2,\dots,k$ is
\begin{align}\label{eq:ch2-time-average-1}
\frac{w(X_1)+\ldots+w(X_n)}{k}.    
\end{align}
This \emph{time-averaged data} may be used as approximations to a desired \emph{true value}. Additionally, one hopes that when additional observations are averaged, the resultant estimate approaches the ideal real value. In other words, gathering increasing amounts of data is only beneficial if a variant of the law of large numbers is true, i.e., there is a $\bar x\in \mathbb R$ 
\begin{align}\label{eq:ch2-time-average-2}
\lim\limits_{k\to \infty} \frac{w\left(X_1\right)+\ldots+w\left(X_n\right)}{k}\to \bar x.    
\end{align}
Additionally, the infinite time horizon average should be independent of the system's initial state, or the limiting value should be challenging to grasp. Consequently, the following logical problems arise: Is there any statistical regularity in the system? How does the initial condition affect the system's statistical properties? Is the system capable of remembering its initial condition, or does it ultimately forget? Next, we try to address these questions in the context of IFS.

\subsection{A Motivating Example from Deterministic Dynamics}\label{subsec:moti-dds}
On $\mathbb R$, a basic discrete linear dynamical system (without randomness) may be described as the following transformation map for some constant $a\in (0,1)$:
\begin{align}\label{eq:ch2-lin-det-dyna}
w_1: \mathbb R\to \mathbb R, w_1(x)=ax
\end{align}
If we apply the map $k$ times, we get $X_k(x)=a X_{k-1}=w_1^k(x)=a^kx\to 0$ as $k\to \infty$, indicating that $0$ is a stable, attractive fixed point of the dynamics regardless of where we begin. The convergence toward this single-point attractor is exponential. Its attractive domain is identical to $\mathbb R$. Thus, the system rapidly loses memory because of the map's contraction and intrinsic stability. Regardless of the original state, the sequence of points $X_k\to 0$ as $k\to \infty$, i.e., the system's initial condition is eventually forgotten. On the other hand, one may illustrate a non-stable (see \cite{Bakhtin2015}) system with the state space $\mathcal X=[0,1)$ and the transformation
\begin{align}
w_2: [0,1)\to [0,1),\quad w_2(x)=\{2x\} \text{ i.e fractional part of the real number } 2x,\; x\in [0,1).    
\end{align}  
As a result, we can show that stability at the trajectory level does not hold in the classical fashion by demonstrating it using these two above cases. However, a definition of stability more focused on system statistics might still be legitimate. To put it another way, if $f$ is an integrable and Borel measurable function on $[0,1)$, then it is true that for any $x\in [0,1)$ with respect to the \emph{Lebesgue measure} on $[0,1]$,
\begin{align}\label{eq:ch2-erid}
\lim\limits_{k \rightarrow \infty} \frac{1}{k}\sum_{i=1}^{k} f\left(w_2^{i-1}(x)\right)=\int_{[0,1)} f(x) dx. \end{align}
Given that the right-hand side reflects the average of $f$ with respect to the Lebesgue measure on $[0,1)$, one may conclude that the \emph{Lebesgue measure} accurately characterises the attributes of this dynamic system or uniform distribution on $[0,1)$ in the long run.

Indeed, one interpretation of \eqref{eq:ch2-erid} is that for practically every initial condition $x$, the distribution of $\phi(x)$ with assigning equal mass of $\frac{1}{k}$ unit to each of $x, w_2(x),\ldots, w_2^{k-1}(x)$ converges to the \emph{Lebesgue measure}. Additionally, the convergence in \eqref{eq:ch2-erid} might be seen as a law of large numbers and more precisely \emph{Lebesgue measure} is unique in this system since it defines the limits in \eqref{eq:ch2-erid} and invariant under $w_2$. It is ergodic; it cannot be reduced to two nontrivial invariant measures.

\subsection{A Motivating Example from Random Dynamics}\label{subsec:moti-rds}
Similar concerns, as stated just before, arise naturally when considering random maps rather than deterministic maps. In general, the long-term behaviour of chaotic deterministic dynamical system trajectories is uncertain \cite[Chapter 2]{Boyarsky2012}. As a result, it is appropriate to use statistical methods to characterize the system's behaviour in aggregate. Random maps may also form naturally due to random perturbations of deterministic dynamics. 

For example, recall that the previous dynamics stated in \eqref{eq:ch2-lin-det-dyna} where the map was $w_1: x\mapsto ax$ for which $0$ was a stable equilibrium point. Let us introduce random noise into this system to destroy the equilibrium. Assume we have a sequence of independent, identically normally distributed random variables $\{\sigma_k(\omega)\}_{k\in \mathbb Z}\sim N(0, v^2)$ specified on $(\mathcal{X}, \mathcal E, \mathbb P)$ with $\omega$ as usual represent a sample point. Define random maps for $k\in\mathbb Z$ as follows: 
\begin{align}\label{eq:ch2-rd}
w_{k,\omega}(x)=ax+\sigma_k(\omega)    
\end{align}
A natural analogue of the forward orbit from the example \eqref{eq:ch2-lin-det-dyna} is a
a stochastic process $\{X_k\}_{k\ge 0}$ with starting from $X_0=x_0\in \mathbb R$ such that
\begin{align}\label{eq:ch2-stoch-proces}
X_{k+1}=aX_k+\sigma_k
\end{align}
Stability analysis, in this case, is not straightforward as it was for \eqref{eq:ch2-lin-det-dyna}. No fixed equilibrium point serves all maps $\{w_{k,\omega}\}_{k\in\mathbb Z}$ at once. It is possible that the solution of $w_{k,\omega}(x)=x$ for some values of $k$ doesn't make any sense to all other different values of $k\in\mathbb Z$. Nevertheless, this system can produce an ergodic result analogous to \eqref{eq:ch2-erid}. Let $\mu\sim N(0, \frac{\sigma^2}{1-a^2})$. Then for any integrable function $f$ with respect to $\mu$, for almost every $(x_0, \omega)\in \mathbb R \times \Omega$ with respect to $\mu\times \mathbb P$
\begin{align}\label{eq:ch2-ergo-like-prop}
\lim _{k \rightarrow \infty} \frac{1}{k} \sum_{i=1}^{k} f\left(X_{i-1}\right)=\int_{\mathbb{R}} f(x) \nu(d x)
\end{align}
This finding explains that $\mu$ is a unique invariant measure for the Markov operator associated with the system. This example illustrates a scenario without deterministic stability, yet statistical stability exists.

\subsection{Ergodic Theorems for IFS on Compact Space}\label{subsec:ergo-thm-cpts}
For dynamical systems, Birkhoff's Ergodic Theorem \cite[Chapter 5, Theorem 5.1.1]{Silva2008} is well-known: Let
\begin{align*}
T: (\mathcal X, \mathcal B(\mathcal X), \nu)\to (\mathcal X, \mathcal B(\mathcal X), \nu)    
\end{align*}
be an \emph{ergodic dynamical system}, further, suppose $w\in L^1(\mathcal X, \mathcal B(\mathcal X), \nu)$, then for almost every $x$
\begin{align}
\lim _{k \rightarrow \infty} \frac{1}{k} \sum_{i=0}^{k-1} w\left(T^{i}(x)\right)=\int w \mathrm{d} \nu.
\end{align}
The \emph{point-wise Ergodic theorem or strong law of large number} for IFS is due to \cite{Elton1987}. The result was later sharpened in \cite[Proposition 7.1]{Peigne1993} under some hypothesis. 
\begin{theorem}[\cite{Elton1987}]
Consider an IFS described in Definition~\ref{dfn:ch2-ifswsdpwfnm} with associated probabilities $p_i(x):\mathcal X\to [0,1]$ that satisfy the mild technical condition in Definition~\ref{dfn:ch2-dini-cont}. Let $\nu^{\star}$ denote the unique invariant probability measure of the IFS. Then for almost all $x\in \mathcal X$ and for all $w\in C(\mathcal X,\mathbb R)$
\begin{align}
\lim _{k \rightarrow \infty} \frac{1}{k+1} \sum_{i=0}^{k} w\left(\left(w_{\sigma_{i}}\circ w_{\sigma_{i-1}}\circ \cdots\circ w_{\sigma_{0}}\right)(x)\right)=\int w \mathrm{d} \nu.
\end{align}
\end{theorem}
See also \cite{Mendivil1998(1)}, and \cite{Tricot1999} for short proofs of the above result. All these results hold for the state space of IFS as a closed subset of $\mathbb R^n$.

\subsection{Ergodic Theorems for IFS  on Polish Space}\label{subsec:ergo-thm-pols}
In \cite{Stenflo2001(3)} an weak ergodic theorem including
rate of convergence for Markov chains arising from IFS with a countably infinite number of functions under a \emph{stochastically contractive and boundedness condition}\cite[Condition C \& D, Theorem 2]{Stenflo2001(3)} and the method of proving the results were inspired from work introduced in \cite{Letac1986, Burton1995, Loskot1995, Silvestrov1998, Janoska1995}. Recently, in \cite[Section 2.4,3.4,3.7]{Hermer2022} under some greater regularity assumptions on non-expansive transformations (such as $\alpha$-firmly non-expansive), the ergodic and convergence aspects of random function iteration have been extensively explored. In \cite[section 2]{Joanna2008} the concept of \emph{generalized contraction of a transformation map} was introduced. Subsequently, the asymptotic behaviour of IFS with a finite or countably infinite number of maps with state-dependent probabilities consisting of at least one \emph{generalized contraction} was studied to prove their asymptotic stability property. For more on ergodic properties of IFS on Polish space we refer \cite{Lasota1998, Szarek1997, Szarek2000(1), Szarek2000(2), Szarek2003(1), Szarek2003(2)}.

\section{Central Limit Theorem (CLT) and Invariance Principle for IFS}\label{sec:CLT-IFS-invp}
Limit theory for stochastic dynamical systems arising from the IFS defined as in the Definition~\ref{df:ch2-ifswsindp} see, e.g in \cite{Borovkov1998, Bhattacharya2007, Meyn1993}. Suppose an IFS of the form \ref{df:ch2-ifswsindp} have a stationary distribution $\nu^{\star}$, let $w\in \mathcal L_0^2(\nu^{\star})$ be square integrable function on $\mathcal X$ with mean zero  and let
\begin{align}\label{eq:ch2-clt}
\mathcal S_k(w)=\sum\limits_{i=1}^{k} w(X_i),
\end{align}
then \cite{Wu2000} studied central limit theorems for IFS i.e, under what conditions on $w$ and $w_{\sigma_k}$, is $\frac{\mathcal S_k w}{\sqrt k}$ asymptotically normal as $k\rightarrow \infty$. By decomposing \eqref{eq:ch2-clt} into a martingale and a stochastically bounded component as in \cite{Gordin1978}, \cite{Benda1998} proved, that the above result is true if \eqref{eq:ch2-exp-lip-distant-condi} holds for $\alpha=2$, $w$ is Lipschitz continuous and $\mathbb E_{\mu}\left[l^2_{\sigma}\right]<1$. Even for a discontinuous $w$ and more relaxed conditions on the $l_{\sigma}$ and \eqref{eq:ch2-exp-lip-distant-condi}, \cite{Wu2000, WuShao2004} showed that the central limit theorems hold. If $w_1,\dots, w_m$ are Lipschitz self-maps on $\mathcal X$ with $l_1,\dots, l_m$ Lipschitz constants respectively, then \cite{Loskot1995} has derived central limit theorem for the random variables $w(X_k)$ for any Lipschitz function $w: \mathcal X\to \mathbb R$ for the IFS defined as in Definition \ref{df:ch2-ifswsindp} with the additional condition on the Lipschitz constants, and the probabilities have been taken as follows:
\begin{align}
\sum\limits_{i=1}^{m}\mu(\sigma_i) l^2_{\sigma_i}<1.   
\end{align}
For the IFS in Definition \ref{subsec:ch2-dfn-ifswsdpwfnm},
from \cite{Barnsley1985, Hutchinson1981, Elton1992} one can conclude that there exists an unique compact set $\mathcal K$ that is invariant for the IFS in the sense that
\begin{align*}
\mathcal K=\bigcup\limits_{i=1}^{N} w_i(\mathcal K).    
\end{align*}
Such chains allow greater flexibility for generating a \emph{uniform stationary probability distribution} on $
\mathcal K$. These kinds of Markov chains are seen in the thermodynamic formalism of statistical mechanics.

Existence and uniqueness of possessing a unique stationary distribution are not usual for these chains \cite{Berger2018, Bramson1993, Lacroix2000, Stenflo2001(2), Stenflo2002}, but some additional regularity conditions are needed, see, e.g., \cite{Harris1955, Johansson2003, Stenflo2002, Stenflo2003}. In \cite{Peigne1993}, a CLT has been proved by the spectral decomposition technique of the associated Markov operator. For details on the spectral decomposition of the Markov, operator see, e.g., \cite{Hennion1993, Ionescu1948, Ionescu1959}. On a non-compact metric space, a CLT for IFS with Lipschitz maps has been proved in \cite{Hennion2004}; for more literature on these, we refer \cite{Andreas2005, Herkenrath2007, Horbacz2016(2), Haggstrom2007, Haggstrom2005}.

\subsection{CLT for IFS with State-Independent Probabilities with Countably Infinite Number of Self Maps on Polish Space}\label{subsec: CLT-IFSWSIDP}
Let us formulate some background to demonstrate the law of large numbers for the chain established in \cite{Kapica2018}. Let $P$ be a Markov operator acting on the space of measures on $\mathcal X$ such that $P\nu_n=\nu_{n+1}$ i.e $P(\text{ law of } X_n)=(\text{ law of } X_{n+1})$, where $\nu_n$ is the distribution of $X_n$. The following assumptions were made to establish LLN for time-homogeneous Markov chain $(X_n)_{n\ge 1}$.
\begin{assumption}
\begin{itemize}
\item [(a)] There exists a probability measure $\nu_{\star}\in \mathcal M_p(\mathcal X)$, a continuous function $g: \mathcal X\to \mathbb R_{\ge 0}=[0,+\infty)$, and a $q\in (0,1)$ such that
\begin{align}\label{asmp1}
\left\|P^{n} \delta_{x}-\nu_{\star}\right\|_{F M} \leqslant g(x) q^{n}.
\end{align}
\end{itemize}
\end{assumption}
\begin{theorem}[\cite{Kapica2018}]
Let $P: \mathcal M(\mathcal X)\to \mathcal M(\mathcal X)$ be a Markov operator which satisfies \eqref{asmp1}. Let $X_n$ be a time-homogeneous Markov chain taking values in $\mathcal X$, with transition operator $P$ and initial distribution $\nu$ i.e law of $X_0=\nu$. Further  to this, there exists a $\nu_{\star}\in \mathcal M(\mathcal X)$ such that $\lim\limits_{n\to \infty} P^n\nu= \nu_{\star}$ for every starting distribution $\nu$.  Then for every Lipschitz and bounded function $f: \mathcal X\to \mathbb R$ the sequence $f(X_n)$ satisfies LLN i.e
\begin{align}
\frac{1}{n} \sum_{k=0}^{n-1} f\left(X_{k}\right) \underset{n \rightarrow \infty}{\longrightarrow} \int_{X} f(x) \nu_{*}(\mathrm{d} x). 
\end{align}
\end{theorem}
In addition, with the \cite[Assumption \bf{B1}-\bf{B4}]{Kapica2016},
\begin{theorem}[\cite{Kapica2018}]
If $X_n$ be a homogeneous Markov chain on a polish space $(\mathcal X, \rho)$ with transition operator $P$ and initial distribution $\mathcal L(X_0)=\nu$, and let $g: \mathcal X\to \mathbb R$ be any bounded, Lipschitz function then
\begin{align}
\mathbb{E} g\left(X_{n}\right) \underset{n \rightarrow \infty}{\longrightarrow} \int_{\mathcal X} g(x) \nu_{\star}(\mathrm{d} x),
\end{align}
where $\nu_{\star}$ is invariant probability distribution for $P$ and moreover the sequence $g(X_n)$ satisfies LLN i.e
\begin{align}
\frac{1}{n} \sum_{k=0}^{n-1} g\left(X_{k}\right) \underset{n \rightarrow \infty}{\longrightarrow} \int_{\mathcal X} g(x) \nu_{\star}(\mathrm{d} x) \quad \text { in probability } \mathbb{P}.
\end{align}
\end{theorem}
\cite{Horbacz2016(1)} demonstrates strong law of large number for random dynamical systems with randomly chosen jumps has been established, and they have shown their result holds for the IFS with state or place-dependent probabilities with a finite number of self-maps on a Polish space. For more results on the central limit on the Polish space theorem, see \cite{Kaijser1979, Kwiecinska2000, Marta1997}.

\subsection{Invariance Principle of IFS}\label{subsec:ifs-inv-prn}
Dynamical systems and probability theory are concerned with the limiting behaviour of successive observations of random variables. CLT examines deviations from the average behaviour of sums, for example, whereas the strong law of large numbers(SLLN) (ergodic theorem) specifies the average behaviour of sums. One extension of the previous two forms of the average behaviour of a sequence of random variables is the well-known notion of  \emph{almost sure invariance principle:}
\begin{definition}[\cite{Strassen1964, Strassen1967}]
Let $\{X_i\}_{i\ge 0}$ be  i.i.d random variables with finite $(2+\delta)$ moments, define $\mathcal S_{n}=\sum\limits_{i=0}^{n-1}X_i$. Then there exists a Brownian motion and a probability space $\Omega$ on which $B$ and $\mathcal S_n$ can be re-defined such that
$\mathcal S_n= B(n)+o(n^{\frac{1}{2}})$.
\end{definition}
This instantly enables one to infer a variety of standard statistical features for $\mathcal S_n$, such as the SLLN and different modifications of the CLT, assuming that they hold for the Brownian motion $B$. One such result is the following:
\begin{theorem}[\cite{Walkden2007}, Almost sure invariance principle]
Suppose $f: \mathcal X\to \mathbb R$ is a bounded Lipschitz function with $\nu(f)=0$ and $\sigma^2(f)=\sigma^2>0$. Fix $x\in \mathcal X$, then there exists a probability space $(\Omega, \mathcal F, \mathbb P)$ and a one-dimensional Brownian motion
$W: \Omega\to C(\mathbb R^{+}, \mathbb R)$ such that the random variable $\omega\mapsto W(\omega)(t)$ has variance $\sigma^2t$ and sequence of random variables $\phi_{n,x}: \Sigma\to \mathbb R, \psi_{n,x}: \Omega\to \mathbb R$ with the following properties:
\begin{itemize}
\item [(a)] for all $\epsilon>0$ we have,
\begin{align}
f^{n}(x, \cdot)=\phi_{n, x}(\cdot)+O\left(n^{\frac{1}{4}+\varepsilon}\right)\quad  \mu_x \text{ a.e };
\end{align}
\item[(b)] the sequence $\{\phi_{n,x}(\cdot)\}$ and $\{\psi_{n,x}(\cdot)\}$ are equal in distribution.
\item[(c)] for all $\epsilon>0$, we have,
\begin{align}
\psi_{n, x}(\cdot)=W(\cdot)(n)+O\left(n^{\frac{1}{4}+\varepsilon}\right) \quad \mathbb P \text{ a.e}.
\end{align}
\end{itemize}
\end{theorem}
where by a.e we mean \emph{almost everywhere}.

\section{Geometry of the Unique Invariant Measure}\label{sec:geometry}

In the penultimate section of the survey, we consider structural results, which can be seen as the geometry of the unique invariant measure. While there is surely much more to be understood, the current results provide valuable insights within IFS.  

\begin{figure}[ht!]
\centering
\includegraphics[width=0.45\textwidth]{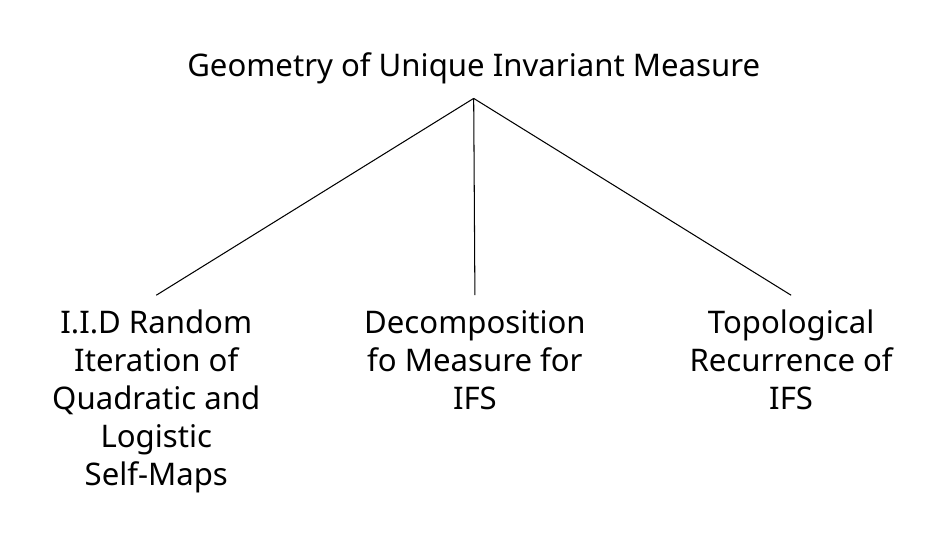}
\caption{Section \ref{sec:geometry} of the survey focuses on structural results, which can be seen as the geometry of the unique invariant measure.}
\label{fig:des-2}
\end{figure}

\subsection{A Special Case: I.I.D Random Iteration of Quadratic and Logistics Maps}\label{sec:IID-random-quad}
For a parameter $\sigma\in [0,4]$, consider the quadratic family of self-maps on $[0,1]$ as follows:
\begin{align}\label{dfn:ch2-logis}
\mathcal {O}=\{ w_{\sigma}(x):=\sigma x(1-x), 0\le \sigma\le 4, 0\le x\le 1\}.
\end{align}
And the realization scheme is as follows:
\begin{align}\label{eq:ch2-logistic-scheme}
X_{k+1}=w_{\sigma_k}\left(X_k\right), \quad k=0, 1,2,3,\dots; \sigma_1,\sigma_2,\dots \text{ are i.i.d} \sim [0,4];\quad \forall k\ge 1,  w_{\sigma_k} \in \mathcal O.
\end{align}
Given a pair of parameter values $(\alpha, \beta)$ with  $\alpha< \beta$, and a number $r\in (0,1)$, random dynamical systems generated by considering i.i.d sequence of maps $\{w_{k}: k\ge 1, k\in \mathbb N\}$ with $\mathbb P\left[w_1=w_{\alpha}\right]=r$ and $\mathbb P\left[w_1=w_{\beta}\right]=1-r$ from the set $\mathcal{O}$ and similar ones has been studied extensively in last seventy years, see \cite{Collet2009, Devaney2008}. This dynamical system family is one of the most influential, despite its apparent simplicity. For results on a few other chaotic families, see, e.g., \cite{May1976, Melo2012, Day1994}. In the economics literature of dynamic optimization model \eqref{dfn:ch2-logis} and \eqref{eq:ch2-logistic-scheme} arise significantly, see, e.g. \cite{Majumdar1994, Majumdar2000}.   In \cite{Bhattacharya1993} for specific values of $\alpha,  \beta$, uniqueness and other properties of the Markov chain arising from such random iteration have been studied. Several interesting results on the uniqueness and support of the invariant measure were obtained, primarily motivated by \cite{Dubins1966} on random monotone maps on an interval.
\begin{theorem}[\cite{Bhattacharya1993}]
\begin{itemize}
\item [(a)] If $0\le \alpha< \beta \le 1$, then $\delta_0$ is the unique invariant probability measure in $[0,1]$.
\item [(b)] If $1< \alpha< \beta \le 2$, then there exists an unique invariant non-atomic probability measure $\pi$ on $(0,1)$.
\item [(c)] If $2< \alpha< \beta \le 1+ \sqrt{5}$ and $\alpha\in [\frac{8}{\beta(4-\beta)},\beta)$ then there exists an unique invariant non-atomic probability measure $\pi$ on $(0,1)$.
\end{itemize}
\end{theorem}
Much is known about the behaviour of the dynamical system generated by the self-maps for a fixed $\sigma\in [0,4]$, defined in the set \eqref{dfn:ch2-logis} which are also known as logistics maps in the literature of dynamical systems and chaos, see \cite{Graczyk1997, Arrowsmith1990, Klebaner1997}. 

In \cite{Athreya2000}, the following recursive relation 
\begin{align}\label{eq:ch2-logistic-ext}
X_{k+1}= \sigma_{k+1} X_k (1-X_k) \quad k=0, 1,2,3,\dots.  
\end{align}
has been studied with great depth and importance, and necessary and sufficient conditions have been established for the existence and uniqueness of non-trivial invariant probability measure for \eqref{eq:ch2-logistic-ext}, in \cite{Steinsaltz2001, Athreya2001, Athreya2003(2), Bhattacharya1999, Bhattacharya2001(1), Dai2000}.

Consider a slightly different version of the chain \eqref{eq:ch2-logistic-ext} as follows:
\begin{align}\label{eq:ch2-logistic-variant}
X_{k+1}= \sigma_{k+1} X_k w_{k+1}(X_k) \quad k=0, 1,2,3,\dots,  
\end{align}
under some assumptions on $\sigma_k$ and $w_k$ for $k=0,1,2,\dots$, \cite[Section 4]{Athreya2003(1)}  investigate conditions for the existence and uniqueness of non-trivial invariant probability measures and their support for \eqref{eq:ch2-logistic-variant}. This class of maps has been widely utilized in the research on ecology and economics to simulate population and growth models. See \cite{Majumdar2009} for additional expansion and related results on the random iteration of monotone maps. In \cite{Bhattacharya2001(2)}, a convergence rate in  Kolmogorov distance of the chain defined in \eqref{dfn:ch2-ifs-diaconis} has shown under some mild assumption on the chain. The utility of these specific models of IFS has been extensively documented in the literature of dynamic economics see, e.g., \cite{Bhattacharya2009, Tong90}, and economic growth and survival under uncertainty, see \cite{Majumdar1989, Majumdar1992, Stokey1989, Tong90, Yahav1975}. In \cite{Bhattacharya2004}, positive Harris recurrence and the existence and uniqueness of invariant probability measure have been established for the chain \eqref{eq:ch2-logistic-scheme}. A significant fact to notice that the state-space of $\{X_k\}_{k\ge 0}$ is $(0,1)$, since the domain of the logistic maps was restricted in $(0,1)$ to avoid the trivial invariant probability measure $\delta_{\{0\}}$. For the motivation and significance of this important Markov model, with applications to problems of \emph{optimization under uncertainty} arising in economics, see \cite{Majumdar2000, Majumdar1994, Bala1992, Mitra1998, Anantharam1997}.

\subsection{Decomposition of Measure for IFS}\label{sec:IFS-decom-meas}
\begin{definition}
A probability measure $\nu_1$ is \emph{absolutely continuous with respect to another
probability measure} $\nu_2$ if
\begin{align}
\nu_2(\mathcal A) = 0 \Rightarrow \nu_1(\mathcal A) = 0 \quad \text{ for all } \mathcal A\in \mathcal B(\mathcal X)    
\end{align}
i.e. if every null set of $\nu_2$ is also a null set of $\nu_1$.
Two probability measures are equivalent if they are mutually absolutely continuous or have the same null sets.
\end{definition}
The following theorem is adapted from \cite{Dubins1966}, and it is noted in \cite{Barnsley1988(2), Barnsley1988(2)erratum} that it holds for a large class of IFS. Consider $\mathcal{M}_{f}$ be the set of all finite non-negative Borel measures on $\mathcal X$, if any $\nu\in \mathcal M_f$ can be uniquely decomposed into sum of two measures $\nu_1$ and $\nu_2$, i.e, $\nu=\nu_1+\nu_2$ then 
we express $\mathcal M_{f}= \mathcal M_{f}^1 \oplus \mathcal M_{f}^2$.
\begin{theorem}\label{thm:ch2-ext-dubin-freed-on-polish}
(Extension of Dubins and Freedman's result to Polish space)
Let $\mathcal M_{f}$ denote the set of all of the non-negative finite measures on some polish space $\mathcal X$, let $T$ be a map on  $\mathcal M_{f}= \mathcal M_{f}^1 \oplus \mathcal M_{f}^2$  which satisfies the following:
\begin{itemize}
\item[a)]$T(\nu_1+\nu_2)= T\nu_1+ T\nu_2 \quad \forall \nu_1, \nu_2\in \mathcal M_f$.
\item[b)] $T\mathcal M_f^1\subseteq \mathcal M_f^1$.
\item[c)] $(T\nu)\left(\mathcal X\right)=\nu\left(\mathcal X\right)\quad \forall \nu\in \mathcal M_f$.
\end{itemize}
Then if $\nu$ is a fixed point of $T$ i.e $T\nu=\nu$ with the fact that $\nu=\nu_1+\nu_2$ for some unique $\nu_1\in \mathcal M_f^1$ $\nu_2\in \mathcal M_f^2$ then $\nu_1$ and $\nu_2$ are also fixed point of $T$ i.e $T\nu_1=\nu_1$ and $T\nu_2=\nu_2$. 
\end{theorem}
\begin{proof}
Fix a $\nu\in \mathcal M_f$ such that $T\nu=\nu$, one can write  $\nu=\nu_1+\nu_2$ and since $T$ is linear as a map, we have by condition $a)$ of the statement
\begin{align}\label{eq:ch2-from cond-a}
T\nu= T(\nu_1+\nu_2)= T\nu_1+ T\nu_2. 
\end{align}
Now, we use the condition $b)$ of the statement, and we get
\begin{align}\label{eq:ch2-from-cond-b}   
T\nu_1\in \mathcal M_f^1. 
\end{align} 
Now, for $T\nu_2$ we get unique $\mu_1\in \mathcal M_f^1, \mu_2\in \mathcal M_f^2$ such that
\begin{align}\label{eq:ch2-from-cond-uniq}
T\nu_2=\mu_1+\mu_2 
\end{align}
So \eqref{eq:ch2-from cond-a} can be re-written in the light of \eqref{eq:ch2-from-cond-uniq}
\begin{align}\label{eq:ch2-re-from cond-a}
T\nu= T(\nu_1+\nu_2)=  T\nu_1+ T\nu_2= T\nu_1+ \mu_1+ \mu_2= (T\nu_1+\mu_1)+ \mu_2, 
\end{align}
where $(T\nu_1+\mu_1)\in \mathcal M_f^1 \text{ and }   \mu_2\in \mathcal M_f^2$. Since the decomposition is unique and using the condition $c)$ of the statement, one can write
\begin{align}
T\nu_1+\mu_1=\nu_1.    
\end{align}
Now, applying the condition $c)$ again one get
\begin{align}
(T\nu_1+\mu_1)(\mathcal X)=(T\nu_1)(\mathcal X)+\mu_1(\mathcal X)=\nu_1(\mathcal X)+\mu_1(\mathcal X).  
\end{align}
But we know $(T\nu_1)(\mathcal X)=\nu_1(\mathcal X)$, thus we have $\mu_1(\mathcal X)=0$, which is to say $\mu_1\equiv 0$.
Hence, 
\begin{align*}
T\nu_1=\nu_1\\
T\nu_2=\nu_2
\end{align*}
\end{proof}

Since excluding singular measures with respect to the Lebesgue measure would suffice to prove ergodicity, we give a necessary and sufficient condition for this, in the case of weak Feller Markov chains on a Polish state-space which is a straightforward generalization of \cite[Theorem 3]{Niclas2005(2)}.
\begin{theorem}\label{thm:cond-sing-inv-msr}
Let $X_k(x)$ be a \emph{weak-Feller Markov chain on a complete separable metric space} $\mathcal X$. There is an invariant measure $\nu$ with a positive singular part if and only if there exists a starting point $x_0\in \mathcal X$ and a compact subset $\mathcal K\subseteq \mathcal X$ with zero Lebesgue measure such that for all $\epsilon>0$:
\begin{align}\label{eq:ch2-cond-for-decom}
\lim_{k\to \infty} \frac{1}{k} \sum\limits_{i=0}^{k-1} \mathbb P\left( X_i(x_0)\in \mathcal K\right)>1-\epsilon.
\end{align}
\end{theorem}
\begin{proof}
Let
\begin{align}
\nu_{0}=\delta_{x_{0}}, \quad \nu_{k}=\frac{1}{k} \sum_{i=0}^{k-1} T^{\star i} \nu_{0}  \end{align}
Then \eqref{eq:ch2-cond-for-decom} gives 
\begin{align}
\lim_{k\to \infty}\nu_k\left(\mathcal K\right) >0
\end{align}
Now choose a weakly convergent sub-sequence $k_i$ such that for some $\widehat{\nu}$, 
\begin{align}
\lim_{i\to \infty}\nu_{k_{i}} \rightarrow \widehat{\nu}\quad\text{ weakly }
\end{align}
Finally due to Portmanteau theorem \cite{Billingsley2013}, 
\begin{align}
\widehat{\nu}(\mathcal K) \geq \limsup _{i \rightarrow \infty} \nu_{k_{i}}(\mathcal K)>0.
\end{align}
Now, let $\mathcal K$ be a compact set of Lebesgue measure zero in $\mathcal{ X}$ such that for some small $\epsilon>0$
\begin{align*}
\nu(\mathcal K )> 1-\epsilon,
\end{align*}
where $\nu$ has a positive singular part. The existence of such a set can be shown as follows: Let $\mu_n$ be the Lebesgue measure on $\mathbb R^n$. If we consider singular measures on $\mathcal X=\mathbb R^n$, then for $n=1$, we can take Cantor measure \cite{Falconer1985, Mattila1995} $\nu$, that is a singular probability measure whose support is the Cantor set $C$. It is known that $C$ is compact, $\mu_1(C)=0$, but $\nu(C) = 1$. For $n>1$, let $x\in \mathbb R^n$ be a generic non-zero vector. Denote $L=\{cx: c\in \mathbb R\}$, be the line generated by $x$, and notice that it is a copy of $\mathbb R$, with $\mu_n(L) =0$. 
Consider the singular measure $\nu$ that behaves like $\mu_1$ on $L$, meaning that
$$\nu(V) = \mu_1(V\cap L)$$ 
for every measurable set $V\subset \mathbb R^n$. This is a singular measure. Since $L$ is closed in $\mathbb R^n$, every compact $\mathcal K\subset L$ is also a compact for $\mathbb R^n$. One can easily find a compact (for example, $[0,1]$) inside $L$ such that $$\mu(\mathcal K) >0,\qquad \mu_k(\mathcal K)\le \mu_k(L)=0.$$
Now, from \cite[Theorem 6 (ii)]{Yosida1995}, for all $w\in \mathcal L^1(\nu)$
\begin{align}\label{ergodic}
\lim _{k \rightarrow \infty} \frac{1}{k} \sum_{i=0}^{k-1} \mathbb{E}\left[ w\left(X_{k}(x)\right)\right]=w^{\star}(x)    
\end{align}
exists $\nu$-almost everywhere, moreover 
\begin{align}
\int_{\mathcal X} w \mathrm{d}\nu= \int_{\mathcal X} w^* \mathrm{d}\nu.
\end{align}
\end{proof}
For more on the decomposition of measures related to Markov operators, we refer to \cite[Chapter 5 - Chapter 8]{Worm2010}, \cite{Zaharopol2005, Zaharopol2009}, and \cite{Buzzi2000} for the Lasota-York type map.

\subsection{Topological aspects of IFS}\label{sec:IFS-Topo-Rec}
\begin{definition}(\cite{Skorokhod1987}\label{df:ch2-topo-recu}
A discrete-time Markov chain which start from $X_0=x\in \mathcal X$ denoted by $\{X_k(x)\}$ on a metric space $(\mathcal X, \rho)$ is said to be \emph{topologically recurrent}, if for any open set $\mathcal U$ in $\mathcal X$ and any starting point $x\in \mathcal X$, we have,
\begin{align}\label{eq:ch2-topo-recu}
\mathbb P\left (\sum\limits_{n=1}^{\infty} \boldsymbol{1}_{\mathcal U}\left(X_k(x)\right)=+\infty\right)=1, 
\end{align}
where $\boldsymbol{1}_{\mathcal U}\left(X_k(x)\right)$ is defined as follows:
\[ \boldsymbol{1}_{\mathcal U}\left(X_k(x)\right) := \begin{cases} 
1 & \text{ if } X_k\in \mathcal U. \\
0 & \text{ if } X_k\in  \mathcal U^c=\left(\mathcal X\setminus \mathcal U\right).
\end{cases}
\]
\end{definition}
The expression $\sum\limits_{n=1}^{+\infty}\boldsymbol{1}_{\mathcal U}\left(X_k(x)\right)$ is infinite i.e diverges iff the set $\{k\mid X_k\in \mathcal U\}$ is infinite. Therefore, the notion of topological recurrence of a Markov chain means that for each starting point $x\in \mathcal X$ and for each open set $\mathcal U$ in $\mathcal X$, the Markov chain $X_k$ will visit the open set $\mathcal U$ an infinite number of times. 

\begin{lemma}(\cite{Skorokhod1987}
Suppose $\{X_k\}$ is a \emph{Feller Markov chain} in a compact metric space $(\mathcal X,\rho)$, then $\{X_k\}$ is \emph{topologically recurrent} if and only if
\begin{align}
\mathbb P\left(X_k(x)\in \mathcal U\right)>0 \text{ for any open set }  \mathcal U  \text{ in } \mathcal X.    
\end{align}
\end{lemma}
\begin{definition}\cite[Definition 2.4]{Usachev1972}\label{df:ch2-stoch-close}
Suppose $X_n$ is a Markov chain with transitional probability kernel $\nu(x, \cdot)$, then a non-empty measurable set $B$ in $\mathcal X$ is called stochastically closed iff $\nu(x, B)=1\; \forall x\in B$.
\end{definition}
In other words, stochastically closed means that with probability $1$, the Markov chain stays in the set once it steps into that set, so this is an analogue of absorbing set and stochastic version of forwarding invariance. No more than one topologically and stochastically closed set exists if a Markov chain is topologically recurrent. For an IFS, topological recurrence does not depend on the fact whether the system is contracting or expanding \cite[Remark 2]{Niclas2005(2)}. 
For topological recurrence and its relation with ergodicity of homogeneous Markov chain on locally compact and compact metric space, we refer to articles by \cite{Skorokhod1987, Skorokhod2007, Harris1956, Blank2018}. In \cite{Samuel2017, Mihail2012} concept of topologically contractive IFS was defined, and recently their existence of invariant probability measure was studied in \cite{Krzysztof2020}.

\subsection{Conformal IFS and Teichm\"{u}ller Theory for IFS}
\emph{Conformal IFS} was introduced in \cite{Mauldin1996}. By using the Perron-Frobenius operator, symbolic dynamics on shift space with an infinite alphabet, part of the article's main result was to show the existence of a unique invariant probability measure for conformal IFS with the countable family of self-maps is equivalent of the existence of a conformal measure. For more development of conformal IFS we refer \cite{Mauldin1999, Mauldin2001,  Lindsay2002, Urbanski2002, Kaenmaki2003, Roy2005, Lau2009, Jaerisch2011, Deng2011, Seuret2015, Rempe2016, Mihailescu2016a, Das2021, Spaulding2022}. D. Sullivan in \cite{Sullivan1982} proposed a measurable version of Mostow's rigidity theorem. If two finite Kleinian groups are conjugate under a nonsingular Borel map $F$, then $F$ agrees practically everywhere with conformal conjugacy. Since this theorem's introduction, measurable rigidity has received much attention, and several generalizations and modifications have been produced. \cite{Hanus1998} proved that two non essentially affine, conformal iterated function systems are conformally comparable if their conformal measures coincide up to generator permutation. The paper \cite{Kessebohmer2006} study rigidity and flexibility for conformal-iterated function systems. They demonstrate that two non-essentially affine systems are conformally identical if and only if there is at least one level set in each Lyapunov spectrum where the Gibbs measures coincide. They compared this result to an affine case, concluded that the essentially affine systems are less rigid than non-essentially affine systems, and  studied their flexibility. Teichm\"{u}ller theory is rich and active research area in mathematics \cite{Hubbard2016}. For the first time, \cite{Hille2012} has introduced Teichm\"{u}ller theory for families of conformal iterated function systems.

\section{Opportunities for Further Research}
\label{sec:further}

In the final section, we sketch out several directions for further research and pointers to existing works along these directions. 

\begin{figure}[h!]
    \centering
    \includegraphics[width=1.02\textwidth]{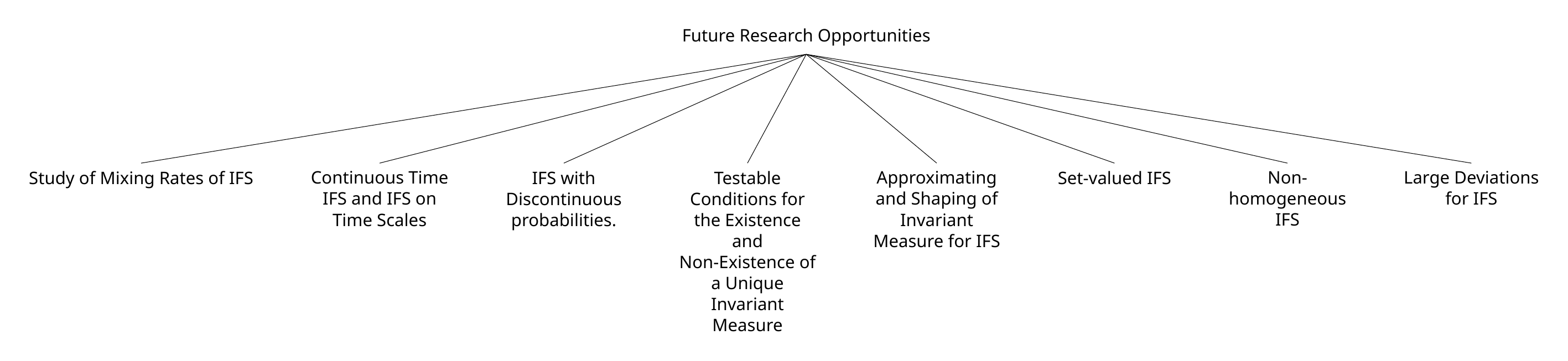}
    \caption{Section \ref{sec:further} of the survey highlights several possible directions for further work.}
    \label{fig:des-3}
\end{figure}

\subsection{Non-asymptotic Study of Mixing Rates in IFS}

It is widely understood that the mixing rate of the IFS is geometric, asymptotically, when there exists a unique invariant measure \cite{Maciej2011} since the corresponding Markov operator is contractive in the Wasserstein norm \cite{Lasota1995}. 
\cite{hairer2002exponential,hairer2006coupling, Hairer2011} have shown how to simplify the analysis using coupling arguments. 

Very recently, \cite{Lesniak2022} showed that the exponent of the rate of convergence could be bounded from below in terms of dimensions of a box enclosing the attractor and from above by various terms, including the Lipschitz constants of the IFS. Related, detailed analyses seem very desirable. 

Within applied probability, detailed analyses of subgeometric rates of convergence under less restrictive assumptions are now appearing \cite{Durmus2016}. These often extend the coupling technique of \cite{hairer2002exponential}. It would be interesting to extend these to iterated function systems.  
See \cite[Chapter 7]{Godland2022} for the first step in this direction. 

\subsection{Continuous-Time IFS and IFS on Time Scales}\label{sec:CTS-IFS}

Continuous-time IFS with place-dependent probabilities
were studied already by \cite{Lasota1999}. They were interested in the biological interpretation and asymptotic stability of IFS and applied stability criteria to describe the stability of cell division structures. The results of Lasota and Mackey were further generalized and extended by \cite{Szarek1999, Szarek2001}, who showed in a general Polish space that such a system is asymptotically stable, and its stationary distribution is singular. A very clear treatment is presented by \cite{Wojewodka2013}.
Applications in biology  \cite[e.g.]{alkurdi2013ergodicity} often rely on these tools, as detailed in Appendix~\ref{subsec:biosc}.

It would be of considerable interest to link the work on continuous-time IFS more closely with the work on stochastic partial differential equations on one hand and with switched systems \cite{liberzon1999basic,zappavigna2010dwell,shorten2007stability} on the other.
In the same spirit, one could develop connections to analysis on time scales, utilizing the tools of \cite{potzsche2003spectral}. See \cite{Ramen2019} for the first small step in this direction. 

\subsection{IFS with Discontinuous Probability Functions}\label{sec:IFS-disc-prob}

Investigating IFS with discontinuous probability functions would be of considerable interest in many practical applications. \cite{Joanna2013} considers an IFS with discontinuous but piecewise constant probability functions was studied and used the \emph{Schauder–Tychonov fixed point theorem} approach, they proved that the system possesses an invariant probability measure. 
\cite{Marecek2022} suggested the use of the machinery of 
\cite{filippov2013differential}.
\cite[Theorem 18]{Fioravanti2019} suggested that one could base an alternative view of IFS with discontinuous probabilities on the work of \cite{Ivan2004, Ivan2005(1), Ivan2005(2), Ivan2005(3), Ivan2005(4), Ivan2005(5), Ivan2005(6), Ivan2006}, although the details remain to be worked out. 

\subsection{Testable Conditions for the Existence and Non-Existence of a Unique Invariant Measure}

In many practical applications, the self-transformations are available only in an oracular fashion, i.e., upon providing index $j \in [m]$ and $X$, one has access to $w_j(X)$. 
In such applications, it may be hard to estimate the moduli of continuity of $w_j, j \in [m]$, to guarantee the contraction on average. Still, a testable condition for a unique invariant measure's (non-)existence would be of considerable interest. 
Perhaps, Malliavin calculus \cite{malliavin1978stochastic,nualart2006malliavin} could provide an avenue towards deriving such conditions. While one could draw connections between Malliavin calculus and IFS following \cite[e.g.]{Gravereaux1988, Viard1994}, much remains to be done in this respect. 

\subsection{Approximating and Shaping of Invariant Measure for IFS}\label{sec:IFS-shaping-inv-meas}

A great deal of interest has been shown in data compression in approximating measures and functions using IFS and other similar approaches. A substantial compression factor may be achieved by representing a target measure or function using a minimal number of IFS parameters, as has been shown in the case of image processing, see, e.g., \cite{Jacquin1988, Jacquin1990(1), Jacquin1990(2), Jacquin1990(3), Jacquin1992, Jacquin1994}.

\cite{Forte1995} state that the inverse problem of measure construction using IFS is as follows: given a target probability measure $\nu\in \mathcal M_p(\mathcal X)$, can we construct using a Markov operator $\mathrm{P}$ an IFS consisting of $m$ functions with constant probabilities $\{p_i\}_{i=1}^{m}$ that has an invariant probability measure of $\nu^{\star}$ that is close to the desired $
\nu$ value, i.e $d_{\text{h}}(\nu^{\star},\nu)<\epsilon$? Definition of \emph{Hutchinson metric} $d_{\text{h}}$ can be found in \cite[Definition 5.2, Chapter IX]{Barnsley1993} is the same as that in \eqref{eq:ch2-Wass-dist}. A solution to the aforesaid problem by using IFS and moment matching process was  proposed by \cite{Barnsley1985, Barnsley1986} and \cite[Chapter 12]{Lasserre2003}  discusses the problem from \emph{optimization} point of view. We recommend that the reader consult \cite{Widder1945, Stieltjes1895, Baker1970, Bessis1991} for further information on the research, the connection between moments and measures and some significance of this problem. An extension of the same problem to the IFS with state-dependent probabilities has been done, and when $\mathcal X=[0,1]$, the solution of the above problem has been achieved in \cite{Mendivil2018}. For more on IFS, the inverse problem of measure and moments, we refer \cite{Elton1989, Diaconis1986, Abenda1990, Abenda1992, Cabrelli1992, Vrscay1991(1), Vrscay1991(2), Vrscay1989, Kunze2012}. The inverse problem for continuous probability measure on $\mathbb R$ has been successfully addressed by \cite{Stenflo2012} and the construction of asymptotically stable IFS has been described in  \cite{Joanna2002, Guzik2016(2)}. 

Let $w$ be a self-map on $(0,\infty)$ and $\{\sigma_k\}_{k\ge 0}$ be i.i.d $(0,\infty)$ valued random variables. \cite{Chamayou1991} studied several explicit invariant probability measures for the Markov chain of the form:
\begin{align}
X_{k+1}=\sigma_{k}(w X_{k}).
\end{align}
When $w$ is a Mobius transformation, its outcome is a particular product of random $2\times 2$ matrices and random continuing fractions when generic exponential families on $(0,+\infty)$ are considered.
Invariant measures of dynamical systems have been researched extensively in the literature. It has been shown that certain algorithms converge to the invariant measure (within the error margins of specific metrics) for certain systems. Asymptotic convergence estimate rates are sometimes provided, see, e.g., \cite{Ding1993, Froyland2007, Murray2010,  Bose2001, Ding1994,  Dellnitz1999,  Dellnitz2002}, however, there are a few results that provide an explicit (rigorous) error bound, see, e.g., \cite{Bahsoun2011, Ippei2011, Liverani2001, Keane1998,  Pollicott2000} and more recent work by \cite{Galatolo2014, Galatolo2009, Galatolo2016}. When the maps are multi-dimensional, \cite{Ding1994, Ding1995} has results on approximating physical invariant measure, which was further generalized by \cite{Froyland1995(1), Froyland1995(2),  Froyland1997(1), Froyland1997(2), Froyland1998} in a series of papers.

When there is a continuous unique invariant measure for a discrete-time dynamical system on a compact metric space arising from IFS, \cite{Hunt1992, Hunt1996(1), Hunt1996(2)} produces techniques for approximating these invariant measures and their convergence rate through a method earlier proposed by \cite{Ulam1960}. For the method description, see \cite[section 11.3]{Van2002}. We refer to several other excellent references for approximating the invariant measure of dynamical systems and invariant measure we refer \cite{Miller1994, Ding2019, Diamond1995, Froyland1999, Froyland2000, Imkeller2003, Boyarsky2012, Christoph2017}. Under some assumptions, in \cite{Robert1995}, invariant measures associated with IFS with finitely many maps on a compact subset of $\mathbb R$ were well approximated. These results were further generalized, and some algorithms were also developed for approximating invariant measures for IFS on higher dimensions in \cite{Oberg2005, Oberg2006}. \cite{Gora2003} has numerical studies on absolutely continuous invariant measures for state-dependent IFS on $[0,1]$, and these results were further extended in \cite{Bahsoun2011, Bahsoun2005(1)}. An interesting observation in \cite{Gora2006} is that there may be a situation where absolutely continuous invariant measures can not be observed experimentally. For more results on the absolutely continuous invariant measure, we refer \cite{Boyarsky1990, Islam2015}  and the references there.

\subsection{Set-valued IFS}

IFS over the space of set-valued self-mappings or multi-functions was re-introduced in \cite{Torre2007, Torre2008, Kunze2008, Colapinto2008} which could be traced back due to \cite{Hutchinson1981}. Fractals and multi-valued fractals are essential in biology, quantum mechanics, computer graphics, dynamical systems, astronomy, and geophysics, see \cite{El1994, Iovane2006(1), Iovane2006(2), Fedeli2005, Shi2004}. IFS over multi-functions have substantial repercussions in applied sciences. In \cite{Llorens2009}, the existence and uniqueness of self-similar mixed-iterated function sets are provided, and the well-posedness of the self-similarity issue for various iterated multi-function systems is also explored. Well-posedness is intimately connected to approximating a fixed point equation, which is vital for constructing fractals using pre-fractals. Due to the wide range of influential application research on set-valued IFS is still active \cite{Nia2015, Guzik2016(1), Georgescu2020, Goyal2021, Pandey2022, Russell2018}.

\subsection{Non-homogeneous IFS}\label{sec:non-hmo-IFS}
A significant portion of the IFS literature assumes time-homogeneous probabilities on drawing functions. In addition, convergence results are obtained within the assumption of certain average cotractivity.  \cite{Stenflo1998} and later \cite{Mendivil2015} have convergence results based on time-dependent and time-varying probabilities. These time-varying assumptions on probabilities have been missed in the IFS literature. The results radically different from those obtained in the conventional time-homogeneous contractive IFS. In \cite{lu1997}, a necessary condition for the existence of a $\sigma$-finite invariant measure for Markov chains with random transitions was discovered, along with probabilities demonstrating the significance of attractors for particular IFS to their invariant measures. In \cite{Ramen2019}, the notion of time-varying piecewise IFS was presented in Polish space. 

\subsection{Large Deviations for IFS}\label{sec:Ldevi-IFS}
Recently \emph{large deviations\cite[Feller XVI.7]{Feller2008} probabilities and their applications} have received a considerable amount of attention; for example, in probability and statistics \cite{Varadhan1984, Hollander2000, Deuschel2001, Dupuis2011}, information theory and queuing theory \cite{Shwartz1995}, statistical mechanics \cite{Touchette2009, Oono1989}, DNA analysis, communications and control \cite{Amir2010, Freidlin2012}. In a dynamical system, the theory of large deviations allows one to find the likelihood that the ergodic average $\frac{1}{n}\sum\limits_{i=0}^{n-1}f(T^k(x))$ deviate from the mean $\int f d \mu=0$.

\cite[Chapter VIII]{Hennion2001} have obtained such results using transfer operators, and \cite{Pollicott2001} did the same for IFSs that contract on average. However, it should be mentioned that Pollicott assumes, without complete proof, that the perturbed transfer operators are analytic perturbations of the unperturbed transfer operator, as noted in \cite[Page-1940]{Hennion2004}. \cite{Broise1996} looked at such results for piecewise expanding maps and was able to say that, like many large deviations results, the probability that ergodic averages deviate from the mean decays exponentially fast. Motivated from the works like \cite[Chapter 10]{Broise1996, Hennion2004, Hennion2001}, and \cite{Santos2013,Bezuglyi2018}, \cite[Chapter 3]{Anthony2015} has several results  for large deviations for \emph{IFS that contract on average}.

\section{In Place of Conclusions}

The survey outlined the development of a relatively simple notion of \emph{contraction on average} into substantial literature on iterated function systems. Indeed, our survey alone covers over 400 papers, several books, and a steady stream of PhD dissertations in Applied Probability and related areas of Mathematics:  \cite{Giles1990, Daryl1991, Hepting1991(2), Vines1993, Christine1995, Stenflo1999, Hanus2000, Dai2001, Kapica2003, Niclas2005(1), Doan2009, Worm2010, Vass2013, Shaun2015, Anthony2015, Davison2015, Helena2016, wojewodka2016phd, Masoumeh2017, Godland2022}. We suggest the breadth of the literature utilizing iterated function systems in several application domains within the appendices.

\bibliographystyle{apalike}
\bibliography{Thesis}

\begin{thebibliography}{}

\bibitem[Abenda, 1990]{Abenda1990}
Abenda, S. (1990).
\newblock Inverse problem for one-dimensional fractal measures via iterated
  function systems and the moment method.
\newblock {\em Inverse Problems}, 6(6):885--896.

\bibitem[Abenda et~al., 1992]{Abenda1992}
Abenda, S., Demko, S., and Turchetti, G. (1992).
\newblock Local moments and inverse problem of fractal measures.
\newblock {\em Inverse Problems}, 8(5):739.

\bibitem[Abraham and Ueda, 2000]{Abraham2000}
Abraham, R. and Ueda, Y. (2000).
\newblock {\em The chaos avant-garde: Memories of the early days of chaos
  theory}, volume~39.
\newblock World scientific.

\bibitem[Abrams et~al., 2003]{Abrams2003}
Abrams, A., Landau, H., Landau, Z., Pommersheim, J., and Zaslow, E. (2003).
\newblock An iterated random function with {L}ipschitz number one.
\newblock {\em Theory of Probability \& Its Applications}, 47(2):190--201.

\bibitem[Acemo{\u{g}}lu et~al., 2013]{Acemouglu2013}
Acemo{\u{g}}lu, D., Como, G., Fagnani, F., and Ozdaglar, A. (2013).
\newblock Opinion fluctuations and disagreement in social networks.
\newblock {\em Mathematics of Operations Research}, 38(1):1--27.

\bibitem[Acemoglu et~al., 2010]{Acemoglu2010}
Acemoglu, D., Ozdaglar, A., and ParandehGheibi, A. (2010).
\newblock Spread of (mis) information in social networks.
\newblock {\em Games and Economic Behavior}, 70(2):194--227.

\bibitem[Alexander et~al., 2012]{Alexander2012}
Alexander, Z., Meiss, J.~D., Bradley, E., and Garland, J. (2012).
\newblock Iterated function system models in data analysis: Detection and
  separation.
\newblock {\em Chaos: An Interdisciplinary Journal of Nonlinear Science},
  22(2):023103.

\bibitem[Alkurdi et~al., 2013]{alkurdi2013ergodicity}
Alkurdi, T., Hille, S.~C., and van Gaans, O. (2013).
\newblock Ergodicity and stability of a dynamical system perturbed by impulsive
  random interventions.
\newblock {\em Journal of Mathematical Analysis and Applications},
  407(2):480--494.

\bibitem[Almeida, 2014]{Almeida2014}
Almeida, J.~S. (2014).
\newblock Sequence analysis by iterated maps, a review.
\newblock {\em Briefings in bioinformatics}, 15 3:369--75.

\bibitem[Almeida and Vinga, 2009]{Almeida2009}
Almeida, J.~S. and Vinga, S. (2009).
\newblock Biological sequences as pictures -- a generic two dimensional
  solution for iterated maps.
\newblock {\em BMC Bioinformatics}, 10(1):100.

\bibitem[Alsmeyer and L{\"o}we, 2013]{Alsmeyer2013}
Alsmeyer, G. and L{\"o}we, M. (2013).
\newblock Random matrices and iterated random functions.
\newblock {\em Springer Proceedings in Mathematics and Statistics}, 53.

\bibitem[Amir and Ofer, 2010]{Amir2010}
Amir, D. and Ofer, Z. (2010).
\newblock {\em Large Deviations Techniques and Applications}, volume~38.
\newblock Springer-Verlag Berlin Heidelberg, second edition.

\bibitem[Anantharam and Konstantopoulos, 1997]{Anantharam1997}
Anantharam, V. and Konstantopoulos, T. (1997).
\newblock Stationary solutions of stochastic recursions describing discrete
  event systems.
\newblock {\em Stochastic Processes and their applications}, 68(2):181--194.

\bibitem[Anatoliy and Wu, 2003]{Anatoliy2003}
Anatoliy, S. and Wu, J. (2003).
\newblock {\em Evolution of Biological Systems in Random Media: Limit Theorems
  and Stability}, volume~18.
\newblock Springer Netherlands.

\bibitem[Andreas and Stenflo, 2005]{Andreas2005}
Andreas, N.~L. and Stenflo, {\"O}. (2005).
\newblock Central limit theorems for contractive {M}arkov chains.
\newblock {\em Nonlinearity}, 18(5):1955--1965.

\bibitem[Andres et~al., 2005]{Andres2005}
Andres, J., Fi{\v{s}}er, J., Gabor, G., and Le{\'s}niak, K. (2005).
\newblock Multivalued fractals.
\newblock {\em Chaos, Solitons \& Fractals}, 24(3):665--700.

\bibitem[Arnold and Crauel, 1992]{Crauel1992}
Arnold, L. and Crauel, H. (1992).
\newblock {\em Iterated Function Systems and Multiplicative Ergodic Theory},
  volume~2, pages 283--305.
\newblock Birkhäuser, Boston, MA.

\bibitem[Arrowsmith et~al., 1990]{Arrowsmith1990}
Arrowsmith, D.~K., Place, C.~M., Place, C., et~al. (1990).
\newblock {\em An introduction to dynamical systems}.
\newblock Cambridge university press.

\bibitem[Ashlock and Golden, 2003]{Ashlock2003}
Ashlock, D. and Golden, J. (2003).
\newblock Chaos automata: Iterated function systems with memory.
\newblock {\em Physica D: Nonlinear Phenomena}, 181(3-4):274--285.

\bibitem[Athreya, 2003a]{Athreya2003(1)}
Athreya, K.~B. (2003a).
\newblock Iteration of {I.I.D} random maps on {$\mathbb R^{+}$}.
\newblock {\em Lecture Notes-Monograph Series}, pages 1--14.

\bibitem[Athreya, 2003b]{Athreya2003(4)}
Athreya, K.~B. (2003b).
\newblock Stationary measures for some {M}arkov chain models in ecology and
  economics.
\newblock {\em Economic Theory}, 23(1):107--122.

\bibitem[Athreya, 2016]{Athreya2016}
Athreya, K.~B. (2016).
\newblock Dynamical systems, {I.I.D} random iterations, and {M}arkov chains.
\newblock In {\em Rabi N Bhattacharya}, pages 265--275. Springer.

\bibitem[Athreya and Bhattacharya, 2001]{Athreya2001}
Athreya, K.~B. and Bhattacharya, R.~N. (2001).
\newblock Random iteration of {I.I.D} quadratic maps.
\newblock In {\em Stochastics in finite and infinite dimensions}, pages 49--58.
  Springer.

\bibitem[Athreya and Dai, 2000]{Athreya2000}
Athreya, K.~B. and Dai, J. (2000).
\newblock Random logistic maps {I}.
\newblock {\em Journal of Theoretical Probability}, 13(2):595--608.

\bibitem[Athreya and Schuh, 2003]{Athreya2003(2)}
Athreya, K.~B. and Schuh, H.~J. (2003).
\newblock Random logistic maps {II}. {T}he critical case.
\newblock {\em Journal of Theoretical Probability}, 16(4):813--830.

\bibitem[Athreya and Stenflo, 2003]{Athreya2003(3)}
Athreya, K.~B. and Stenflo, {\"O}. (2003).
\newblock Perfect sampling for {D}oeblin chains.
\newblock {\em Sankhya: The Indian Journal of Statistics}, pages 763--777.

\bibitem[Bahsoun and Bose, 2011]{Bahsoun2011}
Bahsoun, W. and Bose, C. (2011).
\newblock Invariant densities and escape rates: rigorous and computable
  approximations in the ${\ell}^{\infty}$-norm.
\newblock {\em Nonlinear Analysis: Theory, Methods \& Applications},
  74(13):4481--4495.

\bibitem[Bahsoun et~al., 2005a]{Bahsoun2005(1)}
Bahsoun, W. et~al. (2005a).
\newblock Position dependent random maps in one and higher dimensions.
\newblock {\em Studia Mathematica}, 166:271--286.

\bibitem[Bahsoun et~al., 2005b]{Bahsoun2005(2)}
Bahsoun, W., G{\'o}ra, P., and Boyarsky, A. (2005b).
\newblock {M}arkov switching for position dependent random maps with
  application to forecasting.
\newblock {\em SIAM Journal on Applied Dynamical Systems}, 4(2):391--406.

\bibitem[Baker and Gammel, 1970]{Baker1970}
Baker, G.~A. and Gammel, J.~L. (1970).
\newblock {\em The Pad{\'e} approximant in theoretical physics}.
\newblock Academic Press.

\bibitem[Bakhtin, 2015]{Bakhtin2015}
Bakhtin, Y. (2015).
\newblock Some topics in ergodic theory.
\newblock {\em Lecture notes}.

\bibitem[Bala and Majumdar, 1992]{Bala1992}
Bala, V. and Majumdar, M. (1992).
\newblock Chaotic tatonnement.
\newblock {\em Economic Theory}, 2(4):437--445.

\bibitem[Bandt and P., 2017]{Christoph2017}
Bandt, C. and P., H. (2017).
\newblock Polynomial approximation of self-similar measures and the spectrum of
  the transfer operator.
\newblock {\em Discrete \& Continuous Dynamical Systems-A}, 37(9):4611--4623.

\bibitem[Baraviera et~al., 2009]{Baraviera2009}
Baraviera, A.~T., Lardizabal, C.~F., Lopes, A.~O., and Cunha, M. O.~T. (2009).
\newblock Quantum stochastic processes, quantum iterated function systems and
  entropy.
\newblock {\em arXiv: Dynamical Systems}, pages 53--87.

\bibitem[Barnsley, 1993]{Barnsley1993}
Barnsley, M.~F. (1993).
\newblock {\em Fractals everywhere, 2nd Edition}.
\newblock Academic press.

\bibitem[Barnsley and Demko, 1985]{Barnsley1985}
Barnsley, M.~F. and Demko, S. (1985).
\newblock Iterated function systems and the global construction of fractals.
\newblock {\em Proceedings of the Royal Society of London. A. Mathematical and
  Physical Sciences}, 399(1817):243--275.

\bibitem[Barnsley et~al., 1988]{Barnsley1988(2)}
Barnsley, M.~F., Demko, S.~G., Elton, J.~H., and Geronimo, J.~S. (1988).
\newblock Invariant measures for {M}arkov processes arising from iterated
  function systems with place-dependent probabilities.
\newblock In {\em Annales de l'IHP Probability et statistiques}, volume~24,
  pages 367--394.

\bibitem[Barnsley et~al., 1989a]{Barnsley1988(2)erratum}
Barnsley, M.~F., Demko, S.~G., Elton, J.~H., and Geronimo, J.~S. (1989a).
\newblock Erratum invariant measures for {M}arkov processes arising from
  iterated function systems with place-dependent probabilities.
\newblock In {\em Annales de l'IHP Probabilit{\'e}s et statistiques},
  volume~25, pages 589--590.

\bibitem[Barnsley and Elton, 1988]{Barnsley1988(3)}
Barnsley, M.~F. and Elton, J.~H. (1988).
\newblock A new class of {M}arkov processes for image encoding.
\newblock {\em Advances in applied probability}, 20(1):14--32.

\bibitem[Barnsley et~al., 1989b]{Barnsley1989}
Barnsley, M.~F., Elton, J.~H., and Hardin, D.~P. (1989b).
\newblock Recurrent iterated function systems.
\newblock {\em Constructive approximation}, 5(1):3--31.

\bibitem[Barnsley et~al., 1986]{Barnsley1986}
Barnsley, M.~F., Ervin, V., Hardin, D., and Lancaster, J. (1986).
\newblock Solution of an inverse problem for fractals and other sets.
\newblock {\em Proceedings of the National Academy of Sciences of the United
  States of America}, 83(7):1975.

\bibitem[Barnsley et~al., 2008]{Barnsley2008}
Barnsley, M.~F., Hutchinson, J.~E., and Stenflo, {\"O}. (2008).
\newblock V-variable fractals: fractals with partial self similarity.
\newblock {\em Advances in Mathematics}, 218(6):2051--2088.

\bibitem[Barnsley and Jacquin, 1988]{Jacquin1988}
Barnsley, M.~F. and Jacquin, A.~E. (1988).
\newblock Application of recurrent iterated function systems to images.
\newblock In {\em Visual Communications and Image Processing'88: Third in a
  Series}, volume 1001, pages 122--131. International Society for Optics and
  Photonics.

\bibitem[Barnsley and Sloan, 1990]{Barnsley1990}
Barnsley, M.~F. and Sloan, A.~D. (1990).
\newblock Methods and apparatus for image compression by iterated function
  system.
\newblock {\em Google Patents}.
\newblock US Patent 4,941,193.

\bibitem[Benda, 1998]{Benda1998}
Benda, M. (1998).
\newblock A central limit theorem for contractive stochastic dynamical systems.
\newblock {\em Journal of Applied Probability}, 35(1):200--205.

\bibitem[Berger, 1989]{Berger1989}
Berger, M.~A. (1989).
\newblock Images generated by orbits of {2D} {M}arkov chains.
\newblock {\em Chance}, 2(2):18--28.

\bibitem[Berger, 1992]{Berger1992}
Berger, M.~A. (1992).
\newblock Random affine iterated function systems: curve generation and
  wavelets.
\newblock {\em SIAM review}, 34(3):361--385.

\bibitem[Berger and Soner, 1988]{Berger1988}
Berger, M.~A. and Soner, H.~M. (1988).
\newblock Random walks generated by affine mappings.
\newblock {\em Journal of Theoretical Probability}, 1(3):239--254.

\bibitem[Berger et~al., 2018]{Berger2018}
Berger, N., Hoffman, C., and Sidoravicius, V. (2018).
\newblock Non-uniqueness for specifications in $\ell ^{2+\epsilon}$.
\newblock {\em Ergodic Theory and Dynamical Systems}, 38(4):1342–1352.

\bibitem[Bessis and Demko, 1991]{Bessis1991}
Bessis, D. and Demko, S. (1991).
\newblock Stable recovery of fractal measures by polynomial sampling.
\newblock {\em Physica D: Nonlinear Phenomena}, 47(3):427--438.

\bibitem[Bezuglyi and Jorgensen, 2018]{Bezuglyi2018}
Bezuglyi, S. and Jorgensen, P. E.~T. (2018).
\newblock {\em Iterated Function Systems and Transfer Operators}, pages
  133--142.
\newblock Springer International Publishing, Cham.

\bibitem[Bhattacharya and Majumdar, 1999]{Bhattacharya1999}
Bhattacharya, R.~N. and Majumdar, M. (1999).
\newblock On a theorem of {D}ubins and {F}reedman.
\newblock {\em Journal of Theoretical Probability}, 12(4):1067--1087.

\bibitem[Bhattacharya and Majumdar, 2001]{Bhattacharya2001(2)}
Bhattacharya, R.~N. and Majumdar, M. (2001).
\newblock On a class of stable random dynamical systems: {T}heory and
  applications.
\newblock {\em Journal of Economic Theory}, 96(1-2):208--229.

\bibitem[Bhattacharya and Majumdar, 2003]{Bhattacharya2003}
Bhattacharya, R.~N. and Majumdar, M. (2003).
\newblock Random dynamical systems: a review.
\newblock {\em Economic Theory}, 23(1):13--38.

\bibitem[Bhattacharya and Majumdar, 2007]{Bhattacharya2007}
Bhattacharya, R.~N. and Majumdar, M. (2007).
\newblock {\em Random Dynamical Systems: Theory and Applications}.
\newblock Cambridge University Press.

\bibitem[Bhattacharya et~al., 2004]{Bhattacharya2004}
Bhattacharya, R.~N., Majumdar, M., et~al. (2004).
\newblock Stability in distribution of randomly perturbed quadratic maps as
  {M}arkov processes.
\newblock {\em The Annals of Applied Probability}, 14(4):1802--1809.

\bibitem[Bhattacharya and Rao, 1993]{Bhattacharya1993}
Bhattacharya, R.~N. and Rao, B. (1993).
\newblock Random iterations of two quadratic maps.
\newblock In {\em Stochastic processes}, pages 13--22. Springer.

\bibitem[Bhattacharya and Waymire, 2001]{Bhattacharya2001(1)}
Bhattacharya, R.~N. and Waymire, E.~C. (2001).
\newblock Iterated random maps and some classes of {M}arkov processes.
\newblock {\em Stochastic Processes: Theory and Methods}, 19:145.

\bibitem[Bhattacharya and Waymire, 2009]{Bhattacharya2009}
Bhattacharya, R.~N. and Waymire, E.~C. (2009).
\newblock {\em Stochastic processes with applications}.
\newblock SIAM.

\bibitem[Billingsley, 2013]{Billingsley2013}
Billingsley, P. (2013).
\newblock {\em Convergence of probability measures}.
\newblock John Wiley \& Sons.

\bibitem[Blank, 2018]{Blank2018}
Blank, M. (2018).
\newblock Topological and metric recurrence for general {M}arkov chains.
\newblock {\em arXiv e-prints}.

\bibitem[Borovkov, 1998]{Borovkov1998}
Borovkov, A.~A. (1998).
\newblock {\em Ergodicity and stability of stochastic processes}.
\newblock J. Wiley.

\bibitem[Borovkov and Foss, 1992]{Borovkov1992}
Borovkov, A.~A. and Foss, S.~G. (1992).
\newblock Stochastically recursive sequences and their generalizations.
\newblock {\em Siberian Advances in Mathematics}, 2(1):16--81.

\bibitem[Bose and Murray, 2001]{Bose2001}
Bose, C. and Murray, R. (2001).
\newblock The exact rate of approximation in {U}lam's method.
\newblock {\em Discrete \& Continuous Dynamical Systems-A}, 7(1):219.

\bibitem[Boyarsky, 1990]{Boyarsky1990}
Boyarsky, A. (1990).
\newblock Fractal approximation by absolutely continuous invariant measures.
\newblock {\em Physics Letters A}, 149(9):452--456.

\bibitem[Boyarsky and Gora, 2012]{Boyarsky2012}
Boyarsky, A. and Gora, P. (2012).
\newblock {\em Laws of chaos: invariant measures and dynamical systems in one
  dimension}.
\newblock Springer Science \& Business Media.

\bibitem[Bramson and Kalikow, 1993]{Bramson1993}
Bramson, M. and Kalikow, S. (1993).
\newblock Nonuniqueness in {g}-functions.
\newblock {\em Israel Journal of Mathematics}, 84(1-2):153--160.

\bibitem[Branicky, 1998]{Branicky1998}
Branicky, M.~S. (1998).
\newblock Multiple {L}yapunov functions and other analysis tools for switched
  and hybrid systems.
\newblock {\em IEEE Transactions on Automatic Control}, 43(4):475--482.

\bibitem[Bratteli and J{\o}rgensen, 1999]{Bratteli1999}
Bratteli, O. and J{\o}rgensen, P.~E. (1999).
\newblock {\em Iterated function systems and permutation representations of the
  Cuntz algebra}, volume 663.
\newblock American Mathematical Soc.

\bibitem[Breiman, 1960]{Breiman1960}
Breiman, L. (1960).
\newblock The strong law of large numbers for a class of {M}arkov chains.
\newblock {\em Ann. Math. Statist.}, 31(3):801--803.

\bibitem[Bressloff and Stark, 1991]{Bressloff1991}
Bressloff, P. and Stark, J. (1991).
\newblock Neural networks, learning automata and iterated function systems.
\newblock In {\em Fractals and chaos}, pages 145--190. Springer.

\bibitem[Bressloff, 1992]{Bressloff1992(2)}
Bressloff, P.~C. (1992).
\newblock Analysis of quantal synaptic noise in neural networks using iterated
  function systems.
\newblock {\em Phys. Rev. A}, 45:7549--7559.

\bibitem[Bressloff and Stark, 1992]{Bressloff1992(1)}
Bressloff, P.~C. and Stark, J. (1992).
\newblock Analysis of associative reinforcement learning in neural networks
  using iterated function systems.
\newblock {\em IEEE Transactions on Systems, Man, and Cybernetics},
  22(6):1348--1360.

\bibitem[Broise, 1996]{Broise1996}
Broise, A. (1996).
\newblock Transformations dilatantes de l'intervalle et th{\'e}or{\`e}mes
  limites.
\newblock {\em Asterisque Societe Mathematique de France}, 238:1--109.

\bibitem[Burton and Keller, 1993]{Burton1993}
Burton, R.~M. and Keller, G. (1993).
\newblock Stationary measures for randomly chosen maps.
\newblock {\em Journal of Theoretical Probability}, 6(1):1--16.

\bibitem[Burton and Rosler, 1995]{Burton1995}
Burton, R.~M. and Rosler, U. (1995).
\newblock An $l_2$ convergence theorem for random affine mappings.
\newblock {\em Journal of Applied Probability}, 32(1):183--192.

\bibitem[Buzzi, 2000]{Buzzi2000}
Buzzi, J. (2000).
\newblock Absolutely continuous {SRB} measures for random {L}asota-{Y}orke
  maps.
\newblock {\em Transactions of the American Mathematical Society},
  352(7):3289--3303.

\bibitem[C~Ionescu, 1948]{Ionescu1948}
C~Ionescu, G.~M. (1948).
\newblock Sur certaines chaines a liaisons completes.
\newblock {\em C.R.A.S}, 227:667--669.

\bibitem[Cabrelli et~al., 1992]{Cabrelli1992}
Cabrelli, C., Molter, U., and Vrscay, E. (1992).
\newblock Moment matching for the approximation of measures using iterated
  function systems.
\newblock {\em preprint}.

\bibitem[Carlsson, 2005a]{Niclas2005(1)}
Carlsson, N. (2005a).
\newblock {\em Markov Chains on Metric Spaces : {I}nvariant {M}easures and
  {A}symptotic {B}ehaviour}.
\newblock PhD thesis, Abo Akademi University.

\bibitem[Carlsson, 2005b]{Niclas2005(2)}
Carlsson, N. (2005b).
\newblock Some notes on topological recurrence.
\newblock {\em Electronic Communications in Probability}, 10:82--93.

\bibitem[\c{C}inlar, 1974]{Cinlar1975}
\c{C}inlar, E. (1974).
\newblock {\em Introduction to stochastic processes}.
\newblock Englewood Cliffs, N.J. : Prentice-Hall.

\bibitem[Chakraborty, 2016]{Arnab2016}
Chakraborty, A. (2016).
\newblock {\em Probability and Statistics}.
\newblock Levant Books.

\bibitem[Chakraborty and Rao, 1998]{Rao1998}
Chakraborty, S. and Rao, B.~V. (1998).
\newblock Completeness of {B}hattacharya metric on the space of probabilities.
\newblock {\em Statistics \& probability letters}, 36(4):321--326.

\bibitem[Chamayou and Letac, 1991]{Chamayou1991}
Chamayou, J.~F. and Letac, G. (1991).
\newblock Explicit stationary distributions for compositions of random
  functions and products of random matrices.
\newblock {\em Journal of Theoretical Probability}, 4:3--36.

\bibitem[Chen and Ueta, 2002]{Chen2002}
Chen, G. and Ueta, T. (2002).
\newblock {\em Chaos in circuits and systems}, volume~11.
\newblock World Scientific.

\bibitem[Chiu, 2015]{Anthony2015}
Chiu, A. (2015).
\newblock {\em Iterated function systems that contract on average}.
\newblock PhD thesis, University of Manchester, Faculty of Engineering and
  Physical Sciences.

\bibitem[Chiu and Jain, 1989]{Chiu1989}
Chiu, D.~M. and Jain, R. (1989).
\newblock Analysis of the increase and decrease algorithms for congestion
  avoidance in computer networks.
\newblock {\em Computer Networks and ISDN systems}, 17(1):1--14.

\bibitem[Colapinto and Torre, 2008]{Colapinto2008}
Colapinto, C. and Torre, D.~L. (2008).
\newblock Iterated function systems, iterated multifunction systems, and
  applications.
\newblock In {\em Mathematical and Statistical Methods in Insurance and
  Finance}, pages 83--90. Springer.

\bibitem[Collet and Eckmann, 2009]{Collet2009}
Collet, P. and Eckmann, J.~P. (2009).
\newblock {\em Iterated maps on the interval as dynamical systems}.
\newblock Springer Science \& Business Media.

\bibitem[Corless et~al., 2016]{Corless2016}
Corless, M., King, C., Shorten, R., and Wirth, F. (2016).
\newblock {\em {AIMD} Dynamics and Distributed Resource Allocation}.
\newblock Number~29 in Advances in Design and Control. SIAM, Philadelphia, PA.

\bibitem[Corless and Shorten, 2012a]{Corless2012}
Corless, M. and Shorten, R. (2012a).
\newblock Deterministic and stochastic convergence properties of {AIMD}
  algorithms with nonlinear back-off functions.
\newblock {\em Automatica}, 48(7):1291--1299.

\bibitem[Corless and Shorten, 2012b]{Martin2012}
Corless, M. and Shorten, R. (2012b).
\newblock An ergodic {AIMD} algorithm with application to high-speed networks.
\newblock {\em International Journal of control}, 85(6):746--764.

\bibitem[Cox et~al., 1979]{Cox1979}
Cox, J.~C., Ross, S.~A., and Rubinstein, M. (1979).
\newblock Option pricing: A simplified approach.
\newblock {\em Journal of financial Economics}, 7(3):229--263.

\bibitem[Crilly et~al., 1991]{Crilly1991}
Crilly, A., Eamshaw, R., and Jones, H. (1991).
\newblock Introduction fractals and chaos.
\newblock In {\em Fractals and chaos}, pages 1--4. Springer.

\bibitem[Crisostomi et~al., 2016]{Bob2016}
Crisostomi, E., Shorten, R., and Wirth, F. (2016).
\newblock Smart cities: A golden age for control theory?
\newblock {\em IEEE Technology and Society Magazine}, 35(3):23--24.

\bibitem[Cuzzocrea et~al., 2017]{Alfredo2017}
Cuzzocrea, A., Mumolo, E., and Grasso, G.~M. (2017).
\newblock Genetic estimation of iterated function systems for accurate fractal
  modeling in pattern recognition tools.
\newblock In Gervasi, O., Murgante, B., Misra, S., Borruso, G., Torre, C.~M.,
  Rocha, A. M.~A., Taniar, D., Apduhan, B.~O., Stankova, E., and Cuzzocrea, A.,
  editors, {\em Computational Science and Its Applications -- ICCSA 2017},
  pages 357--371, Cham. Springer International Publishing.

\bibitem[Czapla et~al., 2019]{Czapla2019}
Czapla, D., Horbacz, K., and Wojew{\'o}dka, H. (2019).
\newblock Limit theorems for a {M}arkov model of autoregulated gene expression.
\newblock In {\em AIP Conference Proceedings}, volume 2116, page 450055. AIP
  Publishing LLC.

\bibitem[Czapla et~al., 2020]{Czapla2020}
Czapla, D., Horbacz, K., and Wojew{\'o}dka, H. (2020).
\newblock Ergodic properties of some piecewise-deterministic {M}arkov process
  with application to gene expression modelling.
\newblock {\em Stochastic Processes and their Applications}, 130(5):2851--2885.

\bibitem[Dai, 2000]{Dai2000}
Dai, J.~J. (2000).
\newblock A result regarding convergence of random logistic maps.
\newblock {\em Statistics \& probability letters}, 47(1):11--14.

\bibitem[Dai, 2001]{Dai2001}
Dai, J.~J. (2001).
\newblock {\em Some results in probability and theoretical computer science}.
\newblock PhD thesis, Iowa State University.

\bibitem[Dan and Alexandru, 2009]{Dumitru2009}
Dan, D. and Alexandru, M. (2009).
\newblock A sufficient condition for the connectedness of the attractors of
  infinite iterated function systems.
\newblock {\em Analele ştiinţifice ale Universităţii “Alexandru Ioan
  Cuza” din Iaşi. Matematică (SERIE NOUĂ).}, 55(1):87--94.
\newblock The journals is also known as the Annals of the University
  ``Alexandru Ioan Cuza'' .

\bibitem[Daryl, 1991]{Daryl1991}
Daryl, H.~H. (1991).
\newblock {\em Approximation and visualization of sets defined by Iterated
  Function Systems}.
\newblock PhD thesis, Computer Science, Faculty of Science, University of
  Regina.

\bibitem[Das and Simmons, 2021]{Das2021}
Das, T. and Simmons, D. (2021).
\newblock On the dimension spectra of infinite conformal iterated function
  systems.
\newblock {\em Journal of Fractal Geometry}.

\bibitem[David and Moson, 1992]{David1992}
David, S.~M. and Moson, H. (1992).
\newblock Using iterated function systems to model discrete sequences.
\newblock {\em IEEE Trans on Signal Processing}, 40(7):1724--1734.

\bibitem[Davison, 2018]{Davison2018}
Davison, T. (2018).
\newblock A positive operator-valued measure for an iterated function system.
\newblock {\em Acta Applicandae Mathematicae}, 157(1):1–24.

\bibitem[Davison, 2015]{Davison2015}
Davison, T.~H. (2015).
\newblock {\em Generalizing the {K}antorovich Metric to Projection-Valued
  Measures: With an Application to Iterated Function}.
\newblock PhD thesis, University of Colorado at Boulder.

\bibitem[Davoine and Chassery, 2009]{Davoine2010}
Davoine, F. and Chassery, J.~M. (2009).
\newblock {\em Iterated Function Systems and Applications in Image Processing},
  chapter~10, pages 333--366.
\newblock John Wiley \& Sons, Ltd.

\bibitem[Day et~al., 1994]{Day1994}
Day, R.~H. et~al. (1994).
\newblock Complex {E}conomic {D}ynamics {V}olume {I}: An {I}ntroduction to
  {D}ynamical {S}ystems and {M}arket {M}echanisms.
\newblock {\em MIT Press Books}.

\bibitem[De~Melo and Van~Strien, 2012]{Melo2012}
De~Melo, W. and Van~Strien, S. (2012).
\newblock {\em One-dimensional dynamics}, volume~25.
\newblock Springer Science \& Business Media.

\bibitem[Decrouez et~al., 2009]{Decrouez2009}
Decrouez, G., Amblard, P.~O., Brossier, J.~M., and Jones, O. (2009).
\newblock Galton-watson iterated function systems.
\newblock {\em Journal of Physics A: Mathematical and Theoretical},
  42(9):095101.

\bibitem[Dellnitz and Junge, 1999]{Dellnitz1999}
Dellnitz, M. and Junge, O. (1999).
\newblock On the approximation of complicated dynamical behavior.
\newblock {\em SIAM Journal on Numerical Analysis}, 36(2):491--515.

\bibitem[Dellnitz and Junge, 2002]{Dellnitz2002}
Dellnitz, M. and Junge, O. (2002).
\newblock Set oriented numerical methods for dynamical systems.
\newblock {\em Handbook of dynamical systems}, 2(1):900.

\bibitem[Deng and Ngai, 2011]{Deng2011}
Deng, Q.-R. and Ngai, S.-M. (2011).
\newblock Conformal iterated function systems with overlaps.
\newblock {\em Dynamical Systems}, 26(1):103--123.

\bibitem[Denker and Waymire, 2016]{Denker2016}
Denker, M. and Waymire, E.~C. (2016).
\newblock {\em Rabi N Bhattacharya: Selected Papers}.
\newblock Birkh{\"a}user.

\bibitem[Deuschel and Stroock, 2001]{Deuschel2001}
Deuschel, J.-D. and Stroock, D.~W. (2001).
\newblock {\em Large deviations}, volume 342.
\newblock American Mathematical Soc.

\bibitem[Devaney, 2008]{Devaney2008}
Devaney, R. (2008).
\newblock {\em An introduction to chaotic dynamical systems}.
\newblock Westview press.

\bibitem[Diaconis and Freedman, 1999]{Diaconis1999}
Diaconis, P. and Freedman, D. (1999).
\newblock Iterated random functions.
\newblock {\em SIAM Review}, 41(1):45--76.

\bibitem[Diaconis and Shahshahani, 1986]{Diaconis1986}
Diaconis, P. and Shahshahani, S. (1986).
\newblock Products of random matrices and computer image generation.
\newblock {\em Contemp. Math}, 50:173--182.

\bibitem[Diamond et~al., 1995]{Diamond1995}
Diamond, P., Kloeden, P., and Pokrovskii, A. (1995).
\newblock Analysis of an algorithm for computing invariant measures.
\newblock {\em Nonlinear Analysis: Theory, Methods \& Applications},
  24(3):323--336.

\bibitem[Ding et~al., 1993]{Ding1993}
Ding, J., Du, Q., and Li, T.-Y. (1993).
\newblock High order approximation of the {F}robenius-{P}erron operator.
\newblock {\em Applied Mathematics and Computation}, 53(2-3):151--171.

\bibitem[Ding and Wang, 2019]{Ding2019}
Ding, J. and Wang, Z. (2019).
\newblock A modified monte carlo approach to the approximation of invariant
  measure.
\newblock {\em Integral Methods in Science and Engineering}, pages 125--130.

\bibitem[Ding and Zhou, 1994]{Ding1994}
Ding, J. and Zhou, A. (1994).
\newblock The projection method for computing multidimensional absolutely
  continuous invariant measures.
\newblock {\em Journal of Statistical Physics}, 77(3-4):899--908.

\bibitem[Ding and Zhou, 1995]{Ding1995}
Ding, J. and Zhou, A.~H. (1995).
\newblock Piecewise linear {M}arkov approximations of {F}robenius-{P}erron
  operators associated with multi-dimensional transformations.
\newblock {\em Nonlinear Analysis: Theory, Methods \& Applications}, 25(4):399
  -- 408.

\bibitem[Doan, 2009]{Doan2009}
Doan, T.~S. (2009).
\newblock {\em Lyapunov Exponents for Random Dynamical Systems}.
\newblock PhD thesis, Faculty of Mathematics and Natural Science, Technical
  University at Dresden.

\bibitem[Doebeli, 1995]{Doebeli1995}
Doebeli, M. (1995).
\newblock Evolutionary predictions from invariant physical measures of dynamic
  processes.
\newblock {\em Journal of theoretical biology}, 173(4):377--387.

\bibitem[Douc et~al., 2014]{Douc2014}
Douc, R., Moulines, E., and Stoffer, D. (2014).
\newblock {\em Nonlinear time series: Theory, methods and applications with R
  examples}.
\newblock CRC press.

\bibitem[Doughty, 1995]{Christine1995}
Doughty, C.~A. (1995).
\newblock {\em Estimation of Hydrologic Properties of Heterogeneous Geologic
  Media with an Inverse Method Based on Iterated Function Systems}.
\newblock PhD thesis, University of California, Berkeley.

\bibitem[Dubins and Freedman, 1966]{Dubins1966}
Dubins, L.~E. and Freedman, D.~A. (1966).
\newblock Invariant probabilities for certain {M}arkov processes.
\newblock {\em Ann. Math. Statist.}, 37(4):837--848.

\bibitem[Dudley, 1976]{Dudley1976}
Dudley, R. (1976).
\newblock {\em Probabilities and metrics: Convergence of laws on metric spaces,
  with a view to statistical testing}.
\newblock Aarhus: Aarhus Universitet, Matematisk Institut.

\bibitem[Duflo, 2013]{Duflo2013}
Duflo, M. (2013).
\newblock {\em Random iterative models}, volume~34.
\newblock Springer Science \& Business Media.

\bibitem[Dumas et~al., 2002]{Dumas2002}
Dumas, V., Guillemin, F., and Robert, P. (2002).
\newblock A {M}arkovian analysis of additive-increase multiplicative-decrease
  algorithms.
\newblock {\em Advances in Applied Probability}, pages 85--111.

\bibitem[Dumitru, 2013]{Dumitru2013}
Dumitru, D. (2013).
\newblock Attractors of infinite iterated function systems containing
  contraction type functions.
\newblock {\em An. {\c{S}}tiin{\c{t}}. Univ. Al. I. Cuza Ia{\c{s}}i, Mat. NS},
  59:281--298.

\bibitem[Dupuis and Ellis, 2011]{Dupuis2011}
Dupuis, P. and Ellis, R.~S. (2011).
\newblock {\em A weak convergence approach to the theory of large deviations},
  volume 902.
\newblock John Wiley \& Sons.

\bibitem[Durmus et~al., 2016]{Durmus2016}
Durmus, A., Fort, G., and Moulines, {\'E}. (2016).
\newblock Subgeometric rates of convergence in {W}asserstein distance for
  {M}arkov chains.
\newblock {\em Annales de l'Institut Henri Poincaré, Probabilités et
  Statistiques}, 52(4):1799 -- 1822.

\bibitem[Dutkay and Jorgensen, 2007a]{Dutkay2007(2)}
Dutkay, D.~E. and Jorgensen, P.~E. (2007a).
\newblock Analysis of orthogonality and of orbits in affine iterated function
  systems.
\newblock {\em Mathematische Zeitschrift}, 256(4):801--823.

\bibitem[Dutkay and Jorgensen, 2007b]{Dutkay2007(1)}
Dutkay, D.~E. and Jorgensen, P.~E. (2007b).
\newblock Fourier frequencies in affine iterated function systems.
\newblock {\em Journal of Functional Analysis}, 247(1):110--137.

\bibitem[Edalat, 1995]{Edalat1995}
Edalat, A. (1995).
\newblock Dynamical systems, measures, and fractals via domain theory.
\newblock {\em Information and Computation}, 120(1):32 -- 48.

\bibitem[Edalat, 1996]{Edalata1996}
Edalat, A. (1996).
\newblock Power domains and iterated function systems.
\newblock {\em Information and Computation}, 124(2):182 -- 197.

\bibitem[Edgar, 2007]{Edgar2007}
Edgar, G. (2007).
\newblock {\em Measure, topology, and fractal geometry}.
\newblock Springer Science \& Business Media.

\bibitem[El~Naschie, 1994]{El1994}
El~Naschie, M. (1994).
\newblock Iterated function systems and the two-slit experiment of quantum
  mechanics.
\newblock {\em Chaos, Solitons \& Fractals}, 4(10):1965--1968.

\bibitem[{El Naschie}, 1994]{Naschie1994}
{El Naschie}, M. (1994).
\newblock Iterated function systems and the two-slit experiment of quantum
  mechanics.
\newblock {\em Chaos, Solitons \& Fractals}, 4(10):1965 -- 1968.

\bibitem[Elton, 1987]{Elton1987}
Elton, J.~H. (1987).
\newblock An ergodic theorem for iterated maps.
\newblock {\em Ergodic Theory and Dynamical Systems}, 7(04):481--488.

\bibitem[Elton, 1990]{Elton1990}
Elton, J.~H. (1990).
\newblock A multiplicative ergodic theorem for {L}ipschitz maps.
\newblock {\em Stochastic Processes and their Applications}, 34(1):39--47.

\bibitem[Elton and Piccioni, 1992]{Elton1992}
Elton, J.~H. and Piccioni, M. (1992).
\newblock Iterated function systems a rising from recursive estimation
  problems.
\newblock {\em Probability Theory and Related Fields}, 91(1):103--114.

\bibitem[Elton and Yan, 1989]{Elton1989}
Elton, J.~H. and Yan, Z. (1989).
\newblock Approximation of measures by {M}arkov processes and homogeneous
  affine iterated function systems.
\newblock In {\em Constructive Approximation}, pages 69--87. Springer.

\bibitem[{Epperlein} and {Mareček}, 2017]{Jakub2017}
{Epperlein}, J. and {Mareček}, J. (2017).
\newblock Resource allocation with population dynamics.
\newblock In {\em 2017 55th Annual Allerton Conference on Communication,
  Control, and Computing (Allerton)}, pages 1293--1300.

\bibitem[Ethier and Kurtz, 1986]{Ethier1986}
Ethier, S.~N. and Kurtz, T.~G. (1986).
\newblock Markov processes: Characterization and convergence. a john wiley \&
  sons.
\newblock {\em Inc., Publication}, 1.

\bibitem[Evans, 1999]{Evans1999}
Evans, L. (1999).
\newblock Partial differential equations and monge-kantorovich mass transfer
  (surveypaper.
\newblock In {\em Current Developments in Mathematics, 1997, International
  Press}.

\bibitem[Ezhumalai et~al., 2021]{Ezhumalai2021}
Ezhumalai, A., Ganesan, N., and Balasubramaniyan, S. (2021).
\newblock An extensive survey on fractal structures using iterated function
  system in patch antennas.
\newblock {\em International Journal of Communication Systems}, 34(15):e4932.

\bibitem[Falconer, 2004]{Falconer2004}
Falconer, K. (2004).
\newblock {\em Fractal geometry: mathematical foundations and applications}.
\newblock John Wiley \& Sons.

\bibitem[Falconer, 1985]{Falconer1985}
Falconer, K.~J. (1985).
\newblock {\em The Geometry of Fractal Sets}.
\newblock Cambridge University Press, USA.

\bibitem[Fatehinia et~al., 2018]{Fatehinia2018}
Fatehinia, M., Molaei, M.~R., and Sayyari, Y. (2018).
\newblock Entropy of iterated function systems and their relations with black
  holes and bohr-like black holes entropies.
\newblock {\em Entropy}, 20.

\bibitem[Fayolle et~al., 1999]{Fayolle1999}
Fayolle, G., Malyshev, V., and Iasnogorodski, R. (1999).
\newblock {\em Random walks in the quarter-plane}, volume~40.
\newblock Springer.

\bibitem[Fedeli, 2005]{Fedeli2005}
Fedeli, A. (2005).
\newblock On chaotic set-valued discrete dynamical systems.
\newblock {\em Chaos, Solitons \& Fractals}, 23(4):1381--1384.

\bibitem[Feller, 2008]{Feller2008}
Feller, W. (2008).
\newblock {\em An introduction to probability theory and its applications, vol
  2}.
\newblock John Wiley \& Sons.

\bibitem[Filippov, 2013]{filippov2013differential}
Filippov, A.~F. (2013).
\newblock {\em Differential equations with discontinuous righthand sides:
  control systems}, volume~18.
\newblock Springer Science \& Business Media.

\bibitem[Fioravanti et~al., 2019]{Fioravanti2019}
Fioravanti, A.~R., Marecek, J., Shorten, R.~N., Souza, M., and Wirth, F.
  (2019).
\newblock On the ergodic control of ensembles.
\newblock {\em Automatica}, 108:108483.

\bibitem[Fioravanti et~al., 2017]{Fioravanti2017}
Fioravanti, A.~R., Mare{\v c}ek, J., Shorten, R.~N., Souza, M., and Wirth,
  F.~R. (2017).
\newblock On classical control and smart cities.
\newblock In {\em 2017 IEEE 56th Annual Conference on Decision and Control
  (CDC)}, pages 1413--1420.

\bibitem[Fisher, 2012]{Fisher2012}
Fisher, Y. (2012).
\newblock {\em Fractal image compression: theory and application}.
\newblock Springer Science \& Business Media.

\bibitem[Foguel, 1966a]{Foguel1966(2)}
Foguel, S. (1966a).
\newblock The ergodic theorem for {M}arkov processes.
\newblock {\em Israel Journal of Mathematics}, 4(1):11--22.

\bibitem[Foguel, 1966b]{Foguel1966(1)}
Foguel, S. (1966b).
\newblock Existence of invariant measures for {M}arkov processes {II}.
\newblock {\em Proceedings of the American Mathematical Society},
  17(2):387--389.

\bibitem[Foguel, 1966c]{Foguel1966(3)}
Foguel, S. (1966c).
\newblock Limit theorems for {M}arkov processes.
\newblock {\em Transactions of the American Mathematical Society},
  121(1):200--209.

\bibitem[Foguel, 1969]{Foguel1969}
Foguel, S.~R. (1969).
\newblock {\em The Ergodic theory of {M}arkov processes, ({V}an {N}ostrand
  mathematical studies 21)}.
\newblock Van Nostrand Reinhold Co (1969).

\bibitem[Forte and Mendivil, 1998]{Mendivil1998(1)}
Forte, B. and Mendivil, F. (1998).
\newblock A classical ergodic property for iterated function systems: a simple
  proof.
\newblock {\em Ergodic Theory and Dynamical Systems}, 18(3):609--611.

\bibitem[Forte and Vrscay, 1995]{Forte1995}
Forte, B. and Vrscay, E. (1995).
\newblock Solving the inverse problem for measures using iterated function
  systems: a new approach.
\newblock {\em Advances in applied probability}, pages 800--820.

\bibitem[Fosson, 2011]{Sophie2011}
Fosson, S.~M. (2011).
\newblock Analysis of a deconvolution algorithm for quantized-input linear
  systems through iterated random functions.
\newblock {\em IFAC Proceedings Volumes}, 44(1):11302--11307.

\bibitem[Frank, 2000]{Hollander2000}
Frank, d.~H. (2000).
\newblock {\em Large Deviations}.
\newblock American Mathematical Society.

\bibitem[Freeland and Durrani, 1991]{Freeland1991}
Freeland, G. and Durrani, T. (1991).
\newblock On the use of general iterated function systems in signal modelling.
\newblock In {\em Acoustics, Speech, and Signal Processing, IEEE International
  Conference on}, pages 3377,3378,3379,3380, Los Alamitos, CA, USA. IEEE
  Computer Society.

\bibitem[Freidlin et~al., 2012]{Freidlin2012}
Freidlin, M., Sz{\"u}cs, J., and Wentzell, A. (2012).
\newblock {\em Random Perturbations of Dynamical Systems}.
\newblock Grundlehren der mathematischen Wissenschaften. Springer.

\bibitem[Friedkin and Johnsen, 2011]{Friedkin2011}
Friedkin, N.~E. and Johnsen, E.~C. (2011).
\newblock {\em Social influence network theory: A sociological examination of
  small group dynamics}, volume~33.
\newblock Cambridge University Press.

\bibitem[Froyland, 1995]{Froyland1995(2)}
Froyland, G. (1995).
\newblock Finite approximation of {S}inai-{B}owen-{R}uelle measures for
  {A}nosov systems in two dimensions.
\newblock {\em Random and Computational Dynamics}, 3(4):251--263.

\bibitem[Froyland, 1997a]{Froyland1997(1)}
Froyland, G. (1997a).
\newblock Computing physical invariant measures.
\newblock In {\em International Symposium on Nonlinear Theory and its
  Applications, Japan, Research Society of Nonlinear Theory and its
  Applications (IEICE)}, pages 1129--1132.

\bibitem[Froyland, 1997b]{Froyland1997(2)}
Froyland, G. (1997b).
\newblock Estimating physical invariant measures and space averages of
  dynamical systems indicators.
\newblock {\em Bulletin of the Australian Mathematical Society},
  56(1):157--159.

\bibitem[Froyland, 1998]{Froyland1998}
Froyland, G. (1998).
\newblock Approximating physical invariant measures of mixing dynamical systems
  in higher dimensions.
\newblock {\em Nonlinear Analysis: Theory, Methods \& Applications}, 32(7):831
  -- 860.

\bibitem[Froyland, 1999]{Froyland1999}
Froyland, G. (1999).
\newblock Ulam's method for random interval maps.
\newblock {\em Nonlinearity}, 12(4):1029.

\bibitem[Froyland, 2007]{Froyland2007}
Froyland, G. (2007).
\newblock On {U}lam approximation of the isolated spectrum and eigenfunctions
  of hyperbolic maps.
\newblock {\em Discrete Contin. Dyn. Syst. Ser. A}, 17(3):203--221.

\bibitem[Froyland and Aihara, 2000]{Froyland2000}
Froyland, G. and Aihara, K. (2000).
\newblock Rigorous numerical estimation of {L}yapunov exponents and invariant
  measures of iterated function systems and random matrix products.
\newblock {\em International Journal of Bifurcation and Chaos},
  10(01):103--122.

\bibitem[Froyland et~al., 1995]{Froyland1995(1)}
Froyland, G., Judd, K., Mees, A.~I., Watson, D., and Murao, K. (1995).
\newblock Constructing invariant measures from data.
\newblock {\em International Journal of Bifurcation and Chaos},
  5(04):1181--1192.

\bibitem[Fuh, 2004]{Fuh2004}
Fuh, C.~D. (2004).
\newblock {\em Iterated random function system: convergence theorems}, pages
  98--111.
\newblock Probability, Finance and Insurance.

\bibitem[Furstenberg and Kesten, 1960]{Furstenberg1960}
Furstenberg, H. and Kesten, H. (1960).
\newblock Products of random matrices.
\newblock {\em The Annals of Mathematical Statistics}, 31(2):457--469.

\bibitem[Galatolo et~al., 2009]{Galatolo2009}
Galatolo, S., Hoyrup, M., and Rojas, C. (2009).
\newblock Dynamics and abstract computability: computing invariant measures.
\newblock {\em arXiv preprint arXiv:0903.2385}.

\bibitem[Galatolo et~al., 2016]{Galatolo2016}
Galatolo, S., Monge, M., and Nisoli, I. (2016).
\newblock Rigorous approximation of stationary measures and convergence to
  equilibrium for iterated function systems.
\newblock {\em Journal of Physics A: Mathematical and Theoretical},
  49(27):274001.

\bibitem[Galatolo and Nisoli, 2014]{Galatolo2014}
Galatolo, S. and Nisoli, I. (2014).
\newblock An elementary approach to rigorous approximation of invariant
  measures.
\newblock {\em SIAM Journal on Applied Dynamical Systems}, 13(2):958--985.

\bibitem[Garg et~al., 2014]{Ankit2014}
Garg, A., Negi, A., and Agrawal, A. (2014).
\newblock Geometric modelling of complex objects using iterated function
  system.
\newblock {\em International Journal of Scientific \& Technology Research}.

\bibitem[Gaspard, 2017]{Gaspard2017}
Gaspard, P. (2017).
\newblock Iterated function systems for {DNA} replication.
\newblock {\em Phys. Rev. E}, 96:042403.

\bibitem[Georgescu et~al., 2020]{Georgescu2020}
Georgescu, F., Miculescu, R., and Mihail, A. (2020).
\newblock Hardy--rogers type iterated function systems.
\newblock {\em Qualitative Theory of Dynamical Systems}, 19(1):1--13.

\bibitem[Ghosh et~al., 2022]{Ramen2021}
Ghosh, R., Mare{\v{c}}ek, J., Griggs, W.~M., Souza, M., and Shorten, R.~N.
  (2022).
\newblock Predictability and fairness in social sensing.
\newblock {\em {IEEE} Internet Things J.}, 9(1):37--54.

\bibitem[Ghosh et~al., 2019]{Ramen2019}
Ghosh, R., Marecek, J., and Shorten, R. (2019).
\newblock Iterated piecewise-stationary random functions.
\newblock {\em arXiv preprint arXiv:1909.10093}.

\bibitem[Gibbs, 2000]{Gibbs2000}
Gibbs, A.~L. (2000).
\newblock {\em Convergence of {M}arkov Chain {M}onte {C}arlo algorithms with
  applications to image restoration}.
\newblock Citeseer.

\bibitem[Gibbs and Su, 2002]{Gibbs2002}
Gibbs, A.~L. and Su, F.~E. (2002).
\newblock On choosing and bounding probability metrics.
\newblock {\em International Statistical Review / Revue Internationale de
  Statistique}, 70(3):419--435.

\bibitem[Giles, 1990]{Giles1990}
Giles, P.~A. (1990).
\newblock {\em Iterated function systems and shape representation}.
\newblock PhD thesis, Durham University.

\bibitem[Gleick, 2011]{Gleick2011}
Gleick, J. (2011).
\newblock {\em Chaos: Making a new science}.
\newblock Open Road Media.

\bibitem[Godland, 2022]{Godland2022}
Godland, P. (2022).
\newblock {\em Markov Renewal Theory in the Analysis of Random Strings and
  Iterated Function Systems}.
\newblock PhD thesis, Westfaelische Wilhelms-Universitaet Muenster.
\newblock ProQuest Dissertations Publishing id 28946730.

\bibitem[G{\'o}ra and Boyarsky, 2003]{Gora2003}
G{\'o}ra, P. and Boyarsky, A. (2003).
\newblock Absolutely continuous invariant measures for random maps with
  position dependent probabilities.
\newblock {\em Journal of mathematical analysis and applications},
  278(1):225--242.

\bibitem[G{\'o}ra et~al., 2006]{Gora2006}
G{\'o}ra, P., Boyarsky, A., Islam, M.~S., and Bahsoun, W. (2006).
\newblock Absolutely continuous invariant measures that cannot be observed
  experimentally.
\newblock {\em SIAM Journal on Applied Dynamical Systems}, 5(1):84--90.

\bibitem[Gordin and Lifsic, 1978]{Gordin1978}
Gordin, M.~I. and Lifsic, A. (1978).
\newblock The central limit theorem for stationary {M}arkov processes.
\newblock In {\em Doklady Akademii Nauk}, volume 239(04), pages 766--767.
  Russian Academy of Sciences.

\bibitem[Goyal and Prasad, 2021]{Goyal2021}
Goyal, K. and Prasad, B. (2021).
\newblock Generalized iterated function systems in multi-valued mapping.
\newblock In {\em AIP Conference Proceedings}, volume 2316-1, page 040001. AIP
  Publishing LLC.

\bibitem[Graczyk and Swiatek, 1997]{Graczyk1997}
Graczyk, J. and Swiatek, G. (1997).
\newblock Generic hyperbolicity in the logistic family.
\newblock {\em Annals of mathematics}, 146(1):1--52.

\bibitem[Gravereaux, 1988]{Gravereaux1988}
Gravereaux, J.~B. (1988).
\newblock Calcul de malliavin et probabilité invariante d'une chaîne de
  markov.
\newblock {\em Annales de l'I.H.P. Probabilités et statistiques},
  24(2):159--188.

\bibitem[Guivarc'h and Raugi, 1985]{Guivarc1985}
Guivarc'h, Y. and Raugi, A. (1985).
\newblock Frontiere de furstenberg, propri{\'e}t{\'e}s de contraction et
  th{\'e}oremes de convergence.
\newblock {\em Zeitschrift f{\"u}r Wahrscheinlichkeitstheorie und Verwandte
  Gebiete}, 69(2):187--242.

\bibitem[Guzik, 2016a]{Guzik2016(2)}
Guzik, G. (2016a).
\newblock On construction of asymptotically stable iterated function system
  with probabilities.
\newblock {\em Stochastic Analysis and Applications}, 34(1):24--37.

\bibitem[Guzik, 2016b]{Guzik2016(1)}
Guzik, G. (2016b).
\newblock Semiattractors of set-valued semiflows.
\newblock {\em Journal of Mathematical Analysis and Applications},
  435(2):1321--1334.

\bibitem[Gw{\'o}{\'z}d{\'z}-Lukawska and Jachymski, 2005]{Gwozdz2005}
Gw{\'o}{\'z}d{\'z}-Lukawska, G. and Jachymski, J. (2005).
\newblock The hutchinson-barnsley theory for infinite iterated function
  systems.
\newblock {\em Bulletin of the Australian Mathematical Society},
  72(3):441--454.

\bibitem[H{\"a}ggstr{\"o}m, 2005]{Haggstrom2005}
H{\"a}ggstr{\"o}m, O. (2005).
\newblock On the central limit theorem for geometrically ergodic {M}arkov
  chains.
\newblock {\em Probability Theory and Related Fields}, 132(1):74--82.

\bibitem[Haggstrom and Rosenthal, 2007]{Haggstrom2007}
Haggstrom, O. and Rosenthal, J. (2007).
\newblock On variance conditions for {M}arkov chain {CLTs}.
\newblock {\em Electron. Commun. Probab.}, 12:454--464.

\bibitem[Hairer, 2002]{hairer2002exponential}
Hairer, M. (2002).
\newblock Exponential mixing properties of stochastic pdes through asymptotic
  coupling.
\newblock {\em Probability theory and related fields}, 124(3):345--380.

\bibitem[Hairer, 2006a]{hairer2006coupling}
Hairer, M. (2006a).
\newblock Coupling stochastic pdes.
\newblock In {\em XIVth International Congress on Mathematical Physics}, pages
  281--289. World Scientific.

\bibitem[Hairer, 2006b]{Hairer2006}
Hairer, M. (2006b).
\newblock Ergodic properties of {M}arkov processes.
\newblock {\em Lecture notes}.

\bibitem[Hairer et~al., 2011]{Hairer2011}
Hairer, M., Mattingly, J.~C., and Scheutzow, M. (2011).
\newblock Asymptotic coupling and a general form of {H}arris’ theorem with
  applications to stochastic delay equations.
\newblock {\em Probability theory and related fields}, 149(1-2):223--259.

\bibitem[Hanin, 1992]{Hanin1992}
Hanin, L.~G. (1992).
\newblock Kantorovich-{R}ubinstein norm and its application in the theory of
  {L}ipschitz spaces.
\newblock {\em Proceedings of the American Mathematical Society},
  115(2):345--352.

\bibitem[Hanin, 1999]{Hanin1999}
Hanin, L.~G. (1999).
\newblock An extension of the {K}antorovich norm.
\newblock {\em Contemporary Mathematics}, 226:113--130.

\bibitem[Hanus and Urba{\'n}ski, 1998]{Hanus1998}
Hanus, P. and Urba{\'n}ski, M. (1998).
\newblock Rigidity of infinite one-dimensional iterated function systems.
\newblock {\em Real Analysis Exchange}, pages 275--287.

\bibitem[Hanus, 2000]{Hanus2000}
Hanus, P.~G. (2000).
\newblock {\em Examples and applications of infinite iterated function
  systems}.
\newblock PhD thesis, University of North Texas.

\bibitem[Harris, 1956]{Harris1956}
Harris, T. (1956).
\newblock The existence of stationary measures for certain {M}arkov processes.
\newblock {\em Proceedings of the Third {B}arkeley Symposium on Mathematical
  Statistics and Probability}, II:113--124.

\bibitem[Harris et~al., 1955]{Harris1955}
Harris, T.~E. et~al. (1955).
\newblock On chains of infinite order.
\newblock {\em Pacific Journal of Mathematics}, 5(Suppl. 1):707--724.

\bibitem[Hennion, 1993]{Hennion1993}
Hennion, H. (1993).
\newblock Spectral dcomposition of {D}oeblin-{F}ortet operators.
\newblock In {\em Proc. Bitter. Math. Soc}, volume~69, pages 627--634.

\bibitem[Hennion and Herv{\'e}, 2001]{Hennion2001}
Hennion, H. and Herv{\'e}, L. (2001).
\newblock {\em Limit theorems for {M}arkov chains and stochastic properties of
  dynamical systems by quasi-compactness}, volume 1766.
\newblock Springer Science \& Business Media.

\bibitem[Hennion and Herve, 2004]{Hennion2004}
Hennion, H. and Herve, L. (2004).
\newblock Central limit theorems for iterated random {L}ipschitz mappings.
\newblock {\em Ann. Probab.}, 32(3):1934--1984.

\bibitem[Hepting et~al., 1991]{Hepting1991(1)}
Hepting, D., Prusinkiewicz, P., and Saupe, D. (1991).
\newblock {\em Rendering methods for iterated function systems}.
\newblock North-Holland.

\bibitem[Hepting, 1991]{Hepting1991(2)}
Hepting, D.~H. (1991).
\newblock Approximation and visualization of sets defined by iterated function
  systems.
\newblock {\em Citeseer}.

\bibitem[Herkenrath and Iosifescu, 2007]{Herkenrath2007}
Herkenrath, U. and Iosifescu, M. (2007).
\newblock On a contractibility condition for iterated random functions.
\newblock {\em Revue Roumaine de Mathematiques Pures et Appliquees},
  52(5):563--572.

\bibitem[Hermer et~al., 2019]{Hermer2019}
Hermer, N., Luke, D.~R., and Sturm, A. (2019).
\newblock Random function iterations for consistent stochastic feasibility.
\newblock {\em Numerical Functional Analysis and Optimization}, 40(4):386--420.

\bibitem[Hermer et~al., 2020]{Hermer2020}
Hermer, N., Luke, D.~R., and Sturm, A. (2020).
\newblock Random function iterations for stochastic fixed point problems.
\newblock {\em arXiv: Functional Analysis}.

\bibitem[Hermer et~al., 2022]{Hermer2022}
Hermer, N., Luke, D.~R., and Sturm, A. (2022).
\newblock Nonexpansive {M}arkov operators and random function iterations for
  stochastic fixed point problems.
\newblock {\em arXiv}.

\bibitem[Hille and Snigireva, 2012]{Hille2012}
Hille, M. and Snigireva, N. (2012).
\newblock Teichm{\"u}ller space for iterated function systems.
\newblock {\em Conformal Geometry and Dynamics of the American Mathematical
  Society}, 16(8):132--160.

\bibitem[Hille et~al., 2016]{Hille2016}
Hille, S., Horbacz, K., Szarek, T., and Wojew{\'o}dka, H. (2016).
\newblock Limit theorems for some {M}arkov chains.
\newblock {\em Journal of Mathematical Analysis and Applications},
  443(1):385--408.

\bibitem[Hinderer, 1970]{Hinderer1970}
Hinderer, K. (1970).
\newblock {\em Foundations of Non-stationary Dynamic Programming with Discrete
  Time Parameter}.
\newblock Springer-Verlag Berlin Heidelberg.

\bibitem[Horbacz, 2016]{Horbacz2016(2)}
Horbacz, K. (2016).
\newblock The central limit theorem for random dynamical systems.
\newblock {\em Journal of Statistical Physics}, 164(6):1261--1291.

\bibitem[Horbacz and Szarek, 2001]{Szarek2001}
Horbacz, K. and Szarek, T. (2001).
\newblock Continuous iterated function systems on polish spaces.
\newblock {\em Bulletin of the Polish Academy of Sciences, Mathematics}, 49.

\bibitem[Horbacz and Ślęczka, 2016]{Horbacz2016(1)}
Horbacz, K. and Ślęczka, M. (2016).
\newblock Law of large numbers for random dynamical systems.
\newblock {\em Journal of Statistical Physics}, 162(3):671--684.

\bibitem[Hoskins, 1995]{Hoskins1995}
Hoskins, D.~A. (1995).
\newblock An iterated function systems approach to emergence.
\newblock In {\em Evolutionary Programming}, pages 673--692. Citeseer.

\bibitem[Hubbard, 2016]{Hubbard2016}
Hubbard, J.~H. (2016).
\newblock {\em Teichm{\"u}ller theory and applications to geometry, topology,
  and dynamics}, volume~2.
\newblock Matrix Editions.

\bibitem[Hunt, 1996a]{Hunt1996(2)}
Hunt, B.~R. (1996a).
\newblock Estimating invariant measures and {L}yapunov exponents.
\newblock {\em Ergodic Theory and Dynamical Systems}, 16(4):735--750.

\bibitem[Hunt, 1996b]{Hunt1996(1)}
Hunt, F.~Y. (1996b).
\newblock Approximating the invariant measures of randomly perturbed
  dissipative maps.
\newblock {\em Journal of mathematical analysis and applications},
  198(2):534--551.

\bibitem[Hunt and Miller, 1992]{Hunt1992}
Hunt, F.~Y. and Miller, W.~M. (1992).
\newblock On the approximation of invariant measures.
\newblock {\em Journal of statistical physics}, 66(1-2):535--548.

\bibitem[Hutchinson, 1981]{Hutchinson1981}
Hutchinson, J.~E. (1981).
\newblock Fractals and self similarity.
\newblock {\em Indiana University Mathematics Journal}, 30(5):713--747.

\bibitem[Hyong-Chol et~al., 2005]{Hyong2005}
Hyong-Chol, O., Yong-hwa, R., and Won-gun, K. (2005).
\newblock Ergodic theorem for infinite iterated function systems.
\newblock {\em Applied Mathematics and Mechanics}, 26(4):465--469.

\bibitem[Imkeller and Kloeden, 2003]{Imkeller2003}
Imkeller, P. and Kloeden, P. (2003).
\newblock On the computation of invariant measures in random dynamical systems.
\newblock {\em Stochastics and Dynamics}, 3(02):247--265.

\bibitem[Iosifescu, 1963]{Iosifescu1963}
Iosifescu, M. (1963).
\newblock Random systems with complete connections with an arbitrary set of
  states.
\newblock {\em Rev. Math. Pure Appls}, 9:611--645.

\bibitem[Iosifescu, 2003]{Iosifescu2003}
Iosifescu, M. (2003).
\newblock A simple proof of a basic theorem on iterated random functions.
\newblock {\em Proc. Rom. Acad. Ser. A Math. Phys. Tech. Sci. Inf. Sci},
  4(3):167--174.

\bibitem[Iosifescu, 2009]{Iosifescu2009}
Iosifescu, M. (2009).
\newblock Iterated function systems: A critical survey.
\newblock {\em Math. Reports}, 11(3):181--229.

\bibitem[Iosifescu and Grigorescu, 1990]{Iosifescu1990}
Iosifescu, M. and Grigorescu, S. (1990).
\newblock {\em Dependence with complete connections and its applications}.
\newblock Cambridge University Press.

\bibitem[Iovane, 2006a]{Iovane2006(2)}
Iovane, G. (2006a).
\newblock Cantorian space--time and hilbert space: Part {II}—relevant
  consequences.
\newblock {\em Chaos, Solitons \& Fractals}, 29(1):1--22.

\bibitem[Iovane, 2006b]{Iovane2006(1)}
Iovane, G. (2006b).
\newblock Cantorian spacetime and hilbert space: Part {I}—foundations.
\newblock {\em Chaos, Solitons \& Fractals}, 28(4):857--878.

\bibitem[Ippei, 2011]{Ippei2011}
Ippei, O. (2011).
\newblock Computer-assisted verification method for invariant densities and
  rates of decay of correlations.
\newblock {\em SIAM Journal on Applied Dynamical Systems}, 10(2):788--816.

\bibitem[Isaac, 1962]{Isaac1962}
Isaac, R. (1962).
\newblock Markov processes and unique stationary probability measures.
\newblock {\em Pacific J. Math.}, 12(1):273--286.

\bibitem[Islam and Chandler, 2015]{Islam2015}
Islam, M.~S. and Chandler, S. (2015).
\newblock Approximation by absolutely continuous invariant measures of iterated
  function systems with place-dependent probabilities.
\newblock {\em Fractals}, 23(04):1550038.

\bibitem[Jacquin, 1994]{Jacquin1994}
Jacquin, A. (1994).
\newblock An introduction to fractals and their applications in electrical
  engineering.
\newblock {\em Journal of the Franklin Institute}, 331(6):659--680.

\bibitem[Jacquin, 1990a]{Jacquin1990(1)}
Jacquin, A.~E. (1990a).
\newblock Fractal image coding based on a theory of iterated contractive image
  transformations.
\newblock In {\em Visual Communications and Image Processing'90: Fifth in a
  Series}, volume 1360, pages 227--239. International Society for Optics and
  Photonics.

\bibitem[Jacquin, 1990b]{Jacquin1990(2)}
Jacquin, A.~E. (1990b).
\newblock A fractal theory of iterated {M}arkov operators with applications to
  digital image coding.
\newblock {\em Ph. D Thesis, Georgia Tech., 1989}.

\bibitem[Jacquin, 1990c]{Jacquin1990(3)}
Jacquin, A.~E. (1990c).
\newblock A novel fractal block-coding technique for digital images.
\newblock In {\em International Conference on Acoustics, Speech, and Signal
  Processing}, pages 2225--2228. IEEE.

\bibitem[Jacquin et~al., 1992]{Jacquin1992}
Jacquin, A.~E. et~al. (1992).
\newblock Image coding based on a fractal theory of iterated contractive image
  transformations.
\newblock {\em IEEE Transactions on image processing}, 1(1):18--30.

\bibitem[Jadczyk, 2004]{Arkadiusz2004}
Jadczyk, A. (2004).
\newblock On quantum iterated function systems.
\newblock {\em Open Physics}, 2(3):492 -- 503.

\bibitem[Jaerisch and Kesseb{\"o}hmer, 2011]{Jaerisch2011}
Jaerisch, J. and Kesseb{\"o}hmer, M. (2011).
\newblock Regularity of multifractal spectra of conformal iterated function
  systems.
\newblock {\em Transactions of the American Mathematical Society},
  363(1):313--330.

\bibitem[Janoska, 1995]{Janoska1995}
Janoska, P. (1995).
\newblock Stability of a countable iterated function system.
\newblock {\em Univ. Iagell. Acta Math}, 32:105--119.

\bibitem[Jarner and Tweedie, 2000]{Jarner2000}
Jarner, S. and Tweedie, R. (2000).
\newblock Stability properties of {M}arkov chains defined via iterated random
  functions.
\newblock {\em Preprint}.

\bibitem[Jarner and Tweedie, 2001]{Jarner2001}
Jarner, S. and Tweedie, R. (2001).
\newblock Locally contracting iterated functions and stability of {M}arkov
  chains.
\newblock {\em Journal of applied probability}, pages 494--507.

\bibitem[Jaroszewska, 2002]{Joanna2002}
Jaroszewska, J. (2002).
\newblock Iterated function systems with continuous place dependent
  probabilities.
\newblock {\em Zeszyty Naukowe Uniwersytetu Jagiello{\'n}skiego. Universitatis
  Iagellonicae Acta Mathematica}, 1258:137--146.

\bibitem[Jaroszewska, 2008]{Joanna2008}
Jaroszewska, J. (2008).
\newblock How to construct asymptotically stable iterated function systems.
\newblock {\em Statistics \& probability letters}, 78(12):1570--1576.

\bibitem[Jaroszewska, 2013]{Joanna2013}
Jaroszewska, J. (2013).
\newblock A note on iterated function systems with discontinuous probabilities.
\newblock {\em Chaos, Solitons \& Fractals}, 49:28--31.

\bibitem[Johansson and {\"O}berg, 2003]{Johansson2003}
Johansson, A. and {\"O}berg, A. (2003).
\newblock Square summability of variations of $g$-functions and uniqueness of
  $g$-measures.
\newblock {\em Mathematical Research Letters}, 10(5):587--601.

\bibitem[John, 2007]{Chand2007}
John, C.~T. (2007).
\newblock All {B}ezier curves are attractors of iterated function systems.
\newblock {\em New York J. Math}, 13:107--115.

\bibitem[Jones, 2001]{Jones2001}
Jones, H. (2001).
\newblock Iterated function systems for object generation and rendering.
\newblock {\em International Journal of Bifurcation and Chaos},
  11(02):259--289.

\bibitem[Jorgensen, 2006]{Jorgensen2006}
Jorgensen, P.~E. (2006).
\newblock {\em Analysis and probability: wavelets, signals, fractals}, volume
  234.
\newblock Springer Science \& Business Media.

\bibitem[J{\o}rgensen et~al., 2011]{Jorgensen2011}
J{\o}rgensen, P.~E., Kornelson, K.~A., and Shuman, K.~L. (2011).
\newblock {\em Iterated function systems, moments, and transformations of
  infinite matrices}.
\newblock American Mathematical Soc.

\bibitem[Joseph and Sasikumar, 2006]{Joseph2006}
Joseph, J. and Sasikumar, R. (2006).
\newblock Chaos game representation for comparison of whole genomes.
\newblock {\em BMC bioinformatics}, 7(1):243.

\bibitem[Jozsef, 2013]{Vass2013}
Jozsef, V. (2013).
\newblock {\em On the Geometry of IFS Fractals and its Applications}.
\newblock PhD thesis, School of Applied Mathematics, University of Waterloo.

\bibitem[Jurgens and Crutchfield, 2020]{Jurgens2020}
Jurgens, A.~M. and Crutchfield, J.~P. (2020).
\newblock Shannon {E}ntropy {R}ate of {H}idden {M}arkov processes.
\newblock {\em arXiv preprint arXiv:2008.12886}.

\bibitem[K{\"a}enm{\"a}ki, 2003]{Kaenmaki2003}
K{\"a}enm{\"a}ki, A. (2003).
\newblock On the geometric structure of the limit set of conformal iterated
  function systems.
\newblock {\em Publicacions Matem{\`a}tiques}, pages 133--141.

\bibitem[Kaijser, 1978]{Kaijser1978}
Kaijser, T. (1978).
\newblock A limit theorem for {M}arkov chains in compact metric spaces with
  applications to products of random matrices.
\newblock {\em Duke Mathematical Journal}, 45(2):311--349.

\bibitem[Kaijser, 1979]{Kaijser1979}
Kaijser, T. (1979).
\newblock Another central limit theorem for random systems with complete
  connections.
\newblock {\em Rev. Roumaine Math. Pures Appl}, 24:383–412.

\bibitem[Kaijser, 1981]{Kaijser1981}
Kaijser, T. (1981).
\newblock On a new contraction condition for random systems with complete
  connections.
\newblock {\em Rev. Roumaine Math. Pures Appl.}, 26:1075--1117.

\bibitem[Kaijser, 1994]{Kaijser1994}
Kaijser, T. (1994).
\newblock On a theorem of {K}arlin.
\newblock {\em Acta Applicandae Mathematica}, 34(1):51--69.

\bibitem[Kaijser, 2017]{Kaijser2015}
Kaijser, T. (2017).
\newblock A contraction theorem for markov chains on general state spaces.
\newblock {\em Rev. Roumaine Math. Pures Appl.}, 62(2):355--370.

\bibitem[Kapica, 2003]{Kapica2003}
Kapica, R. (2003).
\newblock {\em The Iterates of Random-Valued Vector Functions and the Equation
  Connected With Them}.
\newblock PhD thesis, Uniwersytet Śląski w Katowicach.

\bibitem[Kapica, 2016]{Kapica2016}
Kapica, R. (2016).
\newblock Random iteration and {M}arkov operators.
\newblock {\em Journal of Difference Equations and Applications},
  22(2):295--305.

\bibitem[Kapica and Ślęczka, 2018]{Kapica2018}
Kapica, R. and Ślęczka, M. (2018).
\newblock Law of large numbers for random iteration.
\newblock {\em Journal of Difference Equations and Applications},
  24(5):736--745.

\bibitem[Karlin et~al., 1953]{Karlin1953}
Karlin, S. et~al. (1953).
\newblock Some random walks arising in learning models.
\newblock {\em Pacific Journal of Mathematics}, 3(4):725--756.

\bibitem[Keane, 1972]{Keane1972}
Keane, M. (1972).
\newblock Strongly mixing g-measures.
\newblock {\em Inventiones mathematicae}, 16(4):309--324.

\bibitem[Keane et~al., 1998]{Keane1998}
Keane, M., Murray, R., and Young, L.-S. (1998).
\newblock Computing invariant measures for expanding circle maps.
\newblock {\em Nonlinearity}, 11(1):27--46.

\bibitem[Keller, 1998]{Keller1998}
Keller, G. (1998).
\newblock {\em Equilibrium States in Ergodic Theory}.
\newblock London Mathematical Society Student Texts. Cambridge University
  Press.

\bibitem[Kesseb{\"o}hmer and Stratmann, 2006]{Kessebohmer2006}
Kesseb{\"o}hmer, M. and Stratmann, B.~O. (2006).
\newblock Refined measurable rigidity and flexibility for conformal iterated
  function systems.
\newblock {\em arXiv preprint math/0603571}.

\bibitem[Kifer, 1995]{Kifer1995}
Kifer, Y. (1995).
\newblock Fractals via random iterated function systems and random geometric
  constructions.
\newblock In {\em Fractal Geometry and Stochastics}, pages 145--164. Springer.

\bibitem[Kifer, 2012]{Kifer2012}
Kifer, Y. (2012).
\newblock {\em Ergodic theory of random transformations}, volume~10.
\newblock Springer Science \& Business Media.

\bibitem[King and Shorten, 2006]{King2006}
King, C. and Shorten, R. (2006).
\newblock {AIMD} systems with generalised capacity constraints.
\newblock In {\em Proc. UKACC International Control Conference, Glasgow,
  Scotland}, pages 147--149.

\bibitem[Klebaner, 1997]{Klebaner1997}
Klebaner, F. (1997).
\newblock Population and density dependent branching processes.
\newblock In {\em Classical and modern branching processes}, pages 165--169.
  Springer.

\bibitem[Kolen, 1994]{Kolen1994}
Kolen, J.~F. (1994).
\newblock Recurrent networks: State machines or iterated function systems.
\newblock In {\em Proceedings of the 1993 Connectionist Models Summer School},
  pages 203--210. Hillsdale NJ.

\bibitem[Kouzani, 2008]{Kouzani2008}
Kouzani, A.~Z. (2008).
\newblock Classification of face images using local iterated function systems.
\newblock {\em Machine Vision and Applications}, 19(4):223--248.

\bibitem[Kravchenko, 2006]{Kravchenko2006}
Kravchenko, A.~S. (2006).
\newblock Completeness of the space of separable measures in the
  {K}antorovich-{R}ubinshtein metric.
\newblock {\em Siberian Mathematical Journal}, 47(1):68--76.

\bibitem[Krzysztof, 2015]{Loskot2015}
Krzysztof, L. (2015).
\newblock Random iteration for infinite non-expansive iterated function
  systems.
\newblock {\em Chaos: An Interdisciplinary Journal of Nonlinear Science},
  25(8):083117.

\bibitem[Kungurtsev et~al., 2021]{Slava2021}
Kungurtsev, V., Marecek, J., and Shorten, R. (2021).
\newblock Stochastic model predictive control as an iterated function system.
\newblock {\em arXiv preprint arXiv:2112.08138}.

\bibitem[Kunze et~al., 2012]{Kunze2012}
Kunze, H., La~Torre, D., Mendivil, F., and Vrscay, E. (2012).
\newblock {\em Fractal-Based Methods in Analysis}.
\newblock Springer US.

\bibitem[Kunze et~al., 2011]{kunze2011fractal}
Kunze, H., La~Torre, D., Mendivil, F., and Vrscay, E.~R. (2011).
\newblock {\em Fractal-based methods in analysis}.
\newblock Springer Science \& Business Media.

\bibitem[Kunze et~al., 2008]{Kunze2008}
Kunze, H., La~Torre, D., Vrscay, E.~R., et~al. (2008).
\newblock From iterated function systems to iterated multifunction systems.
\newblock {\em Commun Appl Nonlinear Anal}, 15(4):1--15.

\bibitem[Kwiecinska and Slomczynski, 2000]{Kwiecinska2000}
Kwiecinska, A.~A. and Slomczynski, W. (2000).
\newblock Random dynamical systems arising from iterated function systems with
  place-dependent probabilities.
\newblock {\em Statistics \& probability letters}, 50(4):401--407.

\bibitem[La~Torre et~al., 2018]{Mendivil2018}
La~Torre, D., Maki, E., Mendivil, F., and Vrscay, E. (2018).
\newblock Iterated function systems with place-dependent probabilities and the
  inverse problem of measure approximation using moments.
\newblock {\em Fractals}, 26(05):1850076.

\bibitem[La~Torre and Mendivil, 2008]{Torre2008}
La~Torre, D. and Mendivil, F. (2008).
\newblock Iterated function systems on multifunctions and inverse problems.
\newblock {\em Journal of Mathematical Analysis and Applications},
  340(2):1469--1479.

\bibitem[Lacroix, 2000]{Lacroix2000}
Lacroix, Y. (2000).
\newblock A note on {W}ea$k^{*}$ perturbations of {g}-{M}easures.
\newblock {\em Sankhya: The Indian Journal of Statistics, Series A}, pages
  331--338.

\bibitem[Lakshmivarahan, 2012]{Lakshmivarahan2012}
Lakshmivarahan, S. (2012).
\newblock {\em Learning algorithms theory and applications: Theory and
  Applications}.
\newblock Springer Science \& Business Media.

\bibitem[Lankhorst, 1996]{Lankhorst1996}
Lankhorst, M. (1996).
\newblock Iterated function systems optimization with genetic algorithms.
\newblock {\em Computing Science Report CS-R 9501}.

\bibitem[Lasota, 1995]{Lasota1995}
Lasota, A. (1995).
\newblock From fractals to stochastic differential equations.
\newblock In Garbaczewski, P., Wolf, M., and Weron, A., editors, {\em Chaos:
  The Interplay Between Stochastic and Deterministic Behaviour}, pages
  235--255, Berlin, Heidelberg. Springer Berlin Heidelberg.

\bibitem[Lasota and Mackey, 1984]{Lasota1984}
Lasota, A. and Mackey, M.~C. (1984).
\newblock Globally asymptotic properties of proliferating cell populations.
\newblock {\em Journal of mathematical biology}, 19(1):43--62.

\bibitem[Lasota and Mackey, 1998]{Lasota1998}
Lasota, A. and Mackey, M.~C. (1998).
\newblock Chaos, fractals, and noise: Stochastic aspects of dynamics.
\newblock In {\em Chaos, Fractals, and Noise}.

\bibitem[Lasota and Mackey, 1999]{Lasota1999}
Lasota, A. and Mackey, M.~C. (1999).
\newblock Cell division and the stability of cellular populations.
\newblock {\em Journal of Mathematical Biology}, 38(3):241--261.

\bibitem[Lasota et~al., 1992]{Lasota1992}
Lasota, A., Mackey, M.~C., and Tyrcha, J. (1992).
\newblock The statistical dynamics of recurrent biological events.
\newblock {\em Journal of Mathematical Biology}, 30(8):775--800.

\bibitem[Lasota and Yorke, 1973]{Lasota1973}
Lasota, A. and Yorke, J.~A. (1973).
\newblock On the existence of invariant measures for piecewise monotonic
  transformations.
\newblock {\em Transactions of the American Mathematical Society},
  186:481--488.

\bibitem[Lasota and Yorke, 1994]{Lasota1994(1)}
Lasota, A. and Yorke, J.~A. (1994).
\newblock Lower bound technique for {M}arkov operators and iterated function
  systems.
\newblock {\em Random Comput. Dynam}, 2(1):41--77.

\bibitem[Lau et~al., 2009]{Lau2009}
Lau, K.-S., Ngai, S.-M., and Wang, X.-Y. (2009).
\newblock Separation conditions for conformal iterated function systems.
\newblock {\em Monatshefte f{\"u}r Mathematik}, 156(4):325--355.

\bibitem[Leguesdron, 1989]{Leguesdron1989}
Leguesdron, J.~P. (1989).
\newblock Marche aléatoire sur le semi-groupe des contractions de rd. cas de
  la marche aléatoire sur r+ avec choc élastique en zéro.
\newblock {\em Annales de l'I.H.P. Probabilités et statistiques},
  25(4):483--502.

\bibitem[Lesniak, 2004]{Lesniak2004}
Lesniak, K. (2004).
\newblock Infinite iterated function systems: a multivalued approach.
\newblock {\em Bull. Pol. Acad. Sci. Math}, 52(1):1--8.

\bibitem[Le{\'s}niak et~al., 2021]{Lesniak2021}
Le{\'s}niak, K., Snigireva, N., and Strobin, F. (2021).
\newblock A fractal triangle arising in the {AIMD} dynamics.
\newblock {\em Proceedings of the conference Contemporary Mathematics in Kielce
  2020, February 24-27 2021}, pages 179--194.

\bibitem[Le{\'s}niak et~al., 2022]{Lesniak2022}
Le{\'s}niak, K., Snigireva, N., and Strobin, F. (2022).
\newblock Rate of convergence in the disjunctive chaos game algorithm.
\newblock {\em Chaos: An Interdisciplinary Journal of Nonlinear Science},
  32(1):013110.

\bibitem[Letac et~al., 1986]{Letac1986}
Letac, G. et~al. (1986).
\newblock A contraction principle for certain {M}arkov chains and its
  applications.
\newblock {\em Contemp. Math}, 50:263--273.

\bibitem[Levin et~al., 2006]{Levin2006}
Levin, D.~A., Peres, Y., and Wilmer, E.~L. (2006).
\newblock {\em Markov chains and mixing times}.
\newblock American Mathematical Society.

\bibitem[Leśniak et~al., 2020]{Krzysztof2020}
Leśniak, K., Snigireva, N., and Strobin, F. (2020).
\newblock Weakly contractive iterated function systems and beyond: a manual.
\newblock {\em Journal of Difference Equations and Applications},
  26(8):1114--1173.

\bibitem[Liberzon and Morse, 1999]{liberzon1999basic}
Liberzon, D. and Morse, A.~S. (1999).
\newblock Basic problems in stability and design of switched systems.
\newblock {\em IEEE control systems magazine}, 19(5):59--70.

\bibitem[Lindsay and Mauldin, 2002]{Lindsay2002}
Lindsay, L. and Mauldin, R. (2002).
\newblock Quantization dimension for conformal iterated function systems.
\newblock {\em Nonlinearity}, 15(1):189--199.

\bibitem[Liverani, 2001]{Liverani2001}
Liverani, C. (2001).
\newblock Rigorous numerical investigation of the statistical properties of
  piecewise expanding maps. a feasibility study.
\newblock {\em Nonlinearity}, 14(3):463--490.

\bibitem[Llorens-Fuster et~al., 2009]{Llorens2009}
Llorens-Fuster, E., Petru{\c{s}}el, A., and Yao, J.-C. (2009).
\newblock Iterated function systems and well-posedness.
\newblock {\em Chaos, Solitons \& Fractals}, 41(4):1561--1568.

\bibitem[Loskot and Rudnicki, 1995]{Loskot1995}
Loskot, K. and Rudnicki, R. (1995).
\newblock Limit theorems for stochastically perturbed dynamical systems.
\newblock {\em Journal of Applied Probability}, 32(2):459--469.

\bibitem[\L{}ozi\ifmmode~\acute{n}\else \'{n}\fi{}ski et~al., 2003]{Karol2003}
\L{}ozi\ifmmode~\acute{n}\else \'{n}\fi{}ski, A., \ifmmode~\dot{Z}\else
  \.{Z}\fi{}yczkowski, K., and S\l{}omczy\ifmmode~\acute{n}\else \'{n}\fi{}ski,
  W. (2003).
\newblock Quantum iterated function systems.
\newblock {\em Phys. Rev. E}, 68:046110.

\bibitem[Lu and Mukherjea, 1997]{lu1997}
Lu, G. and Mukherjea, A. (1997).
\newblock Invariant measures and {M}arkov chains with random transition
  probabilities.
\newblock {\em Probability and Mathematical Statistics-Wroclaw University},
  17:115--138.

\bibitem[Majumdar, 2009]{Majumdar2009}
Majumdar, M. (2009).
\newblock Markov processes generated by random iterates of monotone maps:
  Theory and applications.
\newblock In {\em Perspectives In Mathematical Sciences I: Probability and
  Statistics}, pages 203--223. World Scientific.

\bibitem[Majumdar and Mitra, 1994]{Majumdar1994}
Majumdar, M. and Mitra, T. (1994).
\newblock Robust ergodic chaos in discounted dynamic optimization models.
\newblock {\em Economic Theory}, 4(5):677--688.

\bibitem[Majumdar and Mitra, 2000]{Majumdar2000}
Majumdar, M. and Mitra, T. (2000).
\newblock Robust ergodic chaos in discounted dynamic optimization models.
\newblock In {\em Optimization and chaos}, pages 240--257. Springer.

\bibitem[Majumdar et~al., 1989]{Majumdar1989}
Majumdar, M., Mitra, T., and Nyarko, Y. (1989).
\newblock Dynamic optimization under uncertainty: non-convex feasible set.
\newblock In {\em Joan Robinson and modern economic theory}, pages 545--590.
  Springer.

\bibitem[Majumdar and Radner, 1992]{Majumdar1992}
Majumdar, M. and Radner, R. (1992).
\newblock Survival under production uncertainty.
\newblock In {\em Equilibrium and dynamics}, pages 179--200. Springer.

\bibitem[Malliavin, 1978]{malliavin1978stochastic}
Malliavin, P. (1978).
\newblock Stochastic calculus of variations and hypoelliptic operators.
\newblock In {\em Proc. Internat. Symposium on Stochastic Differential
  Equations, Kyoto Univ., Kyoto, 1976}. Wiley.

\bibitem[Mandelbrot and Mandelbrot, 1982]{Mandelbrot1982}
Mandelbrot, B.~B. and Mandelbrot, B.~B. (1982).
\newblock {\em The fractal geometry of nature}, volume~1.
\newblock WH freeman New York.

\bibitem[Marecek et~al., 2022]{Marecek2022}
Marecek, J., Roubalik, M., and Difonzo, F.~V. (2022).
\newblock Predictability and fairness in load aggregation with deadband.
\newblock {\em to appear}.

\bibitem[Marecek et~al., 2021]{Marecek2021}
Marecek, J., Roubalik, M., Ghosh, R., Shorten, R.~N., and Wirth, F. (2021).
\newblock Predictability and fairness in load aggregation and operations of
  virtual power plants.
\newblock {\em arXiv preprint arXiv:2110.03001}.

\bibitem[Markov, 1906]{Markov1906}
Markov, A.~A. (1906).
\newblock Extensions of the law of large numbers for dependent variables
  (russian).
\newblock {\em Izv. Fiz.-Mat. Obshch. Kazansk. Univ. 15}, 1:135--156.

\bibitem[Marta, 1997]{Marta1997}
Marta, T.~M. (1997).
\newblock Generic properties of iterated function systems with place dependent
  probabilities.
\newblock {\em Zeszyty Naukowe-Uniwersytetu Jagiellonskiego-All Series},
  1209:213--224.

\bibitem[Masoumeh, 2017]{Masoumeh2017}
Masoumeh, G. (2017).
\newblock {\em Iterated Function Systems of Interval Maps}.
\newblock PhD thesis, Universiteit van Amsterdam.

\bibitem[Mattila, 1995]{Mattila1995}
Mattila, P. (1995).
\newblock {\em Geometry of Sets and Measures in Euclidean Spaces: Fractals and
  Rectifiability}.
\newblock Cambridge Studies in Advanced Mathematics. Cambridge University
  Press.

\bibitem[Mauldin and Urba{\'n}ski, 1999]{Mauldin1999}
Mauldin, R. and Urba{\'n}ski, M. (1999).
\newblock Conformal iterated function systems with applications to the geometry
  of continued fractions.
\newblock {\em Transactions of the American Mathematical Society},
  351(12):4995--5025.

\bibitem[Mauldin and Urba{\'n}ski, 2000]{Mauldin2000}
Mauldin, R. and Urba{\'n}ski, M. (2000).
\newblock Parabolic iterated function systems.
\newblock {\em Ergodic Theory and Dynamical Systems}, 20(5):1423--1447.

\bibitem[Mauldin, 1995]{Mauldin1995}
Mauldin, R.~D. (1995).
\newblock {\em Infinite iterated function systems: theory and applications}.
\newblock Springer.

\bibitem[Mauldin et~al., 2001]{Mauldin2001}
Mauldin, R.~D., Przytycki, F., and Urba{\'n}ski, M. (2001).
\newblock Rigidity of conformal iterated function systems.
\newblock {\em Compositio Mathematica}, 129(3):273--299.

\bibitem[Mauldin and Urba{\'n}ski, 1996]{Mauldin1996}
Mauldin, R.~D. and Urba{\'n}ski, M. (1996).
\newblock Dimensions and measures in infinite iterated function systems.
\newblock {\em Proceedings of the London Mathematical Society}, 3(1):105--154.

\bibitem[May, 1976]{May1976}
May, R.~M. (1976).
\newblock Simple mathematical models with very complicated dynamics.
\newblock {\em Nature}, 261(5560):459--467.

\bibitem[Mendivil, 1998]{Mendivil1998(2)}
Mendivil, F. (1998).
\newblock A generalization of ifs with probabilities to infinitely many maps.
\newblock {\em The Rocky Mountain journal of mathematics}, pages 1043--1051.

\bibitem[Mendivil, 2015]{Mendivil2015}
Mendivil, F. (2015).
\newblock Time-dependent iteration of random functions.
\newblock {\em Chaos, Solitons \& Fractals}, 75:178--184.

\bibitem[Meyn and Tweedie, 1993]{Meyn1993}
Meyn, S.~P. and Tweedie, R.~L. (1993).
\newblock {\em {M}arkov chains and stochastic stability}.
\newblock Springer-Verlag, London, UK.

\bibitem[Mihail, 2010]{Mihail2010(2)}
Mihail, A. (2010).
\newblock A neccessary and sufficient condition for connectivity of the
  attractor of infinite iterated function systems.
\newblock {\em Rev. Roumaine Math. Pures Appl}, 55(2):147--157.

\bibitem[Mihail, 2012]{Mihail2012}
Mihail, A. (2012).
\newblock A topological version of iterated function systems.
\newblock {\em An. Stiint. Univ. Al. I. Cuza, Ia si,(SN), Matematica},
  58(1):105--120.

\bibitem[Mihail and Miculescu, 2009]{Mihail2009}
Mihail, A. and Miculescu, R. (2009).
\newblock The shift space for an infinite iterated function system.
\newblock {\em Math. Rep.(Bucur.)}, 11(61):1.

\bibitem[Mihail and Miculescu, 2010]{Mihail2010(1)}
Mihail, A. and Miculescu, R. (2010).
\newblock Generalized {IFSs} on noncompact spaces.
\newblock {\em Fixed Point Theory and Applications}, 2010(1):584215.

\bibitem[Mihailescu and Urba{\'n}ski, 2016]{Mihailescu2016a}
Mihailescu, E. and Urba{\'n}ski, M. (2016).
\newblock Overlap functions for measures in conformal iterated function
  systems.
\newblock {\em Journal of Statistical Physics}, 162(1):43--62.

\bibitem[Miller, 1994]{Miller1994}
Miller, W.~M. (1994).
\newblock Stability and approximation of invariant measures for a class of
  nonexpanding transformations.
\newblock {\em Nonlinear Analysis: Theory, Methods \& Applications},
  23(8):1013--1025.

\bibitem[Mirek, 2011]{Mirek2011}
Mirek, M. (2011).
\newblock Heavy tail phenomenon and convergence to stable laws for iterated
  {L}ipschitz maps.
\newblock {\em Probability Theory and Related Fields}, 151(3):705--734.

\bibitem[Mitra, 1998]{Mitra1998}
Mitra, K. (1998).
\newblock On capital accumulation paths in a neoclassical stochastic growth
  model.
\newblock {\em Economic Theory}, 11(2):457--464.

\bibitem[Mitra et~al., 2000]{Mitra2000}
Mitra, S., Murthy, C., and Kundu, M. (2000).
\newblock Partitioned iterative function system: A new tool for digital
  imaging.
\newblock {\em IETE Journal of Research}, 46(5):279--298.

\bibitem[Mostafaei and Kordnourie, 2011]{Mostafa2011}
Mostafaei, H. and Kordnourie, S. (2011).
\newblock Probability metrics and their applications.
\newblock {\em Applied Mathematical Sciences}, 5(4):181--192.

\bibitem[Mukherjea, 1992]{Mukherjea1992}
Mukherjea, A. (1992).
\newblock Recurrent random walks in nonnegative matrices: attractors of certain
  iterated function systems.
\newblock {\em Probability Theory and Related Fields}, 91(3):297--306.

\bibitem[Murray, 2010]{Murray2010}
Murray, R. (2010).
\newblock Ulam’s method for some non-uniformly expanding maps.
\newblock {\em Discrete. Contin. Dyn. Syst}, 26(3):1007--1018.

\bibitem[Narendra and Thathachar, 1974]{Narendra1974}
Narendra, K.~S. and Thathachar, M.~A. (1974).
\newblock Learning automata-a survey.
\newblock {\em IEEE Transactions on systems, man, and cybernetics},
  SMC-4:323--334.

\bibitem[Narendra and Thathachar, 2012]{Narendra2012}
Narendra, K.~S. and Thathachar, M.~A. (2012).
\newblock {\em Learning automata: an introduction}.
\newblock Courier corporation.

\bibitem[Natoli, 2012]{Natoli2012}
Natoli, C. (2012).
\newblock Fractals as fixed points of iterated function systems.
\newblock In {\em Fractals As Fixed Points Of Iterated Function Systems}.

\bibitem[Nia, 2015]{Nia2015}
Nia, M.~F. (2015).
\newblock Iterated function systems with the average shadowing property.
\newblock {\em arXiv preprint arXiv:1505.06547}.

\bibitem[Niclas~Carlsson, 2005]{Niclas2005(3)}
Niclas~Carlsson, G.~H. (2005).
\newblock Asymptotic properties of a {TCP} model with time-outs.
\newblock {\em Discrete \& Continuous Dynamical Systems - B}, 5(3):543--564.

\bibitem[Nikiel, 2007]{Nikiel2007}
Nikiel, S. (2007).
\newblock {\em Iterated Function Systems for Real-Time Image Synthesis}.
\newblock Springer-Verlag, Berlin, Heidelberg.

\bibitem[Norman, 1968]{Norman1968}
Norman, M.~F. (1968).
\newblock Some convergence theorems for stochastic learning models with
  distance diminishing operators.
\newblock {\em Journal of Mathematical Psychology}, 5(1):61--101.

\bibitem[Norman, 1972]{Norman1972}
Norman, M.~F. (1972).
\newblock {\em Markov processes and learning models}, volume~84.
\newblock Academic Press New York.

\bibitem[Nualart, 2006]{nualart2006malliavin}
Nualart, D. (2006).
\newblock {\em The Malliavin calculus and related topics}, volume 1995.
\newblock Springer.

\bibitem[Nummelin, 1984]{Nummelin1984}
Nummelin, E. (1984).
\newblock {\em General Irreducible {M}arkov Chains and Non-Negative Operators}.
\newblock Cambridge Tracts in Mathematics. Cambridge University Press.

\bibitem[{\"O}berg, 2005]{Oberg2005}
{\"O}berg, A. (2005).
\newblock Algorithms for approximation of invariant measures for {IFS}.
\newblock {\em manuscripta mathematica}, 116(1):31--55.

\bibitem[{\"O}berg, 2006]{Oberg2006}
{\"O}berg, A. (2006).
\newblock Approximation of invariant measures for random iterations.
\newblock {\em The Rocky Mountain Journal of Mathematics}, pages 273--301.

\bibitem[On{\'e}simo~Hern{\'a}ndez and Lasserre, 2003]{Lasserre2003}
On{\'e}simo~Hern{\'a}ndez, L. and Lasserre, J.~B. (2003).
\newblock {\em {M}arkov Chains and Invariant Probabilities}.
\newblock Birkhäuser, Basel.

\bibitem[Onicescu, 1935]{Mihoc1935}
Onicescu, O.~Mihoc, G. (1935).
\newblock Sur les chaines de variables statistiques.
\newblock {\em Revue de l'Institut International de Statistique}, 200:511--2.

\bibitem[Oono, 1989]{Oono1989}
Oono, Y. (1989).
\newblock Large deviation and statistical physics.
\newblock {\em Progress of Theoretical Physics Supplement}, 99:165--205.

\bibitem[Pandey et~al., 2022]{Pandey2022}
Pandey, M., Som, T., and Verma, S. (2022).
\newblock Set-valued $\alpha$-fractal functions.
\newblock {\em arXiv preprint arXiv:2207.02635}.

\bibitem[Parthasarathy, 1967]{Parthasarathy1967}
Parthasarathy, K.~R. (1967).
\newblock {\em Probability Measures on Metric Spaces}.
\newblock Academic Press.

\bibitem[Peigne, 1992]{Peigne1992}
Peigne, M. (1992).
\newblock Marches de {M}arkov sur le semi-groupe des contractions de.
\newblock In {\em Annales de IHP Probabilit {\ 'e} s et statistics}, volume
  28(1), pages 63--94.

\bibitem[Peign\'e, 1993]{Peigne1993}
Peign\'e, M. (1993).
\newblock Iterated function systems and spectral decomposition of the
  associated {M}arkov operator.
\newblock {\em Publications math\'ematiques et informatique de Rennes}, 2.

\bibitem[Peitgen et~al., 2006]{Peitgen2006}
Peitgen, H.-O., J{\"u}rgens, H., and Saupe, D. (2006).
\newblock {\em Chaos and fractals: new frontiers of science}.
\newblock Springer Science \& Business Media.

\bibitem[Pena, 2016]{Helena2016}
Pena, H. (2016).
\newblock {\em Affine Iterated Function Systems, invariant measures and their
  approximation}.
\newblock PhD thesis, Mathematisch-Naturwissenschaftlichen Fakultät,
  Ernst-Moritz-Arndt-Universität Greifswald.

\bibitem[Peruggia, 1993]{Peruggia1993}
Peruggia, M. (1993).
\newblock {\em Discrete iterated function systems}.
\newblock CRC Press.

\bibitem[Pesin, 2008]{Pesin2008}
Pesin, Y.~B. (2008).
\newblock {\em Dimension theory in dynamical systems: contemporary views and
  applications}.
\newblock University of Chicago Press.

\bibitem[Peters, 1994]{Peters1994}
Peters, E.~E. (1994).
\newblock {\em Fractal market analysis: applying chaos theory to investment and
  economics}, volume~24.
\newblock John Wiley \& Sons.

\bibitem[Peters, 1996]{Peters1996}
Peters, E.~E. (1996).
\newblock {\em Chaos and order in the capital markets: a new view of cycles,
  prices, and market volatility}.
\newblock John Wiley \& Sons.

\bibitem[Pollicott, 2001]{Pollicott2001}
Pollicott, M. (2001).
\newblock Contraction in mean and transfer operators.
\newblock {\em Dynamical Systems: An International Journal}, 16(1):97--106.

\bibitem[Pollicott and Jenkinson, 2000]{Pollicott2000}
Pollicott, M. and Jenkinson, O. (2000).
\newblock Computing invariant densities and metric entropy.
\newblock {\em Communications in Mathematical Physics}, 211(3):687--703.

\bibitem[P{\"o}tzsche et~al., 2003]{potzsche2003spectral}
P{\"o}tzsche, C., Siegmund, S., and Wirth, F. (2003).
\newblock A spectral characterization of exponential stability for linear
  time-invariant systems on time scales.
\newblock {\em Discrete \& Continuous Dynamical Systems}, 9(5):1223.

\bibitem[Propp and Wilson, 1996]{Propp1996}
Propp, J.~G. and Wilson, D.~B. (1996).
\newblock Exact sampling with coupled {M}arkov chains and applications to
  statistical mechanics.
\newblock {\em Random Structures \& Algorithms}, 9(1-2):223--252.

\bibitem[Rachev, 1991]{Rachev1991}
Rachev, S.~T. (1991).
\newblock {\em Probability metrics and the stability of stochastic models},
  volume 269.
\newblock John Wiley \& Son Ltd.

\bibitem[Rempe-Gillen and Urba{\'n}ski, 2016]{Rempe2016}
Rempe-Gillen, L. and Urba{\'n}ski, M. (2016).
\newblock Non-autonomous conformal iterated function systems and moran-set
  constructions.
\newblock {\em Transactions of the American Mathematical Society},
  368(3):1979--2017.

\bibitem[Rescorla, 1972]{Rescorla1972}
Rescorla, R.~A. (1972).
\newblock A theory of pavlovian conditioning: Variations in the effectiveness
  of reinforcement and nonreinforcement.
\newblock {\em Current research and theory}, pages 64--99.

\bibitem[Revuz, 1975]{Revuz1975}
Revuz, D. (1975).
\newblock {\em Markov Chains}.
\newblock Elsevier Science Publishing Co Inc.,U.S.

\bibitem[Riesz, 1909]{Riesz1909}
Riesz, F. (1909).
\newblock Sur les opérations fonctionnelles linéaires (on linear functional
  operations).
\newblock {\em Comptes rendus}, 149:974–977.

\bibitem[Roberts and Rosenthal, 2004]{Roberts2004}
Roberts, G.~O. and Rosenthal, J.~S. (2004).
\newblock General state space {M}arkov chains and {MCMC} algorithms.
\newblock {\em Probab. Surveys}, 1:20--71.

\bibitem[Rom{\'a}n-Rold{\'a}n et~al., 1994]{Roman1994}
Rom{\'a}n-Rold{\'a}n, R., Bernaola-Galv{\'a}n, P., and Oliver, J.~L. (1994).
\newblock Entropic feature for sequence pattern through iterated function
  systems.
\newblock {\em Pattern Recognition Letters}, 15(6):567--573.

\bibitem[Roy and Urba{\'n}ski, 2005]{Roy2005}
Roy, M. and Urba{\'n}ski, M. (2005).
\newblock Regularity properties of hausdorff dimension in infinite conformal
  iterated function systems.
\newblock {\em Ergodic Theory and Dynamical Systems}, 25(6):1961--1983.

\bibitem[Royden and Fitzpatrick, 1988]{Royden1988}
Royden, H.~L. and Fitzpatrick, P. (1988).
\newblock {\em Real analysis}, volume~32.
\newblock Macmillan New York.

\bibitem[Rudnicki, 2000]{Rudnicki2000}
Rudnicki, R. (2000).
\newblock {M}arkov operators: applications to diffusion processes and
  population dynamics.
\newblock {\em Applicationes Mathematicae}, 27:67--79.

\bibitem[Rudnicki et~al., 2002]{Rudnicki2002}
Rudnicki, R., Pichor, K., and Tyran-Kamińska, M. (2002).
\newblock {M}arkov semigroups and their applications.
\newblock {\em Lecture Notes in Physics}, 597:215--238.

\bibitem[Rudnicki and Tyran-Kamińska, 2015]{Rudnicki2015}
Rudnicki, R. and Tyran-Kamińska, M. (2015).
\newblock {\em Piecewise Deterministic {M}arkov Processes in Biological
  Models}, volume 113.
\newblock Springer International Publishing.

\bibitem[Russell~Luke et~al., 2018]{Russell2018}
Russell~Luke, D., Thao, N.~H., and Tam, M.~K. (2018).
\newblock Quantitative convergence analysis of iterated expansive, set-valued
  mappings.
\newblock {\em Mathematics of Operations Research}, 43(4):1143--1176.

\bibitem[Samuel and Tetenov, 2017]{Samuel2017}
Samuel, M. and Tetenov, A.~V. (2017).
\newblock On attractors of iterated function systems in uniform spaces.
\newblock {\em Sib. Elektron. Mat. Izv}, 14(0):151--155.

\bibitem[Santos and Walkden, 2013]{Santos2013}
Santos, S. and Walkden, C. (2013).
\newblock Distributional and local limit laws for a class of iterated maps that
  contract on average.
\newblock {\em Stochastics and Dynamics}, 13(2).

\bibitem[Secelean, 2013]{Secelean2013}
Secelean, N.~A. (2013).
\newblock {\em Countable iterated function systems}.
\newblock LAP Lambert Academic Publishing.

\bibitem[Seuret and Wang, 2015]{Seuret2015}
Seuret, S. and Wang, B.-W. (2015).
\newblock Quantitative recurrence properties in conformal iterated function
  systems.
\newblock {\em Advances in Mathematics}, 280:472--505.

\bibitem[Shaun, 2015]{Shaun2015}
Shaun, A.~M. (2015).
\newblock {\em From Sojourn Times and Boundary Crossings to Iterated Random
  Functions}.
\newblock PhD thesis, School of Mathematics and Statistics, The University of
  Melbourne.

\bibitem[Shi and Chen, 2004]{Shi2004}
Shi, Y. and Chen, G. (2004).
\newblock Chaos of discrete dynamical systems in complete metric spaces.
\newblock {\em Chaos, Solitons \& Fractals}, 22(3):555--571.

\bibitem[Shorten et~al., 2007]{Shorten2007}
Shorten, R., King, C., Wirth, F., and Leith, D. (2007).
\newblock Modelling {TCP} congestion control dynamics in drop-tail
  environments.
\newblock {\em Automatica}, 43(3):441--449.

\bibitem[{Shorten} et~al., 2007]{Bob2007}
{Shorten}, R., {King}, C., {Wirth}, F., and {Leith}, D. (2007).
\newblock On the ergodicity of {AIMD} networks.
\newblock In {\em 2007 {A}merican Control Conference}, pages 3283--3287.

\bibitem[Shorten et~al., 2005]{Shorten2005}
Shorten, R., Leith, D., Foy, J., and Kilduff, R. (2005).
\newblock Analysis and design of {AIMD} congestion control algorithms in
  communication networks.
\newblock {\em Automatica}, 41(4):725--730.

\bibitem[Shorten et~al., 2007]{shorten2007stability}
Shorten, R., Wirth, F., Mason, O., Wulff, K., and King, C. (2007).
\newblock Stability criteria for switched and hybrid systems.
\newblock {\em SIAM review}, 49(4):545--592.

\bibitem[Shwartz and Weiss, 1995]{Shwartz1995}
Shwartz, A. and Weiss, A. (1995).
\newblock {\em Large deviations for performance analysis: queues, communication
  and computing}, volume~5.
\newblock CRC Press.

\bibitem[Silva, 2008]{Silva2008}
Silva, C.~E. (2008).
\newblock {\em Invitation to ergodic theory}, volume~42.
\newblock American Mathematical Soc.

\bibitem[Silvestrov and Stenflo, 1998]{Silvestrov1998}
Silvestrov, D.~S. and Stenflo, {\"O}. (1998).
\newblock Ergodic theorems for iterated function systems controlled by
  regenerative sequences.
\newblock {\em Journal of Theoretical Probability}, 11(3):589--608.

\bibitem[Skorokhod, 1987]{Skorokhod1987}
Skorokhod, A. (1987).
\newblock Topologically recurrent {M}arkov chains: {E}rgodic properties.
\newblock {\em Theory of Probability \& Its Applications}, 31(4):563--571.

\bibitem[Skorokhod, 2007]{Skorokhod2007}
Skorokhod, A. (2007).
\newblock Homogeneous {M}arkov chains in compact spaces.
\newblock {\em Theory of Stochastic Processes}, 13(29)(3):80--95.

\bibitem[Spaulding, 2022]{Spaulding2022}
Spaulding, S.~S. (2022).
\newblock Dimension theory of conformal iterated function systems.
\newblock {\em University of Connecticut}.

\bibitem[Sprott, 1994]{Sprott1994}
Sprott, J.~C. (1994).
\newblock Automatic generation of iterated function systems.
\newblock {\em Computers \& graphics}, 18(3):417--425.

\bibitem[Stark, 1991]{Stark1991}
Stark, J. (1991).
\newblock Iterated function systems as neural networks.
\newblock {\em Neural Networks}, 4(5):679--690.

\bibitem[Steinsaltz, 1999]{Steinsaltz1999}
Steinsaltz, D. (1999).
\newblock Locally contractive iterated function systems.
\newblock {\em Ann. Probab.}, 27(4):1952--1979.

\bibitem[Steinsaltz, 2001]{Steinsaltz2001}
Steinsaltz, D. (2001).
\newblock Random logistic maps and {L}yapunov exponents.
\newblock {\em Indagationes Mathematicae}, 12(4):557--584.

\bibitem[Stenflo, 1998a]{Stenflo1999}
Stenflo, {\"O}. (1998a).
\newblock {\em Ergodic theorems for iterated function systems controlled by
  stochastic sequences}.
\newblock PhD thesis, Department of Mathematics, Umea University.
\newblock Doctoral Thesis No. 14.

\bibitem[Stenflo, 1998b]{Stenflo1998}
Stenflo, {\"O}. (1998b).
\newblock {\em Ergodic theorems for time-dependent random iteration of
  functions}.
\newblock University of Umea, Department of Mathematics.

\bibitem[Stenflo, 2001a]{Stenflo2001(3)}
Stenflo, {\"O}. (2001a).
\newblock Ergodic theorems for {M}arkov chains represented by iterated function
  systems.
\newblock In {\em Bulletin of the Polish Academy of Sciences Mathematics
  49(1)}.

\bibitem[Stenflo, 2001b]{Stenflo2001(2)}
Stenflo, {\"O}. (2001b).
\newblock A note on a theorem of {K}arlin.
\newblock {\em Statistics \& Probability Letters}, 54(2):183--187.

\bibitem[Stenflo, 2002a]{Stenflo2003}
Stenflo, {\"O}. (2002a).
\newblock Uniqueness in {g}-measures.
\newblock {\em Nonlinearity}, 16(2):403.

\bibitem[Stenflo, 2002b]{Stenflo2002}
Stenflo, {\"O}. (2002b).
\newblock Uniqueness of invariant measures for place-dependent random
  iterations of functions.
\newblock In Barnsley, M.~F., Saupe, D., and Vrscay, E.~R., editors, {\em
  Fractals in Multimedia}, pages 13--32, New York, NY. Springer New York.

\bibitem[Stenflo, 2012a]{Stenflo2012}
Stenflo, {\"O}. (2012a).
\newblock Iterated function systems with a given continuous stationary
  distribution.
\newblock {\em Fractals}, 20.

\bibitem[Stenflo, 2012b]{Stenflo2012(s)}
Stenflo, {\"O}. (2012b).
\newblock A survey of average contractive iterated function systems.
\newblock {\em Journal of Difference Equations and Applications},
  18(8):1355--1380.

\bibitem[Stieltjes, 1895]{Stieltjes1895}
Stieltjes, T. (1895).
\newblock Research on continued fractions.
\newblock In {\em Annales de la Facult{\'e}des sciences de Toulouse:
  Math{\'e}matiques}, volume 9(1), pages A5--A47.

\bibitem[Stokey et~al., 1989]{Stokey1989}
Stokey, N.~L., Lucas, R., and Prescott, E. (1989).
\newblock Recursive methods in dynamic economics.
\newblock {\em Cambridge, MA: Harvard University}.

\bibitem[Strassen, 1964]{Strassen1964}
Strassen, V. (1964).
\newblock An invariance principle for the law of the iterated logarithm.
\newblock {\em Zeitschrift f{\"u}r Wahrscheinlichkeitstheorie und Verwandte
  Gebiete}, 3(3):211--226.

\bibitem[Strassen et~al., 1967]{Strassen1967}
Strassen, V. et~al. (1967).
\newblock Almost sure behavior of sums of independent random variables and
  martingales.
\newblock In {\em Proceedings of the Fifth Berkeley Symposium on Mathematical
  Statistics and Probability, Volume 2: Contributions to Probability Theory,
  Part 1}. The Regents of the University of California.

\bibitem[Strichartz et~al., 1995]{Robert1995}
Strichartz, R.~S., Taylor, A., and Zhang, T. (1995).
\newblock Densities of self-similar measures on the line.
\newblock {\em Experimental Mathematics}, 4(2):101--128.

\bibitem[Sullivan, 1982]{Sullivan1982}
Sullivan, D. (1982).
\newblock Discrete conformal groups and measurable dynamics.
\newblock {\em Bulletin of the American Mathematical Society}, 6(1):57--73.

\bibitem[Svetlozar and Rüschendorf, 1998]{Mass1998}
Svetlozar, T.~R. and Rüschendorf, L. (1998).
\newblock {\em Mass Transportation Problems-Vol I and Vol II}.
\newblock Probability and its Applications. Springer, New York, NY.

\bibitem[Swishchuk, 1997]{Anatoly1997}
Swishchuk, A. (1997).
\newblock {\em Random Evolutions and their applications}, volume 408.
\newblock Springer Science \& Business Media.

\bibitem[Swishchuk and Islam, 2011]{Anatoly2011}
Swishchuk, A. and Islam, M.~S. (2011).
\newblock The geometric {M}arkov renewal processes with application to finance.
\newblock {\em Stochastic analysis and applications}, 29(4):684--705.

\bibitem[Swishchuk and Islam, 2013]{Anatoly2013}
Swishchuk, A. and Islam, S. (2013).
\newblock {\em Random dynamical systems in finance}.
\newblock CRC Press.

\bibitem[Szarek, 1997]{Szarek1997}
Szarek, T. (1997).
\newblock Iterated function systems depending on previous transformation.
\newblock {\em Acta Mathematica}.

\bibitem[Szarek, 1999]{Szarek1999}
Szarek, T. (1999).
\newblock Generic properties of continuous iterated function systems.
\newblock {\em Bulletin of the Polish Academy of Sciences. Mathematics},
  47(1):77--89.

\bibitem[Szarek, 2000a]{Szarek2000(1)}
Szarek, T. (2000a).
\newblock Invariant measures for iterated function systems.
\newblock In {\em Annales Polonici Mathematici 75(1)}.

\bibitem[Szarek, 2000b]{Szarek2000(2)}
Szarek, T. (2000b).
\newblock The stability of {M}arkov operators on {P}olish spaces.
\newblock {\em Studia Mathematica}, 143.

\bibitem[Szarek, 2003a]{Szarek2003(1)}
Szarek, T. (2003a).
\newblock Invariant measures for {M}arkov operators with application to
  function systems.
\newblock {\em Studia Mathematica}, 154:207--222.

\bibitem[Szarek, 2003b]{Szarek2003(2)}
Szarek, T. (2003b).
\newblock Invariant measures for non-expansive {M}arkov operators on {P}olish
  spaces.
\newblock {\em Dissertationes Mathematicae}, 415:1--62.

\bibitem[Słomczyński, 1997]{Wojciech1997}
Słomczyński, W. (1997).
\newblock From quantum entropy to iterated function systems.
\newblock {\em Chaos, Solitons \& Fractals}, 8(11):1861 -- 1864.
\newblock Relativity, Locality and Random Fractals in Quantum Theory.

\bibitem[Tino and Dorffner, 1998]{Tino1998}
Tino, P. and Dorffner, G. (1998).
\newblock Recurrent neural networks with iterated function systems dynamics.
\newblock {\em International ICSC/IFAC Symposium on Neural Computation,
  accepted}.

\bibitem[Tong, 1990]{Tong90}
Tong, H. (1990).
\newblock {\em Non-linear time series: a dynamical system approach}.
\newblock University Press, Oxford.

\bibitem[Torre et~al., 2007]{Torre2007}
Torre, D.~L., Mendivil, F., and Vrscay, E.~R. (2007).
\newblock Iterated function systems on multifunctions.
\newblock In {\em Math Everywhere}, pages 125--138. Springer.

\bibitem[Touchette, 2009]{Touchette2009}
Touchette, H. (2009).
\newblock The large deviation approach to statistical mechanics.
\newblock {\em Physics Reports}, 478(1-3):1--69.

\bibitem[Traple, 1996]{Traple1996}
Traple, J. (1996).
\newblock Markov semigroups generated by {P}oisson driven differential
  equations.
\newblock {\em Bulletin of the Polish Academy of Sciences-Mathematics},
  44(2):161--182.

\bibitem[Tricot and Riedi, 1999]{Tricot1999}
Tricot, C. and Riedi, R. (1999).
\newblock Attracteurs, orbites et ergodicit\'e.
\newblock {\em Annales math\'ematiques Blaise Pascal}, 6(1):55--72.

\bibitem[Tulcea, 1959]{Ionescu1959}
Tulcea, C.~I. (1959).
\newblock On a class of operators occurring in the theory of chains of infinite
  order.
\newblock {\em Canadian Journal of Mathematics}, 11:112--121.

\bibitem[Tyrcha, 1988]{Tyrcha1988}
Tyrcha, J. (1988).
\newblock Asymptotic stability in a generalized probabilistic/deterministic
  model of the cell cycle.
\newblock {\em Journal of mathematical biology}, 26(4):465--475.

\bibitem[Tyson and Hannsgen, 1986]{Tyson1986}
Tyson, J.~J. and Hannsgen, K.~B. (1986).
\newblock Cell growth and division: a deterministic/probabilistic model of the
  cell cycle.
\newblock {\em Journal of Mathematical Biology}, 23(2):231--246.

\bibitem[Ulam, 1960]{Ulam1960}
Ulam, S.~M. (1960).
\newblock A collection of mathematical problems.
\newblock {\em Bull. Amer. Math. Soc.}, 66:361--363.

\bibitem[Urba{\'n}ski, 2002]{Urbanski2002}
Urba{\'n}ski, M. (2002).
\newblock Diophantine analysis of conformal iterated function systems.
\newblock {\em Monatshefte f{\"u}r Mathematik}, 137(4):325--340.

\bibitem[Usachev, 1972]{Usachev1972}
Usachev, Y.~S. (1972).
\newblock Asymptotic properties of random processes that arise in learning
  models.
\newblock {\em Journal of Cybernetics}.

\bibitem[Van~der Hoek and Elliott, 2006]{Van2006}
Van~der Hoek, J. and Elliott, R.~J. (2006).
\newblock {\em Binomial models in Finance}.
\newblock Springer Science \& Business Media.

\bibitem[van Wyk and Ding, 2002]{Van2002}
van Wyk, M.~A. and Ding, J. (2002).
\newblock Stochastic analysis of electrical circuits.
\newblock In {\em Chaos in Circuits and Systems}, pages 215--236. World
  Scientific.

\bibitem[Varadhan, 1984]{Varadhan1984}
Varadhan, S. R. I.~S. (1984).
\newblock {\em Large deviations and applications}.
\newblock SIAM.

\bibitem[Vershik, 2006]{Vershik2006}
Vershik, A.~M. (2006).
\newblock Kantorovich metric: Initial history and little-known applications.
\newblock {\em Journal of Mathematical Sciences}, 133(4):1410--1417.

\bibitem[Viard, 1994]{Viard1994}
Viard, R. (1994).
\newblock Calcul de malliavin et développement asymptotique de la densité
  d'une probabilité invariante d'une chaîne de markov.
\newblock {\em Publications mathématiques et informatique de Rennes}, 2:1--48.

\bibitem[Villani, 2009]{Villani2009}
Villani, C. (2009).
\newblock {\em Optimal {T}ransport: {O}ld and {N}ew}, volume 338 of {\em
  Grundlehren der mathematischen Wissenschaften}.
\newblock Springer-Verlag, Berlin Heidelberg.

\bibitem[Vince, 2013]{Vince2013}
Vince, A. (2013).
\newblock M{\"o}bius iterated function systems.
\newblock {\em Transactions of the {A}merican Mathematical Society},
  365(1):491--509.

\bibitem[Vines, 1993]{Vines1993}
Vines, G. (1993).
\newblock {\em Signal modeling with iterated function systems}.
\newblock PhD thesis, School of Electrical Engineering, Georgia Institute of
  Technology.

\bibitem[Vrscay, 1991]{Vrscay1991(1)}
Vrscay, E.~R. (1991).
\newblock {\em Iterated function systems: theory, applications and the inverse
  problem}, pages 405--468.
\newblock Springer Netherlands, Dordrecht.

\bibitem[Vrscay and Roehrig, 1989]{Vrscay1989}
Vrscay, E.~R. and Roehrig, C.~J. (1989).
\newblock Iterated function systems and the inverse problem of fractal
  construction using moments.
\newblock In {\em Computers and Mathematics}, pages 250--259. Springer.

\bibitem[Vrscay and Weil, 1991]{Vrscay1991(2)}
Vrscay, E.~R. and Weil, D. (1991).
\newblock Missing moment and perturbative methods for polynomial iterated
  function systems.
\newblock {\em Physica D: Nonlinear Phenomena}, 50(3):478--492.

\bibitem[Walkden, 2007]{Walkden2007}
Walkden, C.~P. (2007).
\newblock Invariance principles for interated maps that contract on average.
\newblock {\em Transactions of the {A}merican Mathematical Society},
  359(3):1081--1097.

\bibitem[Weaver, 2018]{Weaver2018}
Weaver, N. (2018).
\newblock {\em {L}ipschitz algebras}.
\newblock World Scientific.

\bibitem[Wei et~al., 2019]{Karl2019}
Wei, J., Nekouei, E., Wu, J., Cvetkovic, V., and Johansson, K.~H. (2019).
\newblock Steady-state analysis of a human-social behavior model: a
  neural-cognition perspective.
\newblock In {\em 2019 American Control Conference (ACC)}, pages 199--204.

\bibitem[Werner, 2004]{Ivan2004}
Werner, I. (2004).
\newblock Ergodic theorem for contractive {M}arkov systems.
\newblock {\em Nonlinearity}, 17(6):2303.

\bibitem[Werner, 2005a]{Ivan2005(4)}
Werner, I. (2005a).
\newblock Contractive {M}arkov system with constant probabilities.
\newblock {\em Journal of Theoretical Probability}, 18(2):469--479.

\bibitem[Werner, 2005b]{Ivan2005(5)}
Werner, I. (2005b).
\newblock Contractive {M}arkov systems {II}.
\newblock {\em arXiv preprint math/0503633}.

\bibitem[Werner, 2005c]{Ivan2005(6)}
Werner, I. (2005c).
\newblock Contractive {M}arkov systems, {PhD} thesis.
\newblock {\em Journal of the London Mathematical Society}, 71(1):236--258.

\bibitem[Werner, 2005d]{Ivan2005(3)}
Werner, I. (2005d).
\newblock The generalized {M}arkov measure as an equilibrium state.
\newblock {\em Nonlinearity}, 18(5):2261.

\bibitem[Werner, 2005e]{Ivan2005(1)}
Werner, I. (2005e).
\newblock A necessary condition for the uniqueness of the stationary state of a
  {M}arkov system.
\newblock {\em arXiv preprint math/0508054}.

\bibitem[Werner, 2005f]{Ivan2005(2)}
Werner, I. (2005f).
\newblock On coding with feller contractive {M}arkov systems.
\newblock {\em arXiv}, pages math--0506476.

\bibitem[Werner, 2006]{Ivan2006}
Werner, I. (2006).
\newblock Coding map for a contractive {M}arkov system.
\newblock {\em Mathematical Proceedings of the Cambridge Philosophical
  Society}, 140(2):333–347.

\bibitem[Wicks, 2006]{Wicks2006}
Wicks, K.~R. (2006).
\newblock {\em Fractals and hyperspaces}.
\newblock Springer.

\bibitem[Widder et~al., 1945]{Widder1945}
Widder, D. et~al. (1945).
\newblock The problem of moments.
\newblock {\em Bulletin of the American Mathematical Society}, 51(11):860--863.

\bibitem[Williams, 1971]{Williams1971}
Williams, R. (1971).
\newblock Composition of contractions.
\newblock {\em Boletim da Sociedade Brasileira de Matem{\'a}tica}, 2(2):55--59.

\bibitem[Wirth et~al., 2006a]{Wirth2006}
Wirth, F., Shorten, R., and Akar, M. (2006a).
\newblock State dependent {AIMD} algorithms and consensus problems.
\newblock In {\em PAMM: Proceedings in Applied Mathematics and Mechanics},
  volume~6, pages 855--856. Wiley Online Library.

\bibitem[Wirth et~al., 2006b]{Fabian2006}
Wirth, F., Stanojevi{\'c}, R., Shorten, R., and Leith, D. (2006b).
\newblock Stochastic equilibria of {AIMD} communication networks.
\newblock {\em SIAM Journal on Matrix Analysis and Applications},
  28(3):703--723.

\bibitem[Wirth et~al., 2019]{Wirth2019}
Wirth, F., Stuedli, S., Yu, J.~Y., Corless, M., and Shorten, R. (2019).
\newblock Nonhomogeneous place-dependent {M}arkov chains, unsynchronised
  {AIMD}, and optimisation.
\newblock {\em Journal of the ACM (JACM)}, 66(4):1--37.

\bibitem[Wojew{\'o}dka, 2013]{Wojewodka2013}
Wojew{\'o}dka, H. (2013).
\newblock Exponential rate of convergence for some {M}arkov operators.
\newblock {\em Statistics \& Probability Letters}, 83(10):2337--2347.

\bibitem[Wojewódka, 2016]{wojewodka2016phd}
Wojewódka, H. (2016).
\newblock {\em PhD dissertation (in Polish): Ergodyczne własności pewnych
  stochastycznych układów dynamicznych (Ergodic properties of some stochastic
  dynamical systems)}.
\newblock PhD thesis, Faculty of Mathematics, Physics and Informatics,
  University of Gdansk, Poland.

\bibitem[Worm, 2010]{Worm2010}
Worm, D. (2010).
\newblock {\em Semigroups on spaces of measures}.
\newblock PhD thesis, Thomas Stieltjes Institute for Mathematics, Universiteit
  Leiden.

\bibitem[Wu and Shao, 2004]{WuShao2004}
Wu, W.~B. and Shao, X. (2004).
\newblock Limit theorems for iterated random functions.
\newblock {\em Journal of Applied Probability}, 41(2):425--436.

\bibitem[Wu and Woodroofe, 2000]{Wu2000}
Wu, W.~B. and Woodroofe, M. (2000).
\newblock A central limit theorem for iterated random functions.
\newblock {\em Journal of Applied Probability}, 37(3):748--755.

\bibitem[Yahav, 1975]{Yahav1975}
Yahav, J.~A. (1975).
\newblock On a fixed point theorem and its stochastic equivalent.
\newblock {\em Journal of Applied Probability}, pages 605--611.

\bibitem[Yosida, 1995]{Yosida1995}
Yosida, K. (1995).
\newblock {\em Ergodic Theory and Diffusion Theory}, pages 379--418.
\newblock Springer Berlin Heidelberg, Berlin, Heidelberg.

\bibitem[{Zaharopol}, 2005]{Zaharopol2005}
{Zaharopol}, R. (2005).
\newblock {\em Invariant probabilities of {M}arkov-{F}eller operators and their
  supports.}
\newblock Basel: Birkh\"auser.

\bibitem[Zaharopol, 2009]{Zaharopol2009}
Zaharopol, R. (2009).
\newblock Transition probabilities, transition functions and an ergodic
  decomposition.
\newblock {\em Bulletin of the Transilvania University of Brașov. Series III,
  Mathematics, informatics, physics}, 2:149--149.

\bibitem[Zappavigna et~al., 2010]{zappavigna2010dwell}
Zappavigna, A., Colaneri, P., Geromel, J.~C., and Shorten, R. (2010).
\newblock Dwell time analysis for continuous-time switched linear positive
  systems.
\newblock In {\em Proceedings of the 2010 {A}merican control conference}, pages
  6256--6261. IEEE.

\bibitem[Ślęczka, 2011]{Maciej2011}
Ślęczka, M. (2011).
\newblock The rate of convergence for iterated function systems.
\newblock {\em Studia Mathematica}, 205.

\end{thebibliography}
\clearpage
\appendix
\section{Some Applications of IFS}\label{sec:IFS-appli}

\subsection{Physics, Mathematics and Engineering}\label{subsec:phy-math-eng}
The book of \cite{Tong90} contains a wide range of practical examples of the use of IFS in Statistics.
See \cite{Stenflo2002} for a survey of place-dependent IFS and connections to the \emph{Ruelle-Perron-Frobenius theorem of statistical mechanics}, \cite{Furstenberg1960, Guivarc1985, Mukherjea1992, Alsmeyer2013, Kaijser1978} for connections to the \emph{theory of random walk in non-negative matrices, theory of random matrices and the products of random matrices}, \cite{Iosifescu1990} for connections to the \emph{theory of continued fractions}, \cite{Propp1996, Athreya2003(3), Roberts2004} for a connection with \emph{Markov chain Monte-Carlo algorithms}, \cite{Falconer1985, Falconer2004, Barnsley1993} for more on connections to the \emph{theory to fractals}, \cite{Mandelbrot1982, Peters1994, Kifer1995, Peters1996, Bratteli1999, Jorgensen2006, Fisher2012, Gleick2011, Peitgen2006, Loskot2015} for connections to \emph{representation theory}, \cite{Davison2015,Davison2018} in \emph{operator theory}, \cite{Dutkay2007(1),Dutkay2007(2),Mirek2011,Jorgensen2011} in \emph{harmonic analysis}, \cite{Berger1992} for the \emph{theory of wavelets}, Keller \cite{Keller1998} and \cite{Pesin2008} for connections to the \emph{theory of deterministic dynamical systems}. For a number research articles on IFS on graphs see \cite{Ivan2004, Ivan2005(1),Ivan2005(2), Ivan2005(3), Ivan2005(4),Ivan2005(5), Ivan2005(6), Ivan2006}, on stochastic fixed point theory see \cite{Hermer2019, Hermer2020}. Articles which treat some particular type of IFS to mention a few such as: \emph{random walk on $\mathbb R^n$ generated by affine maps} by \cite{Berger1988}, \emph{random walk on $({\mathbb R}^{+})^n$ with elastic collisions on the axis}  \cite{Leguesdron1989}, Peigne \cite{Peigne1992}, \cite{Mauldin2000} discusses \emph{parabolic IFS} and \cite{Vince2013} explored several interesting results on \emph{Mobius IFS}, Fayolles, Malyshev, and Iasnogorodski \cite{Fayolle1999}, concerning quantum information theory \cite{Naschie1994, Wojciech1997, Karol2003, Arkadiusz2004, Baraviera2009, Jurgens2020}, in mathematical finance \cite{Anatoly2011, Anatoly2013}, to other branches of physics \cite{Decrouez2009, Fatehinia2018, Davison2018}. In Control of switched systems \cite{Branicky1998},  signal modeling  \cite{Vines1993, Freeland1991}, improve the gain and efficiency of the antenna \cite{Ezhumalai2021}, to model discrete sequence \cite{David1992}, in power domain and domain theory \cite{Edalata1996, Edalat1995}, electrical circuits \cite{Chen2002}, and in TCP \cite{Niclas2005(3)}.  

\subsection{Computer Science and Computer Graphics}\label{subsec:cs}
IFS was used to investigate stochastic learning automata, see, e.g., \cite{Narendra1974, Narendra2012, Lakshmivarahan2012, Norman1968, Norman1972, Karlin1953}. IFS on the space of input weights was used to develop an associative reinforcement learning theory in neural networks \cite{Bressloff1991, Bressloff1992(1), Bressloff1992(2)}. A recurrent neural network is an artificial neural network in which node connections generate a temporally directed graph \cite{Tino1998}. A recurrent neural network is an artificial neural network whose node connections produce a temporal graph. In \cite{Kolen1994}, IFS is used to characterize the behaviour of recurrent networks in symbolic processing. 

\cite{Stark1991} analyzed various methods for producing images using IFS and demonstrated how they might be described as neural networks with one neuron per image pixel. Such networks can quickly produce complex graphics for real-time animation. IFS application is known in pattern recognition also, see, e.g. \cite{Roman1994}. \cite{Sophie2011} devised a low-complexity, recursive approach for linear deconvolution systems using binary input using IFS.

IFS encodes and creates fractals;  Barnsley showed how to compress images.
Many articles address how IFS with affine maps may be used to draw and compress images, see, e.g., \cite{Hutchinson1981, Dubins1966, Diaconis1986, Barnsley1993, Barnsley1990, Kouzani2008, Peruggia1993, Hepting1991(1), Hepting1991(2), Berger1989}. Many works in the literature on fractals have developed in the past few decades. It has a very deep-rooted connection with the IFS see \cite{Williams1971, Abraham2000}  historical remark on fractals, and for other development and applications, see \cite{Crilly1991, Barnsley1989, Barnsley2008, Natoli2012, Sprott1994, Jones2001, Ankit2014, Mitra2000, Alexander2012}, on computer graphics \cite{Barnsley1993, Jacquin1988, Davoine2010, Nikiel2007, Kunze2012, Chand2007}.

\subsection{Biological Science}\label{subsec:biosc}
Andrzej Lasota and his group have made seminal contributions in a series of articles \cite{Lasota1973, Lasota1984, Tyrcha1988, Lasota1992, Lasota1994(1), Lasota1995, Lasota1998, Lasota1999} in modelling cell division processes using techniques from IFS and Markov operators taking values in Polish state space.
Using the IFS model, they studied general intracellular biochemistry. Each cell has $d$ molecules whose dynamic evolution governs its life history. They showed that cell division might create functionally and physically stable cells despite unstable molecular dynamics. Fast cell multiplication stabilizes cellular populations according to their stability. For more on the application of IFS and, in general, random dynamical systems in the field of Biological science, see \cite{Tyson1986, Doebeli1995, Traple1996, Lankhorst1996, Anatoly1997, Anatoliy2003, Joseph2006, Almeida2009, Almeida2014, Alfredo2017, Gaspard2017, Anatoliy2003, Ashlock2003, Hoskins1995, Lasota1999, Wojewodka2013, Rudnicki2000, Rudnicki2002, Wojewodka2013, Hille2016, Rudnicki2015, Czapla2019, Czapla2020}.

\end{document}